\documentclass[11pt]{article}
\usepackage{amssymb, amsmath,amsthm}
\usepackage{graphicx}
\usepackage{enumitem}
\usepackage[ruled,vlined]{algorithm2e}
\usepackage{bm}
\usepackage{calc,array}
\usepackage[margin=1in]{geometry}
\usepackage[numbers]{natbib}
\usepackage{booktabs}
\usepackage{dutchcal}
\usepackage{mathtools}
\usepackage{hyperref}

\newcommand{\Omegah}{\Omega_{h}}
\newcommand{\Omegatwo}{\Omega_{h,ij}^{c}}
\newcommand{\Omegathree}{\Omega_{h,ijk}^{c}}
\newcommand{\bu}{{\bf u} }
\newcommand{\bb}{{\bf b}}
\newcommand{\md}{\,\mathrm{d}}
\newcommand{\bk}{{\bf k}}
\newcommand{\bn}{{\bf n}}
\newcommand{\bzero}{{\bf 0}}
\newcommand{\bap}{{\bm{\alpha}}}
\newcommand{\mathR}{\mathbb{R}}
\newcommand{\calL}{\mathcal{L}}
\newcommand{\calE}{\mathcal{E}}
\newcommand{\p}{\partial}
\newcommand{\Og}{\Omega}
\newcommand{\fl}[2]{\frac{#1}{#2}}
\newcommand{\nn}{\nonumber}
\newcommand{\ap}{\alpha}
\newcommand{\apzero}{{\alpha_0}}
\newcommand{\Dt}{\Delta}
\renewcommand{\theequation}{\arabic{section}.\arabic{equation}}

\newcommand{\bea}{\begin{eqnarray}}
\newcommand{\eea}{\end{eqnarray}}
\newcommand{\beas}{\begin{eqnarray*}}
\newcommand{\eeas}{\end{eqnarray*}}
\newtheorem{remark}{Remark}[section]

\newtheorem{lemma}{Lemma}[section]
\newtheorem{assumption}{Assumption}[section]
\newtheorem{theorem}{Theorem}[section]
\newtheorem{corollary}{Corollary}[section]
\newtheorem{example}{Example}[section]
\newcommand{\bx}{{\bf x} }
\newcommand{\by}{{\bf y} }
\newcommand{\bxi}{\bm{\xi}}

\title{An accurate parallel finite difference method for the variable-order fractional Laplacian in heterogeneous media}

\author{Shiping Zhou\thanks{Department of Computational Mathematics, Science and Engineering, Michigan State University, East Lansing, MI 48824 (Email: zhouship@msu.edu)}, \ \
Yanzhi Zhang\thanks{Department of Mathematics and Statistics, Missouri University of Science and Technology, Rolla, MO 65409 (Email:  zhangyanz@mst.edu)}}

\begin{document}
\date{}
\maketitle

\begin{abstract} 
Anomalous diffusion in heterogeneous media is ubiquitous across physical and biological systems, where spatially varying microstructures give rise to distinct transport regimes.
Variable-order nonlocal operators provide a natural modeling framework for describing heterogeneity and transitions between diffusion regimes.
However, the position-dependent long-range interactions inherent in these operators impose prohibitive memory costs and make numerical methods developed for constant-order operators computationally impractical.
Consequently, accurate and computationally efficient numerical methods for variable-order nonlocal operators remain scarce.
In this paper, we propose accurate finite difference discretizations for the variable-order fractional Laplacian $(-\Delta)^{\alpha(\mathbf{x})/2}$ and develop a fast Fourier transform (FFT)–accelerated scheme for their efficient implementation.
The computational and memory costs are $\mathcal{O}\!\big(MN\log(2N)\big)$ and $\mathcal{O}(MN)$, respectively, where $N$ is the total number of spatial points and $M \ll N$.
The method is also naturally parallelizable.
We establish its consistency and accuracy through a rigorous error analysis that quantifies the local interplay between solution regularity and the spatially varying order $\alpha(\mathbf{x})$.
To the best of our knowledge, this is the first work to quantify how these two factors jointly determine convergence rates, with numerical experiments demonstrating this dependence.
Extensive numerical experiments confirm the accuracy and efficiency of the proposed method.
Applications to anomalous diffusion in composite materials further demonstrate its ability to capture the effects of material heterogeneity.
\end{abstract}


\section{Introduction}
\label{section1}

Anomalous diffusion in heterogeneous media has attracted increasing attention across many fields, ranging from material sciences \cite{Patnaik2020review, Patnaik2021-fracture}, geophysics \cite{Zhu2014}, astrophysics \cite{Giusti2020, Giusti2020on}, to biology \cite{Farquhar2018}.
In these media, the underlying microstructures may vary spatially, leading to distinct transport mechanisms across different regimes.
For example, a crossover from normal to anomalous diffusion of proteins and lipids occurs in crowded lipid membranes \cite{Lenzi2016}.
To describe these complex phenomena, traditional approaches often construct separate models for each regime, where the transitions between regimes require careful treatments \cite{Delia2021, Zhu2014}.
More recently, variable-order nonlocal operators have been employed to model anomalous transport in heterogeneous media, as they naturally encode heterogeneity and provide a seamless description of regime transitions \cite{Lenzi2016, Farquhar2018, Delia2022}.
Their effectiveness has been demonstrated in diverse applications such as nonlocal continuum mechanics \cite{Patnaik2020review}, fracture modeling \cite{Patnaik2021-fracture}, and seismic dispersion and attenuation \cite{Zhu2014, Mu2021}, where variable-order nonlocal models often require fewer parameters than conventional approaches while effectively capturing accurate physical properties.
More discussions of their applications can be found in recent reviews \cite{Patnaik2020review, Sun2019} and references therein.

In this work, we focus on anomalous diffusion in heterogeneous media governed by the variable-order fractional Laplacian, which is defined in a hypersingular integral form \cite{Bass1988, Kuhn2021, Luo2018, Chen2020, Samko1993}:
\begin{equation}\label{vofl-int}
    (-\Dt)^{\ap(\bx)/2}u(\bx) = c_{d,\ap(\bx)} \text{P.V.} \, \int_{\mathR^d} \fl{u(\bx)-u(\by)}{|\bx-\by|^{d+\ap(\bx)}} \md \by, \quad \ \mbox{for} \ \, 0  < \ap(\bx) < 2, 
\end{equation}
for dimension $d = 1, 2$, or $3$, where P.V. stands for the Cauchy principal value, and the function
\begin{equation*}
    c_{d,\ap(\bx)} = \fl{2^{\ap(\bx)-1} \ap(\bx) \Gamma{\Big(\fl{\ap(\bx)+d}{2}\Big)}}{\sqrt{\pi^d} \Gamma{\Big(1-\fl{\ap(\bx)}{2}\Big)}}
\end{equation*}
with $\Gamma(\cdot)$ denoting the Gamma function.
The variable-order fractional Laplacian in \eqref{vofl-int} can be viewed as a generalization of the celebrated constant-order fractional Laplacian \cite{Samko2013, Landkof1972}.
It is a nonlocal operator in which each point $\bx$ interacts with all the other points $\by\in\mathR^d$.
Crucially, the interaction strength is position-dependent, i.e., $\ap(\bx)$ is a function of $\bx$.
The order $\alpha(\bx)$, often associated with intrinsic material properties, naturally characterizes position-dependent features and enables the operator \eqref{vofl-int} to capture spatial variability \cite{Patnaik2020review, Sun2019}.
In the literature \cite{Kuhn2021}, it is required that $\alpha:\mathbb{R}^d \to (0,2)$ be Hölder continuous and satisfy $\inf_{\bx \in \mathbb{R}^d} \alpha(\bx) > 0$.
The operator \eqref{vofl-int} has been successfully applied in many fields \cite{Zhu2014, Giusti2020, zhang2025}.
For instance, it enables a single model to simultaneously capture the dispersion and attenuation of seismic waves propagating in viscoelastic media, providing a unified framework compared to traditional approaches that require separate models for dispersion and attenuation; see \cite{Zhu2014, Mu2021} and subsequent work.
In this context, the order $\alpha(\bx) = \fl{1}{\pi}{\rm arctan}(1/Q(\bx))$ depends on the material quality factor $Q(\bx)$.
We remark that other variable-order nonlocal operators are also used in the literature \cite{Du2013a, Tao2019, Delia2022}.
The variable-order fractional Laplacian \eqref{vofl-int} has the advantage of admitting a rigorous stochastic interpretation as the infinitesimal generator of a stable-like Feller process \cite{Chen2020, Kuhn2021, Samko1993}.

Let $\Og \subset {\mathbb R}^d$ denote an open bounded domain.
Consider the steady diffusion problem of the form \cite{Kuhn2021, Wu2024, Hao0024}:
\begin{equation}\label{steady}
\begin{aligned}
    (-\Dt)^{\ap(\bx)/2}u(\bx) &= f(\bx), \quad \text{ for }\bx\in\Og, \\
    u(\bx) &= g(\bx), \quad \,\text{ for }\bx\in\Og^{c},
\end{aligned}
\end{equation}
where $\Og^c=\mathR^d\backslash\Og$ represents the complement of domain $\Og$.
Note that \eqref{steady} serves as a fundamental building block for describing heterogeneous media in many applications.
For instance, \eqref{steady} has been used as a modification of Newtonian gravity in dark-matter phenomenology, where the spatially varying exponent $\ap(\bx)$ characterizes the distance dependence of galactic rotational velocities \cite{Giusti2020, Giusti2020on}.
The nonlocality and heterogeneity of the fractional Laplacian $(-\Dt)^{\ap(\bx)/2}$ pose significant challenges for both analytical and numerical studies.
The Schauder estimates for the problem \eqref{steady} have been established in \cite{Kuhn2021}.
The martingale problem associated with the operator \eqref{vofl-int}, including the existence and uniqueness of its solutions, has been studied from a probabilistic perspective in \cite{Bass1988, Tsuchiya1992}.
In addition, recent works \cite{Xiang2019, Vanthin2021, Wang2021, Liang2022} have also explored the existence of solutions to various PDEs involving the variable-order fractional Laplacian.

Compared to analytical studies, numerical investigations of the problem \eqref{steady} remain very limited, and numerical methods for the variable-order fractional Laplacian \eqref{vofl-int} are particularly underdeveloped.
The main difficulties stem from the combined challenges of nonlocality and spatial heterogeneity inherent to this operator, which impose substantial computational burdens.
Recently, several numerical approaches have been proposed to approximate the operator \eqref{vofl-int}.
In \cite{Wu2024}, mesh-free methods based on radial basis functions were introduced to discretize \eqref{vofl-int} with extended Dirichlet boundary conditions.
In \cite{zhang2025}, a fast Fourier pseudospectral method was proposed to solve variable-order fractional wave equations on a torus.
In \cite{Delia2022, Lei2023}, a truncated nonlocal operator is introduced to approximate $(-\Dt)^{\ap(\bx)/2}$, where the integral in \eqref{vofl-int} is restricted from $\mathbb{R}^d$ to a finite ball $B_\delta(\bx)$ with radius $\delta$.
Based on this truncated operator, various finite element methods were developed in \cite{Delia2022, Lei2023}, and a finite difference method was introduced in \cite{Alzahrani2022}.
On the other hand, instead of directly discretizing the operator \eqref{vofl-int}, two discrete operators based on the discrete Fourier transforms were proposed in \cite{Yu2022, Hao0024} to approximate the variable-order fractional Laplacian \eqref{vofl-int}.

Note that most existing studies have focused on approximating the operator~\eqref{vofl-int} itself.
However, two {\it key challenges} remain largely unresolved.
First, rigorous error analysis of the resulting numerical methods is still limited, particularly regarding the influence of the spatially varying order $\alpha(\mathbf{x})$.
Second, variable-order operators pose a significant computational bottleneck: their numerical discretizations typically yield dense matrices, making matrix assembly, storage, and matrix-vector multiplications prohibitively expensive.
In high-dimensional settings, the resulting matrices may be too large to assemble or store explicitly, rendering many discretization schemes impractical (see, e.g., \cite{Lei2023, Alzahrani2022}).
For example, in the constant-order case, numerical discretization of the fractional Laplacian yields a symmetric (multilevel) Toeplitz matrix, which requires storing only its first row or column and enables efficient matrix–vector multiplication via fast Fourier transforms (FFTs) \cite{Duo2018, Duo2019-FDM, Zhou2024}.
In contrast, in the variable-order case, the spatial dependence of $\ap(\bx)$ destroys the Toeplitz structure, making fast algorithms developed for constant-order operators inapplicable.
Nevertheless, {\it systematic studies of efficient implementation strategies remain scarce, despite their critical importance in the variable-order setting.}

In this paper, we develop finite difference discretizations for the variable-order fractional Laplacian $(-\Delta)^{\alpha(\bx)/2}$ and a fast, naturally parallelizable scheme for efficient and scalable implementation.
Our main contributions are twofold:
\vspace{-1.5mm}
\begin{itemize}\itemsep -1pt
\item To address the high memory and computational costs caused by the spatially varying order $\alpha(\bx)$, we separate the position dependence from the distance dependence in the kernel and thus approximate the variable-order fractional Laplacian by a finite sum of operators $\mathcal{L}_{d,m}$, $m=1,\ldots,M$, weighted by variable coefficients.
Each $\mathcal{L}_{d,m}$ is translation invariant, and its discretization yields a symmetric (multilevel) Toeplitz matrix, enabling efficient matrix--vector multiplication using FFTs.
Moreover, the contributions from different $m$ can be computed independently, making the method naturally parallelizable.
The resulting scheme substantially reduces both computational and memory costs, particularly in higher dimensions.
A comparison of the direct finite difference method (dFDM) and the fast finite difference method (fFDM) is provided in Table~\ref{tab:sec2}.

\item We establish a uniform error estimate for the finite difference method under local regularity assumptions.
The convergence rate is governed by the local solution regularity $p+\varepsilon(\hat x_k)$ at low-regularity points $\hat x_k$, through the regularity gaps $p+\varepsilon(\hat x_k)-\alpha(\hat x_k)$ for $\hat x_k\in\bar{X}\cap\overline{\Omega}$ and the exponents $p+\varepsilon(\hat x_k)+1$ for $\hat x_k\in\bar{X}\backslash\overline{\Omega}$.
The method is uniformly second-order accurate when all these quantities exceed two.
To the best of our knowledge, this is the first analysis to explicitly quantify how local solution regularity and the spatially varying order jointly determine the uniform convergence rate (see Theorem~\ref{thm}).
\end{itemize}
\vspace{-1mm}
Extensive numerical experiments verify the accuracy of our method and the efficiency of its parallel implementation.
They also illustrate the effects of spatial heterogeneity on anomalous diffusion in several applications.

\begin{table}[htb!]
\centering
\begin{tabular}{rlll}
\toprule
Scheme & memory Cost & Computational Cost & Parallelizable \\
\midrule
dFDM  & ${\mathcal O}(N^2)$     & ${\mathcal O}(N^2)$             & No \\
fFDM  & ${\mathcal O}(MN)$      & ${\mathcal O}(MN \log(2N))$        & Yes \\
\bottomrule
\end{tabular}
\caption{Comparison of memory and computational costs of dFDM and fFDM in approximating \eqref{vofl-int}, where $N$ represents the total number of spatial points and $M \ll N$.
}\label{tab:sec2}
\end{table}

The rest of this paper is organized as follows.
Section~\ref{section2} introduces both the direct and fast finite difference discretizations for the one-dimensional variable-order fractional Laplacian, with detailed error analysis provided in Appendix~\ref{thm1proof}.
In Section~\ref{section3}, we extend the fast finite difference method to the general $d$-dimensional cases and present the parallel algorithms.
Section~\ref{section4} reports numerical experiments that verify the accuracy and efficiency of our methods and demonstrate their application to both time-independent and time-dependent diffusion problems in composite materials.
Finally, Section~\ref{section5} concludes the paper.

\section{Finite difference methods}
\label{section2}
\setcounter{equation}{0}

The main challenges in approximating the variable-order fractional Laplacian $(-\Dt)^{\ap(\bx)/2}$ lie in its nonlocality and hypersingular kernel.
Moreover, the spatial dependence of the order $\ap(\bx)$ introduces significant new computational challenges \cite{Wu2024, zhang2025, Lei2023, Delia2022}, making most existing methods designed for the constant-order fractional Laplacian infeasible for the variable-order case.
To address these challenges, we first present the key idea of our finite difference scheme and provide rigorous error estimates.
We then develop a fast and parallelizable finite-difference method in Section~\ref{section2-2} to mitigate storage and computational costs.

In the following, we focus on the one-dimensional ($d = 1$) case.
The generalization of our method to higher dimensions is presented in Section~\ref{section3}.
For $d = 1$, we first formulate the variable-order fractional Laplacian in \eqref{vofl-int} into the form:
\begin{equation}\label{neq}
    (-\Dt)^{\ap(x)/2} u(x) = -c_{1,\ap(x)}\int_0^\infty \fl{u(x+\xi)-2u(x) + u(x-\xi)}{\xi^{1+\ap(x)}}\md\xi.\qquad\qquad 
\end{equation}
It clearly reveals the nonlocal nature of the operator, as every point $x$ has an infinite \textit{interaction domain}.
Moreover, the interaction strength depends on the distance between two points (i.e., $\xi$) and the value of $\ap(x)$ at point $x$.
To address the hypersignularity in (\ref{neq}), we introduce
\beas
\psi(x, \xi) = \fl{u(x+\xi)-2u(x) + u(x-\xi)}{\xi^2},\qquad w(x, \xi) = \xi^{1-\ap(x)},
\eeas
and then rewrite (\ref{neq}) into a weighted integral
\begin{equation}\label{newL}
    (-\Dt)^{\ap(x)/2} u(x) = -c_{1,\ap(x)}\int_0^\infty \psi(x, \xi)\,w(x, \xi)\md\xi, \qquad\qquad 
\end{equation}
that is, the singularity is split into two parts, and the dependence of $\alpha(x)$ is retained in the weight function $w(x, \xi)$.
Note that in the constant-order case, the weight function $w$ is independent of $x$, which significantly simplifies the computation.

Next, we discretize \eqref{newL}.
Let $\Og=(a,b)$ denote the one-dimensional \textit{physical domain}.
Choose an integer $N > 0$, and define mesh size $h=L/N$ with $L = |b-a|$.
Denote the grid points in the physical domain as $x_i = a + ih$ for $i = 0, 1, \ldots, N$, and in the interaction domain as $\xi_k = kh$ for $k = 0, 1, \ldots$, respectively.
Let $\Og_h:=\{x_i=a+ih: i=1,\dots,N-1\}$ denote the set of interior grid points.
We then partition the domain of $\xi$ into two parts, i.e. $[0, L]$ and $(L, \infty)$, and rewrite the integral in (\ref{newL}) as
\begin{equation}\label{fdm-1d-split}
    (-\Dt)^{\ap(x)/2} u(x) = -c_{1, \ap(x)}\bigg(\sum_{k=1}^{N}\int_{\xi_{k-1}}^{\xi_k}\,\psi(x,\xi)\,w(x,\xi)\md\xi + \int_{L}^{\infty}\psi(x,\xi)\,w(x,\xi)\md\xi\bigg).
\end{equation}
For each $k$, the integral over $[\xi_{k-1}, \xi_k]$ can be approximated by the weighted trapezoidal method \cite{Duo2018, Duo2019-TFL}, i.e.,
\begin{equation}\label{part1}
    \int_{\xi_{k-1}}^{\xi_{k}}\psi(x,\xi)\,w(x,\xi)\md\xi \approx \fl{1}{2}\Big(\psi(x,\xi_{k-1})+\psi(x,\xi_{k})\Big)     \int_{\xi_{k-1}}^{\xi_{k}}w(x,\xi)\md\xi, \quad\  k = 1, \ldots, N,
\end{equation}
where the integral of $w$ can be found analytically.
Note that at $\xi = 0$, the function $\psi(x, 0)$ is defined in the limit sense, i.e.,
\bea\label{phi_xi0}
\psi(x, 0) = \lim_{\xi \to 0} \psi(x, \xi) \approx  \fl{u(x+h) - 2u(x) + u(x-h)}{h^2}
\eea
with a sufficiently small $h$.
While the second part of \eqref{fdm-1d-split} can be simplified to
\begin{equation}\label{part2}
\int_{L}^{\infty}\psi(x,\xi)\,w(x,\xi)\md\xi = -\fl{2}{\ap(x) L^{\ap(x)}} u(x) + \int_{L}^{\infty}\fl{g(x+\xi)+g(x-\xi)}{\xi^{1+\ap(x)}}\md\xi,
\end{equation}
by the boundary condition $u(x) = g(x)$ for $x \in \Og^c$, since $x \pm \xi \in \Og^c$ for any $x \in \Og$ and $\xi > L$.

Combining (\ref{fdm-1d-split}) with (\ref{part1})--(\ref{part2}), noticing the definition of $\psi(x, \xi)$, and simplifying the calculation, we then obtain the finite difference discretization of the variable-order fractional Laplacian at point $x = x_i$:
\begin{equation}\label{1Dscheme}
 (-\Dt)^{\ap(x)/2}_{h} u(x)\Big|_{x = x_i} = \fl{c_{1, \ap_i}}{(2-\ap_i)h^{\ap_i}} \bigg[\bigg(\sum_{k = 1}^N a_{ik} +  \fl{2(2-\ap_i)}{\ap_i N^{\ap_i}}\bigg)u_i - \fl{1}{2}\sum_{\substack{j=1,  j\neq i}}^{N-1}a_{i|j-i|} u_j \bigg] + b_i + r_i, 
\end{equation}
for $i = 1, 2, \ldots, N-1$, where $\ap_i=\ap(x_i)$, $u_i = u(x_i)$, and the coefficients
\beas
&&a_{ik} = \fl{(k+1)^{2-\ap_i}-(k-1)^{2-\ap_i}}{k^2}, \qquad \mbox{for} \ \ 1 < k < N, \\
&&a_{i1} = 2^{2-\ap_i}+1, \qquad a_{iN} = \fl{N^{2-\ap_i} - (N-1)^{2-\ap_i}}{N^2}. 
\eeas
The terms $b_i$ and $r_i$ account for the contributions of boundary conditions from \eqref{part1} and \eqref{part2}, respectively, and are given by
\beas
&&b_i=-\fl{c_{1, \ap_i}}{2(2-\ap_i)h^{\ap_i}}\bigg(\sum_{j={i-N}}^{0} a_{i|i-j|}\, g(x_j) + \sum_{j=N}^{i+N} a_{i|i-j|}\,g(x_j)\bigg),\\
&&r_i=-c_{1, \ap_i}\int_{L}^{\infty} \fl{g(x_{i} + \xi) + g(x_{i} -\xi)}{\xi^{1+\ap_i}}\md\xi.
\eeas
If homogeneous Dirichlet boundary conditions are considered, we have $b_i = r_i \equiv 0$ for $1 \le i \le N-1$.
Denote the vectors ${\bf u} = (u_1, u_2, \,\ldots,\, u_{N-1})^T$, $\bb = (b_1, b_2, \,\ldots,\, b_{N-1})^T$, and ${\bf r} = (r_1, r_2, \,\ldots,\, r_{N-1})^T$.
Then the scheme in \eqref{1Dscheme} can be written into a matrix-vector form $(-\Dt)_h^{\ap(x)/2}\bu = A \bu + \big(\bb + {\bf r}\big)$ with
$$
A_{ii} = \fl{c_{1, \ap_i}}{(2-\ap_i)h^{\ap_i}} \bigg(\sum_{k = 1}^N a_{ik} + \fl{2(2-\ap_i)}{\ap_i N^{\ap_i}}\bigg),\qquad A_{ij} = -\fl{c_{1, \ap_i}}{2(2-\ap_i)h^{\ap_i}} a_{i|j-i|}, \quad \mbox{for} \ i\neq j.
$$
It is easy to see that the matrix $A_{(N-1)\times (N-1)}$ is a dense matrix, due to the nonlocality of the variable-order fractional Laplacian, which introduces a substantial burden on storing $A$ and performing matrix–vector multiplications.

In the constant-order cases, i.e., $\ap(x)\equiv\ap$, the matrix $A$ reduces to a symmetric Toeplitz matrix, which allows the development of FFT-based algorithms to efficiently perform matrix–vector multiplications.
Consequently, the memory and computational costs in the constant-order case reduce to ${\mathcal O}(N)$ and ${\mathcal O}(2N\log(2N))$, respectively; see \cite{Duo2018, Duo2019-FDM} for more details.
In the variable-order case, however, the matrix no longer possesses a Toeplitz structure, and thus FFT-based fast algorithms cannot be applied to accelerate matrix–vector multiplication because of the spatial dependence of $\alpha(x)$.
Consequently, the memory and computational costs of implementing the scheme \eqref{1Dscheme} are ${\mathcal O}(N^2)$ and ${\mathcal O}(N^2)$, respectively, both of which grow rapidly with $N$.

\subsection{Error analysis}
\label{section2-1}

In this subsection, we present error estimates of our method \eqref{1Dscheme} in approximating the variable-order fractional Laplacian $(-\Dt)^{\ap(x)/2}$.
Compared to the constant-order cases \cite{Duo2018, Duo2019-FDM}, the error analysis in the variable-order case is substantially more complicated.
Since both $\alpha(x)$ and $u(x)$ vary spatially, their interaction plays an important role in determining the accuracy of numerical methods, which is not observed in the analysis of constant-order cases.

For a nonnegative integer $p$ and $s\in(0,1]$, we write $f\in C^{p,s}(D)$ if $f\in C^p(D)$ and $D^pf$ is H\"older continuous on $D$ with exponent $s$, that is, there exists a constant $C>0$ such that $|D^pf(x)-D^pf(y)|\le C|x-y|^s$ for all $x,y\in D$.
To describe regularity by a single exponent $q>0$, we adopt the decomposition $q=\lfloor q\rfloor^-+\{q\}$, where $\lfloor q\rfloor^-:=\lceil q\rceil-1$ is the largest integer strictly less than $q$ and $\{q\}:=q-\lfloor q\rfloor^-\in(0,1]$.
Unlike the standard fractional part, $\{q\}=1$ when $q$ is a positive integer, so that the H\"older exponent is always positive and $C^{\lfloor q\rfloor^-,\{q\}}(D)$ is well defined for every $q>0$.

Let $\Omega_L:=[a-L,b+L]$.
The following assumption specifies the regularity conditions required for the error analysis.

\begin{assumption}\label{assump}
The following conditions hold.
\leavevmode
\begin{enumerate}[label=\textup{(A\arabic*)},ref=A\arabic*]
\item\label{A1}\emph{Regularity and bounds of the variable order.} There exist constants $0<\alpha_-\le\alpha_+<2$, $\gamma\in(0,1]$, and $C_\alpha>0$ such that
$$
\alpha_-\le\alpha(x)\le\alpha_+,
\qquad
|\alpha(x)-\alpha(y)|\le C_\alpha|x-y|^\gamma,
\qquad x,y\in\overline{\Omega}.
$$

\item\label{A2}\emph{Local regularity of $u$.} There exist finitely many distinct points $\hat x_1,\dots,\hat x_K\in\Omega_L$, a nonnegative integer $p$, and positive exponents $\varepsilon(\hat x_k)$ such that $\sigma_k:=p+\varepsilon(\hat x_k)\in(0,4)$.
Set $\bar X=\{\hat x_1,\dots,\hat x_K\}$ and assume $u\in C^0(\Omega_L)\cap C^4(\Omega_L\setminus\bar X)$.
The case $K=0$ is allowed; then $\bar X=\emptyset$ and $u\in C^4(\Omega_L)$.
There exist a constant $\delta>0$ for which the intervals $B_\delta(\hat x_k):=(\hat x_k-\delta,\hat x_k+\delta)$ are pairwise disjoint and a constant $C_u>0$ satisfying the following bounds.
\begin{enumerate}[label=\textup{(\roman*)},ref=\textup{(\roman*)}]
\item\label{A2i} Outside the low-regularity neighborhoods,
\beas
|u^{(m)}(\zeta)|\le C_u,
\qquad
\zeta\in\Omega_L\setminus\bigcup_{k=1}^K B_\delta(\hat x_k),
\qquad m=0,\dots,4.
\eeas
\item\label{A2ii} For each $k$, there exists a polynomial $P_k$ of degree at most three such that $\widetilde r_k:=u-P_k$ satisfies
\beas
|\widetilde r_k^{(m)}(\zeta)|\le C_u|\zeta-\hat x_k|^{\sigma_k-m},
\qquad 
0<|\zeta-\hat x_k|<\delta,\quad \zeta\in\Omega_L,
\qquad m=0,\dots,4.
\eeas
\end{enumerate}

\item\label{A3}\emph{Compatibility.
} For every $\hat x_k\in\bar X\cap\overline\Omega$, we have $\sigma_k>\alpha(\hat x_k)$.
\end{enumerate}
\end{assumption}

Define the error function
\begin{equation}\label{error}
e_{\Dt}^h(x)=(-\Dt)^{\ap(x)/2}u(x)-(-\Dt)_h^{\ap(x)/2}u(x), \qquad x\in\Og_h,
\end{equation}
where $(-\Dt)_h^{\ap(x)/2}$ represents the finite difference approximation of $(-\Dt)^{\ap(x)/2}$ with mesh size $h$.
Then we have the following error estimates.

\begin{theorem}[Uniform consistency]\label{thm}
Suppose that Assumption~\ref{assump} holds, and define
\begin{equation}\label{mu}
\mu:=\min\Bigl\{
\min_{\hat x_k\in\bar X\cap\overline\Omega}
\bigl(\sigma_k-\alpha(\hat x_k)\bigr),\
\min_{\hat x_k\in\bar X\setminus\overline\Omega}
(\sigma_k+1)
\Bigr\},
\end{equation}
with the convention that the minimum over an empty set is $+\infty$.
Then there exist constants $h_0\in(0,1)$ and $C>0$, independent of $x$ and $h$, such that for all $h < h_0$,
\begin{equation}\label{thm1-ineq2}
\|e_\Dt^h\|_{\ell^\infty(\Og_h)} =  \max_{x\in\Og_h}|e_\Delta^h(x)| 
\le C
\begin{cases}
h^\mu, & 0<\mu<2,\\
h^2\log(1/h), & \mu=2,\\
h^2, & \mu>2.
\end{cases}
\end{equation}
In particular, $\mu>0$ follows from Assumption~\eqref{A3} and the positivity of the exterior exponents.
Hence, the approximation is uniformly consistent in the discrete maximum norm and is second-order accurate when $\mu>2$.
\end{theorem} 
The detailed proof of Theorem~\ref{thm} is given in Appendix~\ref{thm1proof}.
The uniform convergence rate is governed by the minimum of the local regularity gaps $\sigma_k-\alpha(\hat x_k)$ at low-regularity points in $\overline{\Omega}$ and the exponents $\sigma_k+1$ at those outside $\overline{\Omega}$, subject to the second-order threshold.
Thus, the estimate captures the interaction between solution regularity and the fractional order at each low-regularity point, rather than relying on the maximal order over the physical domain.
This distinguishes our estimate from that in \cite{Hao0024}, which simplifies the analysis by using the maximal order $\alpha_{\max}$; see Remark~\ref{remark-Hao} for a detailed comparison.

\begin{corollary}[Uniform second-order accuracy]
Under the assumptions of Theorem~\ref{thm}, the scheme \eqref{1Dscheme} is uniformly second-order accurate if
$$
\sigma_k>\alpha(\hat x_k)+2
\quad\text{for every }\hat x_k\in\bar X\cap\overline{\Omega},
\qquad
\sigma_k>1
\quad\text{for every }\hat x_k\in\bar X\setminus\overline{\Omega}.
$$
These conditions are equivalent to $\mu>2$ and hold vacuously when $\bar X=\varnothing$.
At the critical threshold $\mu=2$, Theorem~\ref{thm} instead gives the uniform bound $O(h^2[1+\log(1/h)])$.
\end{corollary}
To illustrate the interplay between the regularity of $u(x)$ and the order $\ap(x)$, Fig.~\ref{fig:thm1} shows the numerical errors for different mesh sizes $h$ using the compactly supported function $u(x)=[(4-x^2)_+]^q$.
Its low-regularity points are $x=\pm2$, with local regularity exponents $\sigma_1=\sigma_2=q$.
We consider two domains to compare the effects of boundary and exterior low-regularity points.
\begin{figure}[htb!]
    \centerline{
    \raisebox{5mm}{(a)}\includegraphics[height=5.19cm, width=0.43\linewidth]{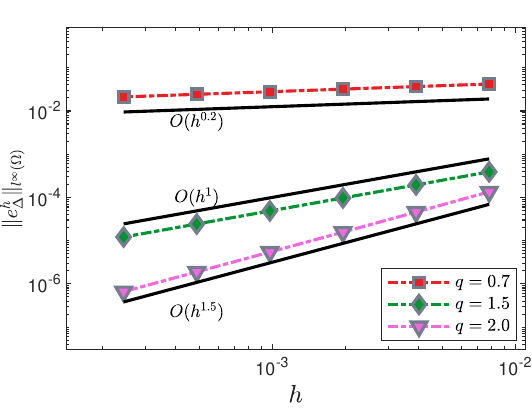}
    \hspace{2mm}
    \raisebox{5mm}{(b)}\includegraphics[height=5.19cm, width=0.43\linewidth]{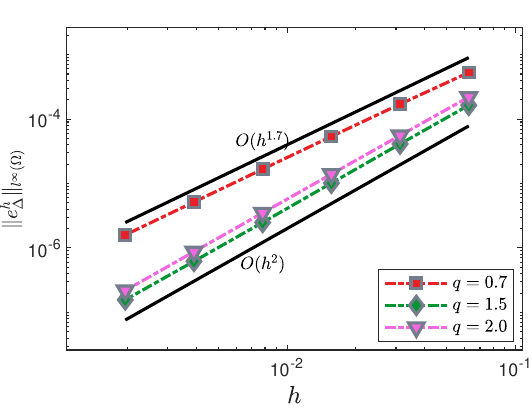}}
    \vspace{-4mm}
\caption{Numerical errors in approximating $(-\Delta)^{\alpha(x)/2}u(x)$ with $u(x) = \big[(2^2-x^2)_+\big]^q$ and $\alpha(x)=0.3-0.1x$.
The errors are measured on: (a) $\Omega=(-2,2)$; (b) $\Omega=(-1,1)$.}
    \label{fig:thm1}
\end{figure}
In Fig.~\ref{fig:thm1}(a), $\Omega=(-2,2)$, so both low-regularity points lie on the boundary.
Since $\alpha(-2)=0.5$ and $\alpha(2)=0.1$, the smaller regularity gap occurs at $x=-2$, giving $\mu=q-0.5$.
The observed convergence rates agree with the bounds in Theorem~\ref{thm}.
In Fig.~\ref{fig:thm1}(b), $\Omega=(-1,1)$, so $u$ is smooth over $\overline{\Omega}$ and both low-regularity points lie outside the domain, giving $\mu=q+1$.
Theorem~\ref{thm} therefore predicts order $q+1$ for $0<q<1$ and second-order accuracy for $q>1$.
The observed rates agree well with these predictions.
More results and discussion are provided in Section~\ref{section4-1}.

\begin{remark}[Constant-order fractional Laplacian]
\label{remark-constant}
When $\alpha(x)\equiv\alpha\in(0,2)$, condition \textup{(A1)} holds automatically, and Theorem~\ref{thm} applies under \textup{(A2)}--\textup{(A3)} with
\begin{equation}\label{thm1-ineq3}
\mu=\min\Bigl\{
\min_{\hat x_k\in\bar X\cap\overline{\Omega}}
(\sigma_k-\alpha),
\quad
\min_{\hat x_k\in\bar X\setminus\overline{\Omega}}
(\sigma_k+1)
\Bigr\}.
\end{equation}
If homogeneous Dirichlet boundary conditions are imposed (i.e., $g(x)\equiv0$), there are no strictly exterior low-regularity points.
The convergence rate is therefore governed by the smallest gap $\sigma_k - \alpha$ at interior and boundary low-regularity points, subject to the second-order threshold.
In particular, if $\sigma_k\ge s>\alpha$ at all such points, the uniform error is bounded by $O(h^{\min\{s-\alpha,2\}})$, with a possible logarithmic factor when $s-\alpha=2$.
This provides a connection with constant-order error estimates based on global solution regularity, such as those considered in \cite{Duo2018}.
\end{remark}
Remark~\ref{remark-constant} connects Theorem~\ref{thm} with existing constant-order estimates \cite{Duo2018}, while extending the analysis to nonhomogeneous Dirichlet boundary conditions.
The theorem provides a unified framework for uniform error estimates for both constant- and variable-order fractional Laplacians.
Its central contribution is to quantify the local interplay between solution regularity and the fractional order, while accounting for low-regularity points both inside and outside the domain.
This distinguishes our analysis from estimates based solely on global regularity and the maximal fractional order.

\begin{remark}\label{remark-Hao}
In \cite{Hao0024}, a discrete Fourier transform-based operator is proposed to approximate the variable-order fractional Laplacian.
For $u\in B^{s+\alpha_{\max}}(\mathbb{R})$, the authors establish a truncation error bound of $\mathcal{O}(h^s)$, where $B^s(\mathbb{R})$ denotes the spectral Barron space and $\alpha_{\max}:=\sup_{x\in\Omega}\alpha(x)$.
This estimate measures the regularity of $u$ globally against the maximal fractional order $\ap_{\max}$.
In contrast, Theorem~\ref{thm} accounts for the local regularity of $u$ relative to $\alpha(\hat x_k)$ at each low-regularity point in $\overline{\Omega}$.
For a given convergence rate, our conditions allow lower local solution regularity at points where the fractional order is smaller.
\end{remark}

\subsection{Fast finite difference method}
\label{section2-2}

As discussed, the matrix associated with scheme \eqref{1Dscheme} no longer has a Toeplitz structure due to the spatial dependence of $\alpha(x)$.
Consequently, the FFT-based fast algorithms developed for the constant-order cases in \cite{Duo2018, Duo2019-FDM} cannot be applied to accelerate matrix–vector multiplications in the variable-order setting.
In contrast, direct computation requires storing the full dense matrix and performing dense matrix–vector products, and consequently both memory and computational costs of scheme \eqref{1Dscheme} are ${\mathcal O}(N^2)$.
These costs increase rapidly with the number of spatial points $N$, making the direct approaches computationally infeasible, particularly in higher dimensions ($d>1$); see Tables~\ref{Tab4-2-1} and ~\ref{tab:2d-cpu} for further discussion.
To overcome these challenges, we propose a fast and parallelizable finite difference method in this subsection.

From \eqref{newL}, it is evident that the spatial dependence of the weight function $ w(x,\xi) $ causes the matrix $ A $ to lose its Toeplitz structure.
Therefore, to recover the Toeplitz structure, we propose to decouple the $x$- and $\xi$-dependence of the weight function $w(x,\xi)$.
In particular, for any $ \alpha_0 \in (0,2) $, the weight function in \eqref{newL} can be rewritten as
\begin{equation}\label{taylor}
   w(x, \xi) = \xi^{1-\ap_0} \sum_{m=0}^{\infty}\fl{\big[\ap_0-\ap(x)\big]^m \big(\ln{\xi}\big)^m}{m!}.
\end{equation}
In this paper, we choose $\ap_0 = \fl{1}{2}\big(\sup_{x\in\Og}\ap(x)+\inf_{x\in\Og}\ap(x)\big)$.
Note that in the literature \cite{Zhu2014}, a constant-order fractional Laplacian with the averaged power $\bar{\ap} = \int_{\Og}\ap(x) \md x$ has been used to approximate \eqref{vofl-int} in order to alleviate the memory and computational challenges caused by the spatial dependence of $\alpha(x)$.
However, while this approximation significantly simplifies the computation, it homogenizes the problem and fails to accurately capture the inherent heterogeneity of the system (see Fig.~\ref{fig:demo1}(a) and (c)).

Substituting (\ref{taylor}) into (\ref{newL}), we then rewrite the variable-order fractional Laplacian as
\bea\label{FL-exp}
(-\Dt)^{\ap(x)/2}u(x) = -c_{1,\ap(x)} \sum_{m=0}^{\infty}{\big(\ap_0-\ap(x)\big)^m }\underbrace{\int_{0}^{\infty}\psi(x,\xi)\,\fl{\xi^{1-\ap_0}(\ln\xi)^m}{m!}\md\xi}_{{\mathcal L}_{m}u(x)}, 
\eea
where we introduce the nonlocal operator ${\mathcal L}_{m}$ for $m\in {\mathbb N}^0$.
Note that ${\mathcal L}_{m}$ has a singular kernel at $\xi = 0$, but it is spatially homogeneous as its nonlocality is independent of $x$.
For $m = 0$, we have
\beas
-c_{1,\ap_0}{\mathcal L_0}u(x) = (-\Dt)^{\ap_0/2},
\eeas
that is, the leading term ${\mathcal L}_0$ can be viewed as an approximation to the constant-order fractional Laplacian $(-\Dt)^{\ap_0/2}$.
In rewriting \eqref{FL-exp}, the spatial dependence of $\alpha(x)$ is transferred into variable coefficients.
Thus, the variable-order fractional Laplacian in \eqref{FL-exp} may be interpreted as a spatially dependent perturbation of the constant-order operator $(-\Delta)^{\alpha_0/2}$.
Importantly, although these perturbation terms decay rapidly as $m$ increases, they play a crucial role in capturing heterogeneity and therefore cannot be neglected.

\begin{figure}[htb!]
\centerline{
    \raisebox{5mm}{(a)}
    \includegraphics[height=5.19cm, width=0.4\linewidth]{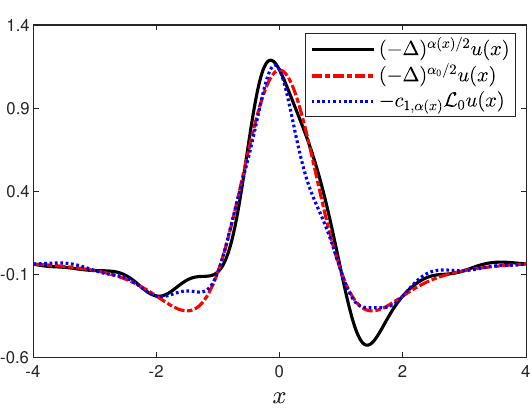}\hspace{3mm}
    \raisebox{5mm}{(b)}
    \includegraphics[height=5.19cm, width=0.4\linewidth]{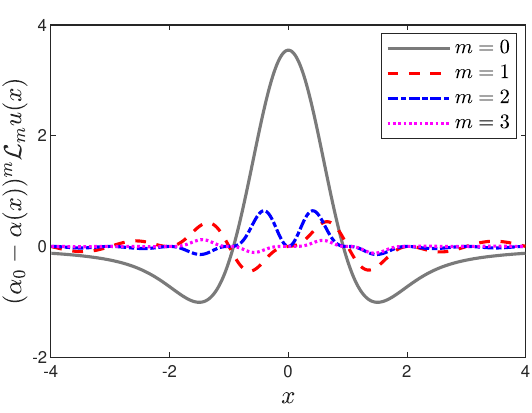}}
\centerline{
    \raisebox{5mm}{(c)}
    \includegraphics[height=5.19cm, width=0.4\linewidth]{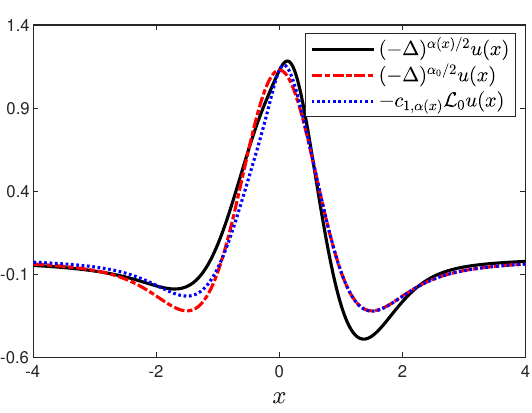}\hspace{3mm}
    \raisebox{5mm}{(d)}
    \includegraphics[height=5.19cm, width=0.4\linewidth]{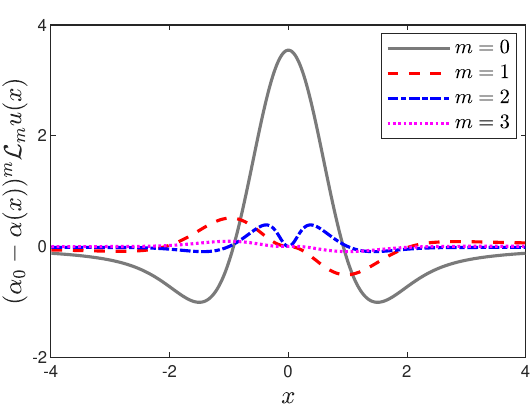}}
\vspace{-3mm}
\caption{(a) and (c): Comparison of $(-\Dt)^{\ap(x)/2}u(x)$ and $(-\Dt)^{\ap_0/2}u(x)$; (b) and (d): comparison of $(\ap_0-\ap(x))^m{\mathcal L}_mu(x)$ for different $m$, where $u(x) = e^{-x^2}$.
(a)-(b): $\ap(x) = 1-0.5\sin(\pi x)$; (c)-(d): $\ap(x) = 1+0.4{\rm tanh}(4x)$.}\label{fig:demo1}
\end{figure}
To illustrate this, we compare $(-\Delta)^{\alpha(x)/2}u(x)$ with $(-\Delta)^{\alpha_0/2}u(x)$ in Fig.~\ref{fig:demo1}(a) and (c), where $u(x)=e^{-x^2}$.
Even though their amplitudes are similar, the two functions differ significantly, with the variable-order fractional Laplacian exhibiting clear spatial heterogeneity.
This suggests that the constant-order operator can serve as a leading approximation to the variable-order fractional Laplacian, but it fails to capture the underlying heterogeneity of the problem.
Furthermore, Fig.~\ref{fig:demo1}(b) and (d) highlight the heterogeneous contributions of the terms $(\alpha_0-\alpha(x))^m \mathcal{L}_m$ for $m \ge 1$.
Among these terms, $\mathcal{L}_0$ is substantially larger than the others and therefore dominates the solution.
While the contributions of higher-order terms decay rapidly as $m$ increases, they nevertheless capture essential heterogeneous effects.
In practice, we truncate the expansion in \eqref{FL-exp} after $M$ terms, where $M$ is chosen large enough.

The formulation in \eqref{FL-exp} shows that discretizing the variable-order fractional Laplacian can be decomposed into approximating the operators ${\mathcal L}_m$ for $m \in \mathbb{N}^0$.
Each operator ${\mathcal L}_m$ can be discretized in a manner similar to \eqref{part1}–\eqref{part2}, with the weight function $w_m(\xi) := \xi^{1-\alpha_0}(\ln \xi)^m/m!$.
For completeness, we briefly outline the main steps below.
For each $m$, we have
\begin{equation}\label{disc-1d}
{\mathcal L}_m u(x) = \sum_{k = 1}^N\int_{\xi_{k-1}}^{\xi_k}\psi(x, \xi) w_{m}(\xi) \md\xi + \int_L^\infty\psi(x, \xi)w_{m}(\xi) \md\xi, \qquad \mbox{for}\ \,  m = 0, 1, \ldots.
\end{equation}
For each $k$, the integral over $[\xi_{k-1}, \xi_k]$ can be first approximated using the weighted trapezoidal rule as in \eqref{part1}, and then the integral of $w_m(\xi)$ can be evaluated analytically.
In detail, we obtain
\begin{equation}\label{int-1d-I}
\int_{\xi_{k-1}}^{\xi_k}\psi(x, \xi) w_{m}(\xi) \md\xi \approx 
       \fl{1}{2}\Big(\psi(x, \xi_k) + \psi(x, \xi_{k-1})\Big)\sum_{p=0}^{m} (-1)^{p-m}\fl{(\ln{\xi_{k}})^{p}\xi_{k}^{2-\ap_0}-(\ln{\xi_{k-1}})^{p}\xi_{k-1}^{2-\ap_0}}{(2-\ap_0)^{m+1-p}(p!)},
\end{equation}
for any $m \in {\mathbb N}^0$ and $2 \le k \le N$.
While for $k=1$, we obtain
\begin{equation}\label{int-1d0}
    \int_{\xi_0}^{\xi_1}\psi(x,\xi)w_{m}(\xi) \md\xi \approx  \psi(x,h)\sum_{p=0}^{m}(-1)^{p-m}\fl{ (\ln h)^p h^{2-\ap_0}}{(2-\ap_0)^{m+1-p}(p!)},
\end{equation} 
using the fact $\xi_1 = h$ and \eqref{phi_xi0}.
Following the similar lines as in (\ref{part2}), we can simplify the second term in (\ref{disc-1d}) as
\begin{equation}\label{int-1d-II}
\int_L^\infty\psi(x, \xi) w_{m}(\xi) \md\xi = -2u(x) \sum_{p=0}^{m}\fl{(\ln{L})^p}{\ap_0^{m+1-p}L^{\ap_0}(p!)} + \int_{L}^{\infty}\Big(g(x+\xi)+g(x-\xi)\Big)\fl{(\ln{\xi})^m}{\xi^{1+\ap_0} (m!)}\md\xi.
\end{equation}
Substituting \eqref{int-1d-I}--\eqref{int-1d-II} into \eqref{disc-1d}, using the definition of $\psi(x, \xi)$, and simplifying the calculation, we obtain the finite difference discretization of $\calL_{m}$ at $x = x_i$ as
\begin{equation}\label{Lmh}
    \calL_{m}^h u(x)\mid_{x = x_i} = 
    -\fl{1}{h^{\ap_0}}\sum_{p=0}^{m}\bigg[\bigg(2\sum_{k=1}^{N}a_{k}^{(m,p)}+\fl{2(\ln L)^p}{\apzero^{m+1-p}N^\apzero(p!)} \bigg)u_i-\sum_{\substack{j=1\\j\neq i}}^{N-1}a_{|i-j|}^{(m, p)} u_j\bigg] + b_{i}^{(m)} + r_{i}^{(m)},
\end{equation} 
where the coefficients $a_k^{(m,\,p)}$ are
\begin{equation*}
a^{(m,\, p)}_k = \fl{(-1)^{p-m}}{2 (2-\ap_0)^{m+1-p}(p!) k^2}
\begin{dcases}
2^{2-\ap_0}\big(\ln{\xi_2}\big)^{p}+\big(\ln{\xi_{1}}\big)^{p}, \hspace{4.1cm} \text{ if } k = 1, \\
{(k+1)}^{2-\ap_0}\big(\ln{\xi_{k+1}}\big)^{p}-{(k-1)}^{2-\ap_0}\big(\ln{\xi_{k-1}}\big)^{p}, \quad\text{ if } 2\le k \le N-1, \\
{N}^{2-\ap_0}\big(\ln{\xi_{N}}\big)^{p}-{(N-1)}^{2-\ap_0}\big(\ln{\xi_{N-1}}\big)^{p}, \hspace{1.25cm} \text{ if } k = N, 
\end{dcases}
\end{equation*}
for $0 \le p \le m$ and $m \in {\mathbb N}^0$.
The terms $b_{i}^{(m)}$ and $r_{i}^{(m)}$ arise from the boundary conditions, given by
\begin{align}\label{Lbm}
 b_{i}^{(m)} &= \fl{1}{h^\apzero}\sum_{p=0}^{m}\bigg(\sum_{j=i-N}^{0}a_{|i-j|}^{(m,p)}g(x_j)+\sum_{j=N}^{i+N}a_{|i-j|}^{(m,p)}g(x_j)\bigg),\\
\label{Lgm}
 r_{i}^{(m)} &=  \int_{L}^{\infty}\fl{\big(\ln{\xi}\big)^m}{\xi^{1+\ap_0}(m!)}\Big(g(x_i+\xi) + g(x_i-\xi)\Big)\md\xi. 
\end{align}
Note that the summation over $p$ in \eqref{Lmh} arises from the analytical integration of $w_m(\xi)$ on $[\xi_{k-1}, \xi_{k+1}]$, although it could alternatively be computed numerically.

For each $m \in {\mathbb N}^0$, the numerical scheme of ${\mathcal L}^h_m u(x)$ in \eqref{Lmh} can be written into a matrix-vector form, i.e., ${\mathcal L}_{m}^h\bu = A^{(m)}\bu + \bb^{(m)} + {\bf r}^{(m)}$, where the entries of $\bb^{(m)}$ and ${\bf r}^{(m)}$ are given in \eqref{Lbm} and \eqref{Lgm}, respectively.
The matrix $A^{(m)}$ is a symmetric Toeplitz matrix with its entries given by
\begin{equation*}
    A^{(m)}_{ij} = \fl{1}{h^{\ap_0}}
\begin{dcases}
-2\sum_{p=0}^{m}\bigg(\sum_{k=1}^{N}a_k^{(m,p)}+\fl{(\ln L)^p}{\apzero^{m+1-p}N^\apzero(p!)} \bigg), \hspace{0.6cm} \text{ for } i=j, \\
\sum_{p=0}^{m}a_{|i-j|}^{(m, p)}, \hspace{5.8cm} \text{ for } i \neq j.\\
\end{dcases}
\end{equation*} 
Combining \eqref{FL-exp} with \eqref{Lmh}, we obtain the {\it fast} finite difference method for the one-dimensional variable-order fractional Laplacian as
\begin{equation}\label{fFDM-1D}
    (-\Dt)^{\ap(x)/2}_{h} \bu = \sum_{m=0}^{M-1}\Big[ {{\bf c}^{(m)}}\circ\big(A^{(m)} \bu +  \bb^{(m)}\big)\Big] + {\bf r},
\end{equation}
where ${\bf c}^{(m)}$ denotes the vector with entries ${c}^{(m)}_i = -c_{1,\ap(x_i)}(a_0 - \ap(x_i))^m$ for $i = 1, \ldots, N-1$, and ${\bf a}\circ{\bf b}$ denotes the Hadamard product of vectors ${\bf a}$ and ${\bf b}$.
The vector ${\bf r}$ is obtained by combining ${\bf r}^{(m)}$ over $m \in \mathbb{N}^0$, and thus its entries are defined exactly as in \eqref{1Dscheme}.
That is, ${\bf r}$ is independent of $m$ and therefore needs to be computed only once.
Both $\mathbf{b}^{(m)}$ and $\mathbf{r}$ arise from the boundary conditions, and if homogeneous boundary conditions are considered, i.e., $g(x)=0$ for $x \in \Omega^c$, these terms vanish.
As observed in Fig.~\ref{fig:demo1}, the contributions of higher-order terms decay rapidly as $m$ increases.
Hence, we truncate the summation in \eqref{fFDM-1D} after $M$ terms, with $M$ chosen sufficiently large.
Further discussion on the selection of $M$ is provided at the end of this subsection.

For each $m \in \mathbb{N}^0$, the matrix $A^{(m)}$ in \eqref{fFDM-1D} has a symmetric Toeplitz structure, unlike the matrix from scheme \eqref{1Dscheme}.
Hence, the memory cost of $A^{(m)}$ is only ${\mathcal O}(N)$.
More importantly, the Toeplitz structure allows the matrix–vector multiplication $A^{(m)}\bu$ to be computed efficiently using FFT-based algorithms at a cost of ${\mathcal O}(2N \log(2N))$.
Consequently, the overall memory cost of the scheme \eqref{fFDM-1D} is ${\mathcal O}(MN)$, and the total computational cost is ${\mathcal O}(MN \log(2N))$, where $M \ll N$.
This is a substantial reduction compared to the cost of \eqref{1Dscheme}.
To highlight this distinction, we refer to \eqref{fFDM-1D} as the { \textit{fast}} finite difference method (fFDM), in contrast to the {\textit{direct}} finite difference method (dFDM) in \eqref{1Dscheme}.
A particularly important advantage of fFDM is its natural compatibility with parallel computing, since the contributions from different $m$ terms can be computed independently.
A detailed comparison between dFDM and fFDM is summarized in Table~\ref{tab:sec2}.

\begin{remark}[Constant-order fractional Laplacian] Our methods provide a finite difference discretization for general $0 < \alpha(x) < 2$. 
If $\ap(x)\equiv\ap$ is constant, the dFDM \eqref{1Dscheme} coincides with the fFDM \eqref{fFDM-1D} with $\ap_0=\ap$ and $M = 1$.
In this case, the scheme reduces to the finite difference discretization for the constant-order fractional Laplacian $(-\Dt)^\fl{\ap}{2}$ as introduced in \cite{Duo2018, Duo2015}.
\end{remark}

\begin{remark}[Choice of $M$]
The choice of $M$ in the scheme \eqref{fFDM-1D} generally depends on the order $\ap(x)$, mesh size $h$, and function $u(x)$.
For each $m \in {\mathbb N}^0$, since $|(\ap_0-\ap(x))^m| < 1$, the truncation is essentially controlled by the factor $|\ln \xi|^m/(m!)$, which attains its largest value at $\xi = h$.
Thus, for a prescribed tolerance $\epsilon$ and a given mesh size $h$, it is sufficient to choose $M$ such that
\begin{equation}\label{critial}
\frac{|\ln h|^M}{M!} < \epsilon.
\end{equation}
This criterion provides a practical guideline for determining $M$ in numerical implementations.
\end{remark}
Fig.~\ref{fig:demo2}(a) illustrates the behavior of $|\ln h|^m/(m!)$ as a function of $m$ for different mesh sizes $h$.
\begin{figure}[htb!]
    \centerline{
    \raisebox{5mm}{(a)}
    \includegraphics[height=5.29cm, width=0.43\linewidth]{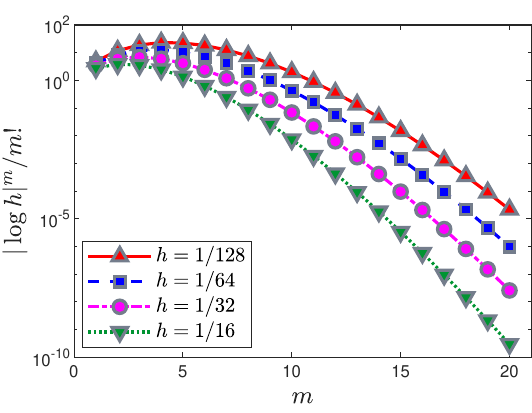}\hspace{4mm}
    \raisebox{5mm}{(b)}
    \includegraphics[height=5.29cm, width=0.43\linewidth]{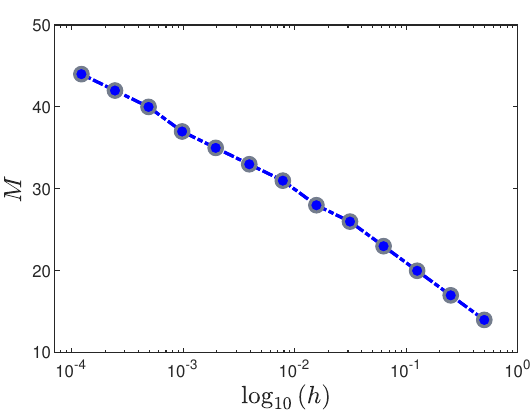}}
    \vspace{-4mm}
\caption{(a) The value of ${|\log{h}|^m}/{m!}$ decays quickly as $m$ increases.
The larger the mesh size $h$, the faster the decay.
(b) Minimal $M$ versus mesh size $h$ to satisfy ${|\log h|^M}/{M!} < \epsilon = 10^{-12}$.}
    \label{fig:demo2}
\end{figure}
It show that, for a fixed mesh size $h$, this term decays rapidly as $m$ increases, which justifies truncating the summation in \eqref{fFDM-1D}.
Fig.~\ref{fig:demo2}(b) presents the minimal $M$ required for each mesh size $h$ when $\epsilon = 10^{-12}$, confirming that the criterion \eqref{critial} offers an effective and practical guideline for determining $M$ in computations.
Moreover, it shows that as the mesh size $h$ decreases, more terms are required in the scheme \eqref{fFDM-1D}.

\section{Generalization to higher dimensions}
\label{section3}
\setcounter{equation}{0}

The spatial dependence of $\ap(\bx)$ significantly increases the difficulty of approximating the variable-order fractional Laplacian in higher ($d > 1$) dimensions.
For $d = 1$, both dFDM and fFDM are applicable, although dFDM incurs substantially higher computational and memory costs.
In contrast, for $d > 1$, the direct discretization of the variable-order fractional Laplacian quickly becomes infeasible as $N$ grows, since it requires both memory and computational costs of $\mathcal{O}(N^2)$; see Table \ref{tab:2d-cpu}.
In this section, we therefore focus on extending the fast finite difference method \eqref{fFDM-1D} to higher dimensions.

Let the $d$-dimensional domain $\Og=(a_1,b_1)\,\times\,\cdots\,\times\,(a_d,b_d)$.
Denote $\bx = \big(x^{(1)}, x^{(2)}, \ldots, x^{(d)}\big)^T$ and $\bxi=\big(\xi^{(1)},\xi^{(2)}, \ldots, \xi^{(d)}\big)^T$ with $\xi^{(i)} = |x^{(i)}-y^{(i)}|$ for $1 \le i \le d$.
Denote the index set
\begin{equation*}
    \aleph_{J} = \big\{(m_1, m_2,\,\ldots,\, m_d)\, |\, 0 \leq m_i \leq J, \ \text{for } 1\leq i\leq d\big\}, 
\end{equation*}  
for $J \in {\mathbb N}$.
Similar to the one-dimensional case, the $d$-dimensional variable-order fractional Laplacian can be formulated into a weighted integral, i.e.,
\begin{equation}\label{newLd}
    (-\Dt)^{\ap(\bx)/2} u(\bx) = -c_{d,\ap(\bx)}\int_{{\mathbb R}_+^d} \psi_{d}(\bx, \bxi)\, |\bxi|^{2-[d+\ap(\bx)]} \md \bxi,
\end{equation}
where ${\mathbb R}_+^d = [0, \infty)^d$, and the functions
\begin{equation}\label{split-psi}
    \psi_{d}(\bx,\bxi) = \fl{1}{|\bxi|^2} \Big(\sum_{\bn\in \aleph_1}u(\bx+(-1)^{\bn}\circ\bxi)-2^du(\bx)\Big). 
\end{equation}
The vector $(-1)^{\bn}=\big((-1)^{n_1}, (-1)^{n_2}, \ldots ,(-1)^{n_d}\big)^T$.
Following the same approach as in \eqref{taylor}, the weight function \eqref{newLd} can be rewritten to separate its dependence on $\bx$ and $\bxi$, yielding
\begin{equation}
    |\bxi|^{2-[d+\ap(\bx)]} = |\bxi|^{2-(d+\ap_0)}\sum_{m=0}^{\infty}\big(\ap_0-\ap(\bx)\big)^m \fl{ (\ln{|\bxi|})^m}{m!}.
\end{equation}
Consequently, the $d$-dimensional variable-order fractional Laplacian in \eqref{newLd} can be expressed as
\begin{equation}\label{vofl-taylor-d}
    (-\Dt)^{\ap(\bx)/2}u(\bx) = -c_{d,\ap(\bx)} \sum_{m=0}^{\infty}\big(\ap_0-\ap(\bx)\big)^m{\mathcal L}_{d,m} u(\bx),
\end{equation}
where the $d$-dimensional nonlocal operator
\begin{equation*}
    {\mathcal L}_{d,m} u(\bx) = \int_{{\mathbb R}_+^d} \psi_{d}(\bx,\bxi)\,w_{d, m}(\bxi) \md \bxi,\qquad\mbox{with}\ \, w_{d, m}(\bxi) = \fl{1}{m!}\,|\bxi|^{2-(d+\ap_0)}(\ln{|\bxi|})^m,
\end{equation*}
for $m \in {\mathbb N}^0$.
Similar to the one-dimensional case, the leading term ${\mathcal L}_{d, 0}$ approximates the constant-order fractional Laplacian, i.e., $-c_{d,\ap_0}\mathcal{L}_{d,0} u(\bx) = (-\Delta)^{\ap_0/2} u(\bx)$.

Now we focus on the discretization of ${\mathcal L}_{d,m}$.
Let $L = \max_{1 \leq i \leq d}|b_i - a_i|$ and set the mesh size $h = L/J$ with integer $J > 0$.
The grid points are then defined as $\bxi_{\bk}=\big(\xi^{(1)}_{k_1},\xi^{(2)}_{k_2}, \,\ldots,\, \xi^{(d)}_{k_d}\big)$ with $\xi^{(i)}_{k_i}=k_{i}h$ for $\bk\in \aleph_J$.
For each ${\bf k} \in\ \aleph_{J-1}$, we define the element ${\mathcal E}_{\bk}=[k_1 h, (k_1+1)h]\,\times\,[k_2 h, (k_2+1)h]\,\times \, \cdots\, \times[k_d h, (k_d+1)h]$.
Then we can write $\calL_{d,m}u(\bx)$ as
\begin{equation}\label{Ldm}
    \calL_{d,m} u(\bx) =  \sum_{\bk\in \aleph_{J-1}}\int_{{\mathcal E}_{\bk}} \psi_{d}(\bx,\bxi) w_{d, m}(\bxi)\md\bxi +\int_{D_d} \psi_{d}(\bx,\bxi) w_{d, m}(\bxi)\md\bxi,\quad \ \mbox{for} \ \, m \in {\mathbb N}^0,
\end{equation}
where $D_{d}=\mathR_+^d\backslash[0,L]^d$.
It is easy to verify that if $\bx \in \Omega$ and $\bxi \in D_d$, then $\bx + (-1)^{\bn} \circ \bxi \in \Omega^c$ for all $\bn \in \aleph_1$.
Hence, by the boundary conditions we have $u(\bx + (-1)^{\bn} \circ \bxi) = g(\bx + (-1)^{\bn} \circ \bxi)$ for $\bn \in \aleph_1$, and consequently the integral
\begin{equation}\label{int-II}
    \int_{D_d} \psi_{d}(\bx,\bxi) w_{d, m}(\bxi)\md\bxi = -2^du(\bx)\int_{D_d}\fl{(\ln{|\bxi|})^m}{(m!)|\bxi|^{d+\ap_0}}\md\bxi + \int_{D_d}\fl{(\ln{|\bxi|})^m}{(m!)|\bxi|^{d+\ap_0}}\sum_{\bn\in\aleph_1} g(\bx+(-1)^{\bn}\circ\bxi)\md\bxi,
\end{equation}
Note that both integrals at the right hand side of \eqref{int-II} are nonsingular and can therefore be evaluated accurately using standard quadrature rules.

For the integrals over element ${\mathcal E}_{\bk}$ (for $\bk\in\aleph_{J-1}$), we can directly apply the $d$-dimensional weighted trapezoidal rule and obtain:
\begin{equation}\label{int-I-k}
    \int_{{\mathcal E}_{\bk}}\psi_{d}(\bx,\bxi) w_{d,m}(\bxi)\md\bxi \approx \fl{1}{2^d}\Big(\sum_{\bn\in \aleph_1} \psi_{d}(\bx,\bxi_{\bk+\bn}) \Big)\int_{{\mathcal E}_{\bk}} w_{d, m}(\bxi)\md\bxi, \quad \ \mbox{for} \ \bk\in\aleph_{J-1}.
\end{equation}
Note that when $|\bk| = |\bn| = 0$, the term $\psi_{d}(\bx,\bxi_{\bzero})$ is defined in a limit sense \cite{Duo2019-FDM}, i.e.,
\begin{equation*}\label{int-I-lim0}
    \psi_d(\bx, \bzero) : = \lim_{\bxi\to\bzero}\psi_{d}(\bx,\bxi) = \Big( \fl{2^d}{d}-1 \Big)\sum_{|\bn|=1}\psi_{d}(\bx,\bxi_{\bn}) - \sum_{\substack{\bn\in \aleph_1, |\bn|\neq 0, 1}}\psi_{d}(\bx,\bxi_{\bn}).
\end{equation*}
Hence, the integral over ${\mathcal E}_{\bzero}$ reduces to
\begin{equation}\label{int-I-0}
    \int_{{\mathcal E}_{\bzero}} \psi_{d}(\bx,\bxi) w_{d, m}(\bxi)\md\bxi \,\approx\, \fl{1}{d}\Big( 
\sum_{|\bn| = 1} \psi_{d}(\bx,{\bxi_{\bf n}}) \Big)\int_{{\mathcal E}_{\bzero}}w_{d, m}(\bxi)\md\bxi,
\end{equation}
Combining \eqref{Ldm}--\eqref{int-I-0} immediately gives the approximation of $\calL_{d, m}$, which we omit here for brevity.

Furthermore, combining \eqref{vofl-taylor-d} with \eqref{Ldm}, we obtain the following approximation of the $d$-dimensional variable-order fractional Laplacian at point $\bx = \bx_{\bf i}$:
\bea\label{discrete-d}
    &&(-\Dt)^{\ap(\bx)/2}_{h}u(\bx) \mid_{\bx = \bx_{\bf i}}=-c_{d,\ap(\bx_{\bf i})}\sum_{m=0}^{M-1}\big(\ap_0-\ap(\bx_{\bf i})\big)^m\bigg[\fl{1}{d}\sum_{\substack{\bn\in\aleph_{1} \\ |\bn|=1}} \psi_{d}(\bx_{\bf i},\bxi_{\bn})\int_{{\mathcal E}_{\bzero}}w_{d,m}(\bxi)\md\bxi \nn\\
    &&\hspace{1.6cm}+ \fl{1}{2^d}\sum_{\substack{\bk\in\aleph_{J-1}\backslash{\{\bzero\}} \\ \bn\in\aleph_1}} \psi_{d}(\bx_{\bf i},\bxi_{\bk+\bn})\int_{{\mathcal E}_{\bk}}w_{d,m}(\bxi)\md\bxi - \bigg(2^d \int_{D_d}\fl{(\ln{|\bxi|})^m}{(m!)|\bxi|^{d+\ap_0}}\md\bxi\bigg) u(\bx_{\bf i}) \bigg] \nn \\
    &&\hspace{1.6cm}-c_{d,\ap(\bx_{\bf i})}\int_{D_d}\fl{1}{|\bxi|^{d+\ap(\bx_{\bf i})}}\sum_{\bn\in\aleph_1} g\big(\bx_{\bf i}+(-1)^{\bn}\circ\bxi\big)\md\bxi, \qquad \mbox{for} \ \ \bx_{\bf i} \in \Og,
\eea 
with a sufficiently large $M > 0$.
The last line in \eqref{discrete-d} is obtained from summing the second term in \eqref{Ldm} over all $m \in \mathbb{N}^0$.
Similar to the one-dimensional case, our method for the $d$-dimensional variable-order fractional Laplacian in \eqref{discrete-d} can be written in matrix–vector form as
\begin{equation}\label{scheme2D}
    (-\Dt)^{\ap(\bx)/2}_{h} \bu = \sum_{m=0}^{M-1}\Big[ {\bf c}^{(m)} \circ \big(A^{(m)}\bu + \bb^{(m)} \big) \Big] + {\bf r},
\end{equation}
where $\bu$, ${\bf c}^{(m)}$, ${\bf b}^{(m)}$, and ${\bf r}$ are block vectors whose structure depends on the spatial dimension $d$.
The matrices $A^{(m)}$ have symmetric multilevel Toeplitz structure, so the matrix–vector products $A^{(m)}\bu$ can be computed efficiently using multidimensional FFTs.
Moreover, since each term in the summation of \eqref{scheme2D} is independent, the computation is naturally parallelizable.
The parallel efficiency will be examined in Section~\ref{section4}.

For the convenience of the reader, we present the detailed forms of the scheme \eqref{scheme2D} for $d=2$ and $d=3$ in Subsections~\ref{section3-1} and \ref{section3-2}, respectively.
\subsection{Two-dimensional (2D) cases}
\label{section3-1}

For the two-dimensional ($d = 2$) case, we let $x = x^{(1)}$ and $y = x^{(2)}$.
Without loss of generality, we set $N_x = J$ and choose $N_y$ as the smallest integer such that $a_2 + N_y h \ge b_2$.
Define the grid points $x_i = a_1 + i h$ and $y_j = a_2 + j h$ for $0 \le i \le N_x$ and $0 \le j \le N_y$.
Let $\bx_{ij} = (x_i, y_j)$, $\bxi_{ij} = (\xi_i^{(1)}, \xi_j^{(2)})$, and $u_{ij} = u(x_i, y_j)$.
Denote the index sets
\begin{align*}
\Omegah &= \{(i,j)\,\big|\, \, 1\leq i\leq N_x-1,\, 1\leq j\leq N_y-1\}, \\
\Omegatwo &= \{(p,q)\,\big|\, \ |p - i| \leq J, \, |q - j| \le J,  \ \mbox{but} \ (p, q)\notin \Og_h \}
\end{align*}
Substituting $d=2$ into \eqref{discrete-d}, using the definition of $\psi_2$ and simplying the calculations, we obtain the finite difference approximation of the 2D variable-order fractional Laplacian as:
\begin{equation}\label{2D-scheme}
    (-\Dt)_h^{\ap_{ij}/2}u_{ij} = -c_{2,\ap_{ij}}\sum_{m=0}^{M-1}\big(\ap_0 -\ap_{ij}\big)^m\bigg[ a_{00}^{(m)}u_{ij} + \sum_{(p,q)\in\Omegah\backslash (i,j)} a_{|i-p||j-q|}^{(m)} u_{pq}+ b^{(m)}_{ij}\bigg] + r_{ij},
\end{equation}
for any $(i,j)\in\Omegah$, where we denote $\ap_{ij} = \ap(x_i, y_j)$, and the terms
\beas
    b^{(m)}_{ij} = \sum_{(p,q)\in\Omegatwo} a^{(m)}_{|i-p||j-q|} u_{pq}, \qquad
    r_{ij} = -c_{2,\ap_{ij}}\sum_{\bn\in\aleph_1} 
    \int_{D_2}\fl{g\big(\bx_{ij}+(-1)^{\bn}\circ\bxi\big)}{|\bxi|^{2+\ap_{ij}}} \md\bxi.
\eeas
The coefficients $a_{pq}^{(m)}$ are given by
\begin{equation*}
a^{(m)}_{pq} = 
\begin{cases}    &\displaystyle\fl{2^{\sigma(p,q)}}{4|\bxi_{pq}|^2} \bigg( \int_{{\mathcal T}_{pq}}\fl{(\ln{|\bxi|})^m}{|\bxi|^{\ap_0}(m!)}\md\bxi + \gamma_{pq}\int_{\calE_{\bf 0}}\fl{(\ln{|\bxi|})^m}{|\bxi|^{\ap_0}(m!)}\md\bxi\bigg), \hspace{1.6cm} \text{ if } p+q>0, \\[10pt]
    &\displaystyle -2\sum_{k=1}^{J}\big(a^{(m)}_{k0}+a^{(m)}_{0k}\big) -4\sum_{k,s=1}^{J }a^{(m)}_{ks} - 4\int_{D_2}\fl{(\ln{|\bxi|})^m}{|\bxi|^{2+\ap_0}(m!)}\md\bxi, \quad \text{ if } p=q=0,
\end{cases}
\end{equation*}
where $\sigma(p,q)$ denotes the number of zeros among the integers $p$ and $q$, ${\mathcal T}_{pq} = \big([(p-1)h, (p+1)h]\times[(q-1)h, (q+1)h]\big)\cap[0,L]^2$, the constant $\gamma_{01}= \gamma_{01} = -\gamma_{11} = 1$, and $\gamma_{pq} = 0$ for other $p, q>0$.

Denote the vector $\bu_{x,j} = (u_{1,j}, u_{2,j}, \ldots, u_{N_x-1,j})$ for $1 \le j \le N_y-1$, and assemble the block vector $\bu = \big(\bu_{x,1}, \bu_{x,2}, \ldots, \bu_{x,N_y-1}\big)^T$.
We define the block vectors $\bap$, ${\bf b}^{(m)}$, and ${\bf r}$ in the same manner.
Then the scheme in \eqref{2D-scheme} can be expressed in the matrix-vector form as in \eqref{scheme2D}, where the matrix is given by
\begin{equation*}
    A^{(m)} = \begin{pmatrix}
    A_{x,0}^{(m)} &A_{x,1}^{(m)} &\cdots &A_{x,N_{y}-3}^{(m)} & A_{x,N_{y}-2}^{(m)} \\
    A_{x,1}^{(m)} &A_{x,0}^{(m)} &A_{x,1}^{(m)} &\cdots &A_{x,N_{y}-3}^{(m)} \\
    \vdots  &\ddots   &\ddots   &\ddots &\vdots \\
    A_{x,N_{y}-3}^{(m)} &\cdots &A_{x,1}^{(m)} & A_{x,0}^{(m)} &A_{x,1}^{(m)} \\
    A_{x,N_{y}-2}^{(m)} & A_{x,N_{y}-3}^{(m)} & \cdots & A_{x,1}^{(m)} & A_{x,0}^{(m)}
    \end{pmatrix}_{N\times N}
\end{equation*}
with $N=(N_{x}-1)(N_{y}-1)$ being the total number of spatial points, and each block $A_{x,j}^{(m)}$ for $0\leq j\leq N_{y}-2$ is a symmetric Toeplitz matrix with its entries defined as
\begin{equation}
    \big(A_{x,j}^{(m)}\big)_{i,k} = \big(A_{x,j}^{(m)}\big)_{i+1,k+1} = a^{(m)}_{|i-k|j}, \qquad 1\leq i,k\leq N_{x}-1, \ \ m \in {\mathbb N}^0.
\end{equation}
In contrast to the one-dimensional case, the coefficients $a^{(m)}_{pq}$, defined through integrals of the weighted function $w_{d,m}(\bxi)$, cannot be evaluated analytically.
Nevertheless, these integrals are nonsingular and can be accurately computed using standard numerical quadrature rules.
The coefficients can be precomputed and reused, especially for time-dependent problems.

\subsection{Three-dimensional (3D) cases}
\label{section3-2}

For the three-dimensional $(d=3)$ case, we let $x=x^{(1)}$, $y=x^{(2)}$, $z=x^{(3)}$.
Similarly, we set $N_x = J$ and choose $N_y$ and $N_z$ as the smallest integers such that $a_2 + N_y h \ge b_2$ and $a_3 + N_z h \ge b_3$.
Define the grid points $x_i = a_1 + i h$, $y_j = a_2 + j h$, and $z_k = a_3 + k h$ for $0 \le i \le N_x$, $0 \le j \le N_y$, and $0 \le k \le N_z$.
Let $\bx_{ijk} = (x_i, y_j, z_k)$, $\bxi_{ijk} = (\xi_i^{(1)}, \, \xi_j^{(2)}, \, \xi_k^{(3)})$, $u_{ijk} = u(x_i, y_j, z_k)$, and $\ap_{ijk} = \ap(x_i, y_j, z_k)$.
Define the index sets
\begin{align*}
    \Omegah &= \big\{(i,j,k)\, \big|\, 1\leq i\leq N_x-1, \,1\leq j\leq N_y-1, \,1\leq k\leq N_z-1 \big\}, \\
    \Omegathree &= \big\{(p,q,s)\,\big|\,   |p-i| \le J,\,  |q-j| \le J, \,  |s-k| \le J, \, \mbox{but} \  (p, q, s)\notin \Og_h\big\}.
\end{align*}
Substituting $d=3$ into \eqref{discrete-d} and using the definition of $\psi_3$ in \eqref{split-psi}, we obtain the fast finite difference approximation of the 3D variable-order fractional Laplacian at point $\bx = \bx_{ijk}$:
\bea\label{3D-scheme}
    &&(-\Dt)_h^{\ap_{ijk}/2}u_{ijk} = -c_{3,\ap_{ijk}}\sum_{m=0}^{M-1}\big(\ap_0 -\ap_{ijk}\big)^m\bigg[ a_{000}^{(m)}u_{ijk} \nonumber\\
    &&\hspace{4.5cm}+ \sum_{(p,q,s)\in\Omegah\backslash(i,j,k)} a_{|i-p||j-q||k-s|}^{(m)} u_{pqs}+ b^{(m)}_{ijk}\bigg] + r_{ijk},\qquad\qquad
\eea
for any $(i, j, k) \in \Og_h$, where
\begin{equation*}
    b^{(m)}_{ijk} = \sum_{(p,q,s)\in\Omegathree} a^{(m)}_{|i-p||j-q||k-s|} u_{pqs}, \quad \
    r_{ijk} = -c_{3,\ap_{ijk}}\sum_{\bn\in\aleph_1}\int_{D_3}\fl{g\big(\bx_{ijk}+(-1)^{\bn}\circ\bxi\big)}{|\bxi|^{3+\ap_{ijk}}} \md\bxi. 
\end{equation*}
For $p + q + s > 0$, the coefficient $a_{pqs}^{(m)}$ is given by
\beas
a^{(m)}_{pqs} = \fl{2^{\sigma(p,q,s)}}{8|\bxi_{pqs}|^2} \bigg( \int_{{\mathcal T}_{pqs}}\fl{(\ln{|\bxi|})^m}{|\bxi|^{1+\ap_0}(m!)}\md\bxi + \gamma_{pqs}\int_{\calE_{\bf 0}}\fl{(\ln{|\bxi|})^m}{|\bxi|^{1+\ap_0}(m!)} \md\bxi\bigg)
\eeas
where $\sigma(p, q, s)$ denotes the number of zeros among the integers $p$, $q$, and $s$, and ${\mathcal T}_{pqs} = \big([(p-1)h,(p+1)h]\times[(q-1)h,(q+1)h]\times[(s-1)h,(s+1)h]\big)\cap[0,L]^3$.
The constant $\gamma_{001} = \gamma_{010} = \gamma_{100} = 5/3$, $\gamma_{101} = \gamma_{110} = \gamma_{011} = \gamma_{111} = -1$, and $\gamma_{pqs}=0$ for other $p, q, s >0$.
While
\begin{equation*}
a^{(m)}_{000} =  -2\sum_{k=1}^{J}\big(a_{k00}+a_{0k0}+a_{00k}\big) -4\sum_{k,s=1}^{J }\big(a_{ks0}+a_{k0s}+a_{0ks}\big)-8\sum_{k,s,n=1}^{J} a_{ksn}-8 \int_{D_3}\fl{(\ln{|\bxi|})^m}{|\bxi|^{3+\ap_0}(m!)}\md\bxi.
\end{equation*}
Denote the block vector $\bu = \left(\bu_{x,y,1},\bu_{x,y,2},\ldots,\bu_{x,y,N_{z}-1}\right)^{T}$ with $\bu_{x,y,k}=\left(\bu_{x,1, k},\bu_{x,2, k},\ldots,\bu_{x,(N_{y}-1), k}\right)$ and $\bu_{x,j, k}=\left(u_{1jk},\, u_{2jk},\, \ldots,\, u_{(N_{x}-1)jk}\right)$.
In the same manner, we define the block vectors $\bap$, ${\bf b}^{(m)}$, and ${\bf r}$.
Then the scheme in \eqref{3D-scheme} can be written in the matrix-vector form in \eqref{scheme2D} with
\begin{equation*}
    A^{(m)} = \begin{pmatrix}
    A_{x,y,0}^{(m)} &A_{x,y,1}^{(m)} &\cdots &A_{x,y,N_{z}-3}^{(m)} & A_{x,y,N_{z}-2}^{(m)} \\
    A_{x,y,1}^{(m)} &A_{x,y,0}^{(m)} &A_{x,y,1}^{(m)} &\cdots &A_{x,y,N_{z}-3}^{(m)} \\
    \vdots  &\ddots   &\ddots   &\ddots &\vdots \\
    A_{x,y,N_{z}-3}^{(m)} &\cdots &A_{x,y,1}^{(m)} & A_{x,y,0}^{(m)} &A_{x,y,1}^{(m)} \\
    A_{x,y,N_{z}-2}^{(m)} & A_{x,y,N_{z}-3}^{(m)} & \cdots & A_{x,y,1}^{(m)} & A_{x,y,0}^{(m)}
    \end{pmatrix}_{N\times N}
\end{equation*}
for $m \in {\mathbb N}^0$, where $N = (N_x-1)(N_y-1)(N_z -1)$.
For $0 \le k \le N_{z}-2$, the block matrix
\begin{equation*}
    A_{x,y,k}^{(m)} = \begin{pmatrix}
    A_{x,0,k}^{(m)} & A_{x,1, k}^{(m)} &\cdots &A_{x,(N_{y}-3), k}^{(m)} & A_{x,(N_{y}-2), k}^{(m)} \\
    A_{x,1,k}^{(m)} &A_{x,0,k}^{(m)} &A_{x,1,k}^{(m)} &\cdots &A_{x,(N_{y}-3),k}^{(m)} \\
    \vdots  &\ddots   &\ddots   &\ddots &\vdots \\
    A_{x,(N_{y}-3),k}^{(m)} &\cdots &A_{x,1,k}^{(m)} & A_{x,0,k}^{(m)} &A_{x,1,k}^{(m)} \\
    A_{x,(N_{y}-2),k}^{(m)} & A_{x,(N_{y}-3),k}^{(m)} & \cdots & A_{x,1,k}^{(m)} & A_{x,0,k}^{(m)}
    \end{pmatrix},
\end{equation*}
with each block $A_{x,j,k}^{(m)}$ is a symmetric Toeplitz matrix, and its entry is defined as
\begin{equation*}
    \left(A_{x,j,k}^{(m)}\right)_{is} = \left(A_{x,j,k}^{(m)}\right)_{i+1,s+1} = a_{|i-s|jk}^{(m)}, \quad 1\leq i,s \leq N_x-1.
\end{equation*}
Similar to the 2D case, the matrix-vector product $A^{(m)}{\bu}$ can be efficiently computed via the 3D FFTs, and the computational cost is $\mathcal{O}\big(2N\log(2N)\big)$, and the storage cost is $\mathcal{O}(N)$.

\medskip
In summary, our $d$-dimensional fast finite difference scheme in \eqref{scheme2D} is naturally parallelizable, as the $M$ terms can be computed independently.
The detailed algorithm is given in Algorithm~\ref{alg2}, and the total computational and memory costs are summarized in Table~\ref{tab:sec2}.
\begin{algorithm}[htb!]
\caption{Parallel implementation of fast FDM in \eqref{scheme2D} }\label{alg2}
\KwIn{Grid points $\mathbf{x}$, grid  function $\mathbf{u}$, number $M$}
\KwOut{Approximate grid function $(-\Dt)_h^{\ap(\bx)/2}{\bf u}$}
\vspace{1mm}
1: Compute vector ${\bf c}_{d,\bap}$ and ${\bf r}$\\
2: \textbf{parallel for} $m = 0$ to $M-1$ \textbf{do} \\
3:\qquad Compute vector ${\bf c}^{(m)}$, ${\bf b}^{(m)}$, and $\{a_{\bf k}^{(m)}\}$ for ${\bf k} \in \aleph_{J}$ \\
4:\qquad Compute $A^{(m)}\mathbf{u}$ via $d$-dimensional FFTs \\
5:\qquad Compute ${\bf c}^{(m)}\circ [A^{(m)}{\bf u} + {\bf b}^{(m)}]$\\
6: \textbf{end for} \\
7: \Return approximation $(-\Dt)_h^{\ap(\bx)/2}{\bf u}$ as defined in \eqref{scheme2D} 
\end{algorithm}

\begin{remark}[Higher-dimensional accuracy]\label{thm2}
Theorem~\ref{thm} establishes the error estimate in one dimension.
The higher-dimensional analysis can be developed along similar lines but is more involved and is omitted here for brevity.
The numerical experiments in Section~\ref{section4-1} show second-order convergence for smooth functions in higher dimensions.
\end{remark} 

\section{Numerical experiments}
\label{section4}
\setcounter{equation}{0}

In this section, we perform numerical experiments to evaluate the accuracy and efficiency of our fast finite difference method and to study both the steady-state and dynamic behavior of diffusion problems in heterogeneous media.
Despite the growing recognition of variable-order nonlocal models for capturing diffusion in complex media \cite{Lenzi2016}, numerical studies still remain limited, mainly due to the significant memory and computational challenges posed by the spatial varying $\ap(\bx)$.
Our goal is to validate our methods through extensive experiments and gain deeper insights into the underlying operators so as to advance this emerging field.

\subsection{Accuracy verification}
\label{section4-1}

We test the numerical accuracy of our method in approximating the variable-order fractional Laplacian for functions $u$ with different regularity and verify our analytical results.
Here, we denote the error function $e_\Dt^h(\bx)$ in approximating the fractional Laplacian as in \eqref{error}.
In addition, we examine the effect of $M$-term truncation in the fast finite difference method.

\begin{example}\label{example4-1}
Consider the one-dimensional compactly supported function
\begin{equation}\label{ex1-u}
    u(x)=\bigl[(a^2-x^2)_+\bigr]^q,
    \qquad x\in\mathbb{R},
\end{equation}
where constants $a,q>0$.
The exact solution of $(-\Dt)^{\ap(x)/2}u(x)$ on $(-a, a)$ is given by \cite{Wu2024}:
\beas
(-\Dt)^{\ap(x)/2}u(x) = \fl{2^{\ap(x)}\Gamma{\big(q+1\big)}\Gamma{\big((1+\ap(x))/2\big)}a^{2q-\ap(x)}}{\sqrt{\pi}\Gamma{\big(q+1-\ap(x)/2\big)}} {}_{2}F_{1}\Big( \fl{1+\ap(x)}{2},-q+\fl{\ap(x)}{2};\fl{1}{2};\fl{x^2}{a^2}\Big),
\eeas 
where ${}_{2}F_{1}$ represents the Gauss hypergeometric function.
The function $u$ has two potential low-regularity points at $\hat {x}_{1, 2}=\pm a$, with local regularity exponents $\sigma_1=\sigma_2=q$.
\end{example}
In Fig.~\ref{fig:thm1}, we study the numerical errors $\|e_{\Delta}^h\|_{\ell^\infty(\Omega)}$ for $a=2$ and $\alpha(x)=0.3-0.1x$ on two different domains.
On $\Omega=(-2,2)$, the low-regularity points lie on the boundary, and the smallest regularity gap is $q-\alpha(-2)=q-0.5$, which determines the convergence rate when $0.5 < q < 2.5$.
On $\Omega=(-1,1)$, the low-regularity points lie outside the physical domain, and the convergence order is $q + 1$ if $q < 1$.

\begin{figure}[htb!]
    \centerline{
    \raisebox{5mm}{(a)}\includegraphics[height=5.19cm, width=0.43\linewidth]{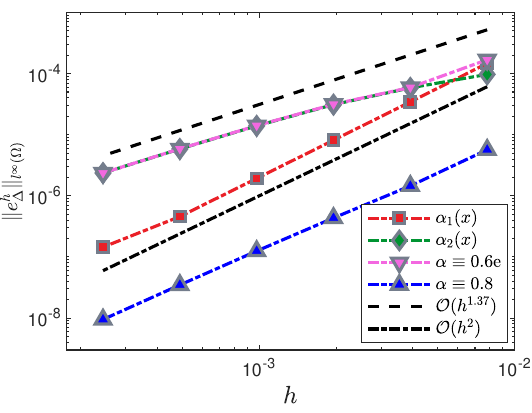}
    \raisebox{5mm}{(b)}\includegraphics[height=5.19cm, width=0.43\linewidth]{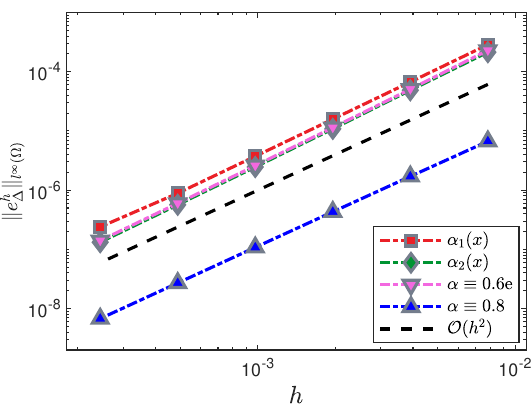}}
    \vspace{-4mm}
\caption{Numerical errors in approximating $(-\Dt)^{\ap(x)/2}u(x)$, where $\alpha_1(x) = 1.8 - |x|$, $\alpha_2(x) = 0.6e^{-x}$, and $u$ is defined in \eqref{ex1-u} with $a = 1$, and $q=3$ (a) and $q=4$ (b).
Reference order lines indicate the convergence rate predicted by \eqref{thm1-ineq2}.}
    \label{Fig2-0-0}
\end{figure}
In the following, we fix $\Omega=(-1,1)$ and take $u$ as defined in \eqref{ex1-u} with $a=1$.
The two low-regularity points $\hat x_1=1$ and $\hat x_2=-1$ lie on the boundary, with $\sigma_1=\sigma_2=q$.
Fig.~\ref{Fig2-0-0} compares the numerical errors for the variable-order and constant-order cases.
For both $\ap_1(x)$ and $\ap_2(x)$, the smallest regularity gap occurs at $\hat x_2=-1$, and thus gap $q-\ap_j(-1)$ governs the convergence rate, subject to the second-order threshold.
To compare, we choose the corresponding constant order as $\alpha=\ap_j(-1)$, $j=1,2$, so that both cases have the same smallest gap.
It shows that for sufficiently smooth functions (e.g., Fig.~\ref{Fig2-0-0} (b) for case of $q \ge 4$), our method has second-order accuracy for all $\ap$, though the errors in the variable-order cases may be larger than those in the constant-order cases.
These results further confirm the theoretical results in Theorem~\ref{thm} and demonstrate that the accuracy is determined by the interplay between the solution regularity and the fractional order, rather than by the maximum value of $\alpha(x)$.

The above results are obtained using the dFDM in \eqref{1Dscheme}.
Our studies show that the fFDM in \eqref{fFDM-1D}, with a sufficiently large $M$, can achieve the same accuracy.
Next, we study the effect of $M$-term truncation in the fFDM by comparing its results with those of dFDM.
Fig. \ref{Figure-Error} compares the numerical errors of these two approaches for various $h$ and $M$.
The fFDM can be viewed as a truncated approximation of dFDM, and its truncation error depends on $M$.
\begin{figure}[ht!]
\centerline{
    \raisebox{5mm}{(a)}
    \includegraphics[height=5.19cm, width=0.42\linewidth]{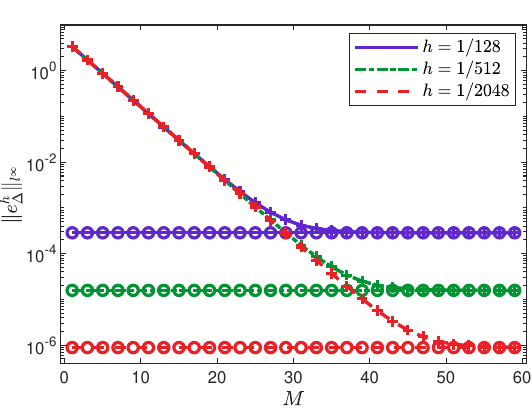}\hspace{4mm}
    \raisebox{5mm}{(b)}
    \includegraphics[height=5.19cm, width=0.42\linewidth]{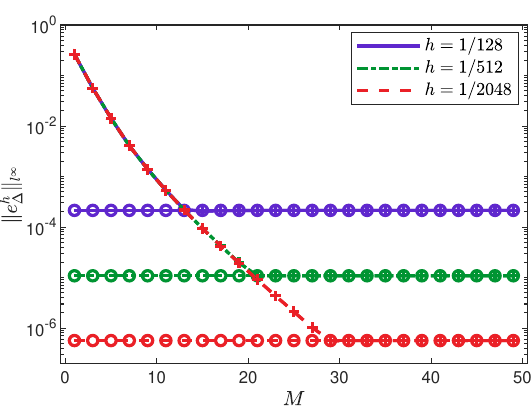}}
    \vspace{-3mm}
\caption{Comparison of numerical errors of fFDM (symbol `$+$') and dFDM (symbol `$\circ$') in approximating $(-\Dt)^{\ap(x)/2}u$, where $u$ is defined in \eqref{ex1-u} with $q = 4$.
(a) $\ap(x) = 1.8-|x|$; (b) $\ap(x) = 0.6e^{-x}$.}\label{Figure-Error}
\end{figure}
We find that numerical errors of the fFDM decrease exponentially as $M$ increases, eventually converging to those of dFDM.
This is because, for small $M$, the truncation error in fFDM dominates the overall numerical error.
However, when $M$ is large enough, the discretization errors of numerical methods become dominant, and the truncation errors can be neglected.
The threshold value of $M$ depends on the mesh size $h$ and $\ap(x)$; see also Fig. \ref{fig:demo2}.
These observations demonstrate the effectiveness of our fFDM in accurately computing the variable-order fractional Laplacian.

\medskip
\begin{example}
To further our study, we approximate $(-\Delta)^{\alpha(x)/2}u(x)$ on domain $\Omega=(-1,1)$ with
\begin{equation}\label{ex2-u-new}
    u(x) = \big[(0.25^2-x^2)_+\big]^{q_1}
    + \big[(1-x^2)_+\big]^{q_2},
    \qquad x\in\mathbb{R}, \quad q_1,q_2>0,
\end{equation}
Its exact solution can be found similarly to that in Example~\ref{example4-1}.
The function $u$ has four potential low-regularity points: $\hat x_{1,2}=\pm0.25$ in the interior and $\hat x_{3,4}=\pm1$ on the boundary, with local regularity exponents $\sigma_{1,2} = q_1$ and $\sigma_{3, 4} = q_2$, respectively.
\end{example}

Compared with Example~\ref{example4-1}, the convergence rate in this case is governed by the smallest of the regularity gaps $q_1-\ap(\pm0.25)$ and $q_2-\ap(\pm1)$, subject to the second-order threshold.
Table~\ref{tab:example4-2} reports the numerical errors and convergence rates for several choices of $(q_1,q_2)$ with $\ap(x)=0.6+0.4x$.
The last row gives the algebraic orders predicted by Theorem~\ref{thm}.\footnote{For $(q_1,q_2)=(2.8,3)$, $\mu=\min\{2.8-0.7,3-1\}=2$, and the theorem gives the bound $O(h^2[1+\log(1/h)])$.
The reported algebraic order $2.000$ does not include this logarithmic factor; the observed second-order convergence is compatible with the bound.}
\begin{table}[htb!]
    \centering
    \begin{tabular}{@{}lcccccccc@{}} 
    \toprule
       &\multicolumn{2}{c}{$(q_1, q_2) = (1, 1.5)$} &\multicolumn{2}{c}{$(q_1, q_2) = (1.8, 1.5)$} & \multicolumn{2}{c}{$(q_1, q_2) = (1.8, 3)$} & \multicolumn{2}{c}{$(q_1, q_2) = (2.8, 3)$}\\ 
    \cmidrule(ll){2-3}
    \cmidrule(ll){4-5}
    \cmidrule(ll){6-7}
    \cmidrule(ll){8-9}
     $h$ & error & rate & error & rate & error & rate & error & rate \\
    \midrule
    1/128  &7.384e-2 &-- -- &2.522e-3 &-- -- &4.417e-5 &-- -- &6.334e-6 &-- --\\
    1/256  &5.998e-2 &0.300 &1.795e-3 &0.491 &1.962e-5 &1.170 &1.589e-6 &1.995 \\
    1/512  &4.872e-2 &0.300 &1.274e-3 &0.495 &8.931e-6 &1.136 &3.977e-7 &1.998 \\
    1/1024 &3.957e-2 &0.300 &9.026e-4 &0.497 &4.116e-6 &1.118 &9.948e-8 &1.999 \\
    1/2048 &3.214e-2 &0.300 &6.390e-4 &0.498 &1.909e-6 &1.109 &2.488e-8 &2.000 \\
    \midrule
    Rate in \eqref{thm1-ineq2} &&0.300&&0.500&&1.100& &2.000 \\
    \bottomrule
    \end{tabular}
\caption{Numerical errors $\|e^h_{\Dt}\|_{l^\infty(\Og)}$ and convergence rate of our method in comparison to the predicted rate in \eqref{thm1-ineq2} , where $u$ is defined in \eqref{ex2-u-new} and $\ap(x) = 0.6+0.4x$.}
    \label{tab:example4-2}
\end{table}
Since $\alpha(x)$ increases monotonically, its maximum value is attained at $x=1$.
Table~\ref{tab:example4-2} shows that the convergence rate is determined by the competition between the values of $q_1-\alpha(0.25)$ and $q_2-\alpha(1)$.
For instance, when $(q_1,q_2)=(1,1.5)$, we have $q_1-\alpha(0.25)<q_2-\alpha(1)$, and hence the convergence rate is governed by $q_1-\alpha(0.25)$.
In contrast, for $(q_1,q_2)=(1.8,1.5)$, $q_1-\alpha(0.25)>q_2-\alpha(1)$, so the convergence rate is determined by $q_2-\alpha(1)$.
For $(q_1,q_2)=(2.8,3)$, the numerical results exhibit second-order convergence.

Our numerical results are in excellent agreement with the theoretical predictions of Theorem~\ref{thm}.
They further confirm that the convergence rate is governed by the interplay between the local regularity of the solution and the fractional order at the low-regularity points, rather than by the maximum value of $\ap(x)$.

\medskip
\begin{example}\label{2d-op}
Consider the $d$-dimensional function
\bea\label{ex2-u}
u(\bx) = e^{-10|\bx|^2}, \qquad \mbox{for} \ \ \bx \in {\mathbb R}^d.
\eea
Here, we approximate $(-\Dt)^{\ap(\bx)/2}u(\bx)$ on the domain $\Og = [-2, 2]^2$, and its exact solution is \cite{Wu2024}
\bea
(-\Dt)^{\ap(\bx)/2}u(\bx) = 2^{\ap(\bx)} \sqrt{10^{\ap(\bx)}} \Gamma\Big(1+\fl{\ap(\bx)}{2}\Big){}_{1}F_{1}\Big(1+\fl{\ap(\bx)}{2}; 1;-10|\bx|^2\Big), \qquad \mbox{for} \ \ \bx \in {\mathbb R}^2,
\eea
with ${}_{1}F_{1}$ representing the confluent hypergeometric function.
\end{example}

Table \ref{tab:op-2d-1} presents the numerical errors for $\alpha(\bx) = 1.8 - |\bx|/4$, computed using the fFDM with $M = 40$, and compares them to the constant-order cases with $\alpha_{\max} \equiv \sup_{\bx\in\Og} \alpha(\bx)$ or $\alpha_{\min} \equiv \inf_{\bx\in\Og} \alpha(\bx)$.
Unless otherwise specified, we set $h = h_x = h_y$ in all two-dimensional experiments.
Notably, in two-dimensional cases, the dFDM becomes impractical for finer meshes (e.g., $h = 1/64$ or smaller) due to excessive memory requirements; see Section \ref{section4-2} for further discussion.
\begin{table}[ht]
    \centering
    \begin{tabular}{@{}lcccccc@{}} 
    \toprule
     & \multicolumn{2}{c}{$\ap(\bx)=1.8-\fl{|\bx|}{4}$}  &\multicolumn{2}{c}{$\ap\equiv 1.8$} &\multicolumn{2}{c}{$\ap\equiv1.8-\fl{\sqrt{2}}{2}$} \\ 
    \cmidrule(ll){2-3}
    \cmidrule(ll){4-5}
    \cmidrule(ll){6-7}
    $h$ & error & rate & error & rate & error & rate  \\
    \midrule
    1/8     &1.5354    & -- -- &1.5354 & -- -- &0.1445 & -- --    \\
    1/16    &0.3825    &2.005 &0.3825 &2.005 &3.502e-2 &2.044  \\
    1/32    &9.344e-2 &2.033 &9.344e-2 &2.033 &8.595e-3 &2.027  \\
    1/64    &2.276e-2 &2.037 &2.276e-2 &2.037 &2.125e-3 &2.016 \\
    1/128   &5.554e-3 &2.035 &5.554e-3 &2.035 &5.280e-4 &2.009 \\
    1/256   &1.359e-3 &2.032 &1.358e-3 &2.032 &1.315e-4 &2.005  \\
    \bottomrule
    \end{tabular}
\caption{Numerical errors $\|e^h_{\Dt}\|_{l^\infty(\Og)}$ and convergence rate of our method, where $u$ is defined in \eqref{ex2-u} and the fast FDM is used with $M = 40$.}
    \label{tab:op-2d-1}
\end{table}
In this case, $u\in C^\infty(\mathbb R^2)$, and the numerical results in Table~\ref{tab:op-2d-1} exhibit second-order convergence.
Furthermore, comparison with the constant-order cases shows that the error for $\alpha(\bx)$ is nearly identical to that for $\alpha_{\max} = 1.8$.
Consistent with the one-dimensional case, when the solution is sufficiently smooth, the maximum error tends to occur near the point where $\alpha(\mathbf{x})$ attains its maximum.
This behavior is further confirmed by the error plots in Fig.~\ref{Fig4-2-1-0}(b).
\begin{figure}[htb!]
\centerline{
\raisebox{5mm}{(a)}\includegraphics[height=5.29cm, width=6.8cm]{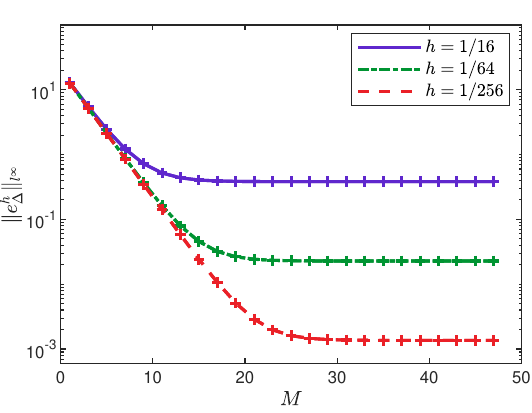}\hspace{5mm}
\raisebox{5mm}{(b)}\includegraphics[height=5.29cm, width=6.8cm]{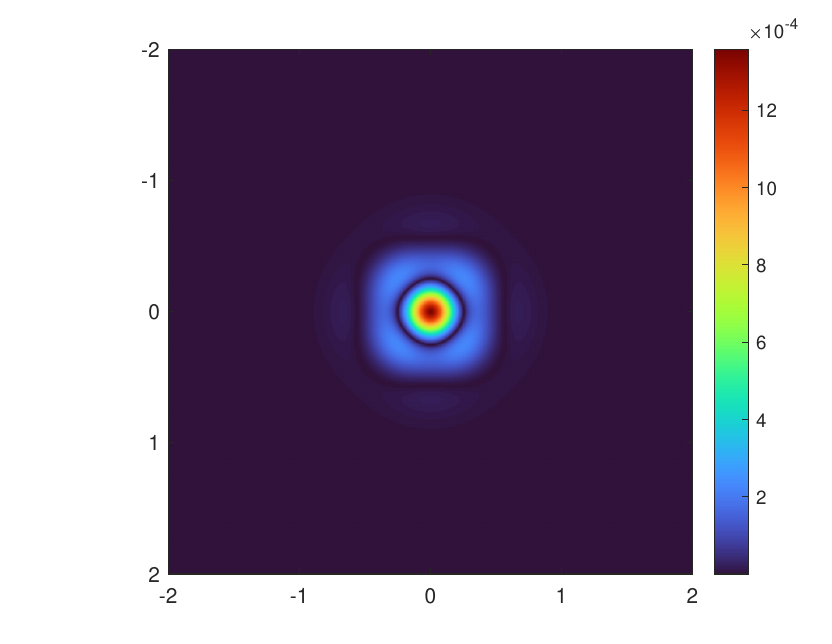}}
\vspace{-3mm}
\caption{(a) Effect of $M$-term truncation in the fast FDM, and (b) pointwise error $|e_\Delta^h(x, y)|$ with $h = 1/256$, in approximating $(-\Delta)^{\alpha(\mathbf{x})/2}$ with $\alpha(\mathbf{x}) = 1.8 - |\mathbf{x}|/4$.}\label{Fig4-2-1-0}
\end{figure}
Fig.~\ref{Fig4-2-1-0}(a) shows the numerical errors of fFDM for various values of $h$ and $M$ to illustrate the effect of $M$-term truncation.
Consistent with our observations in Fig.~\ref{Figure-Error}, numerical errors decrease exponentially as $M$ increases.

\subsection{Efficiency verification}
\label{section4-2}

In the following, we further evaluate the efficiency of our fFDM, compared with dFDM, for approximating the variable-order fractional Laplacian $(-\Delta)^{\ap(\bx)/2} u(\mathbf{x})$.
To this end, we consider the function in \eqref{ex2-u} and approximate $(-\Delta)^{\alpha(\mathbf{x})/2} u(\mathbf{x})$ on domain $\Omega = (-2, 2)^d$.
We compare the computational times of the fast and direct FDMs in both one- and two-dimensional cases.
All experiments are conducted with MATLAB 2022a on a laptop equipped with an Intel(R) Core(TM) i7-12700H CPU and 32 GB of RAM.

Table \ref{Tab4-2-1} compares the CPU time of dFDM and fFDM for one-dimensional cases with various numbers of points $N$, where $M = 65$ is used in fFDM.
The computing time demonstrated here include both the time in assembling matrix and performing matrix-vector multiplications.
The computing time of fFDM depends on the value of $M$.
Our studies in Example 4.1 and 4.2 show that $M = 65$ is large enough for fFDM to obtain the same errors as dFDM.
\begin{table}[ht]
    \centering
    \begin{tabular}{rrrrrr} 
    \toprule
               & $N=1023$ & $N=2047$ & $N=4095$ & $N=8191$ & $N=16383$\\ \hline
    dFDM                  &0.2068 &0.8122 &3.5556 &15.5447 &58.9526\\
    fFDM          &0.1281 &0.1590 &0.1883 &0.3848  &0.7165 \\ 
    \bottomrule
    \end{tabular}
\caption{Computing time comparison (in second) for computing the 1D variable-order fractional Laplacian $(-\Dt)^{\ap(x)/2}u(x)$ with $\ap(x)=1+0.8\cos(\pi x)$, where $M = 65$ is used in fFDM.}
    \label{Tab4-2-1}
\end{table}
The results in Table~\ref{Tab4-2-1} show that fFDM is significantly faster than dFDM.
If the number of points increases from $N_1$ to $N_2$, the computing time of dFDM grows by a factor of $(N_2/N_1)^2$, whereas for fFDM it grows approximately as $(N_2/N_1)$.
The advantages of fFDM become more pronounced, as the number of points increases.
For example, when $N = 16383$, the computing time used by the fFDM is only $1.22\%$ of that of dFDM.
We further observe that the assembly of the coefficient matrix dominates the overall computational cost.

Table \ref{tab:2d-cpu} further compares the performance of dFDM and fFDM in two-dimensional cases, where $M = 40$ is used in fFDM.
Note that the computing time in assembling matrix is ${\mathcal O}(N^2)$ for dFDM, while ${\mathcal O}(MN)$ for fFDM.
The time of matrix-vector multiplications for dFDM and fFDM are ${\mathcal O}(N^2)$ and ${\mathcal O}(2MN\log(2N))$, respectively.
Table~\ref{tab:2d-cpu} shows that for small values of $N$, the fFDM can be slower than dFDM because of the overhead associated with $M$.
However, fFDM becomes significantly more efficient once $N \geq 31^2$.
Furthermore, the memory demands of dFDM grow so quickly with $N$ that fine meshes become impractical.
In particular, when $N \geq 255^2$, only fFDM is feasible.
The memory challenges become even more severe in 3D cases.
\begin{table}[htb!]
    \centering
    \begin{tabular}{rrrrrr} \toprule
          &$N=15^2$  &$N=31^2$   &$N=63^2$   &$N=127^2$ & $N=255^2$  \\ \hline
    dFDM  &16        &217        &3373       &54327 & N.A.        \\ 
    fFDM  &39        &146        &572        &2241  & 9062     \\ 
    parallel fFDM  &22 &99       &324        &1270  & 5066     \\
    \bottomrule
    \end{tabular}
\caption{Computing time (in second) for computing the 2D variable-order fractional Laplacian $(-\Dt)^{\ap(\bx)/2}u(\bx)$, where $M=40$ is used in fFDM, and the parallel fFDM is executed on two CPUs.
For $N = 255^2$, `N.A.' indicates that dFDM is impractical due to memory exhaustion.
}
    \label{tab:2d-cpu}
\end{table}

Moreover, fFDM is readily parallelizable, leading to significant improvements in efficiency.
Table~\ref{tab:2d-cpu} reports the performance of its parallel implementation using two CPUs, verifying the parallelizability of fFDM.
The results show that the computing time of fFDM can be further reduced when parallel computing is used.
This feature makes our approach particularly attractive for real-world high-dimensional applications.
In addition, fFDM avoids the memory bottleneck associated with storing the full matrix, substantially reducing the memory costs.
This advantage becomes even more significant in three dimensions, where the benefits of the fFDM are magnified, particularly for time-dependent problems that require matrix–vector multiplications at every time step.

\subsection{Diffusion in heterogeneous media}
\label{section4-3}

In this section, we assess the performance of our method in solving diffusion problems and explore the role of the variable-order fractional Laplacian in modeling diffusion in heterogeneous media.
To this end, we study both steady and time-dependent diffusion equations.
In particular, we first assess the accuracy of our method by testing it on the benchmark problems.

\begin{example}[Steady diffusion problem] Consider the one-dimensional problem
\begin{equation}\label{Poisson}
(-\Delta)^{\alpha(x)/2}u(x) = f(x), \qquad \mbox{for} \ \ x \in \Omega,
\end{equation}
with homogeneous Dirichlet boundary conditions on $\Omega^c$.
We take $\Omega = (-1,1)$ and solve \eqref{Poisson} for two representative cases.
Both have been widely used as benchmarks for assessing numerical methods for the constant-order fractional Laplacian \cite{Duo2018, Duo2019-TFL, Zhou2024}.
This model also finds broad applications, for instance, in the study of dark matter through modifications of Newtonian gravity \cite{Giusti2020, Giusti2020on}.
\end{example}

\noindent{\bf Benchmark I. }  Choose function $f$ in (\ref{Poisson}) such that the exact solution is given by $u(x) =  x\big[(1-x^2)_+\big]^s$ for $s > 0$.  
Fig.~\ref{fig:ex4-3-1} displays the numerical errors and corresponding convergence rates for different choices of $\alpha(x)$ and $s$.
For easy comparison, each line is accompanied by a reference line of ${\mathcal O}(h^s)$.
\begin{figure}[htb!]
    \centerline{
    \raisebox{5mm}{(a)} \includegraphics[width=0.42\linewidth, height=5.2cm]{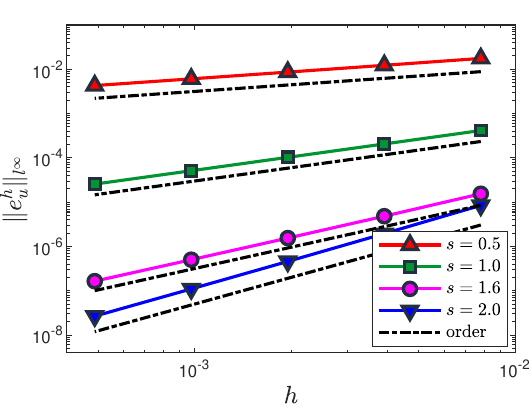}\hspace{3mm}
    \raisebox{5mm}{(b)}
    \includegraphics[width=0.42\linewidth, height=5.2cm]{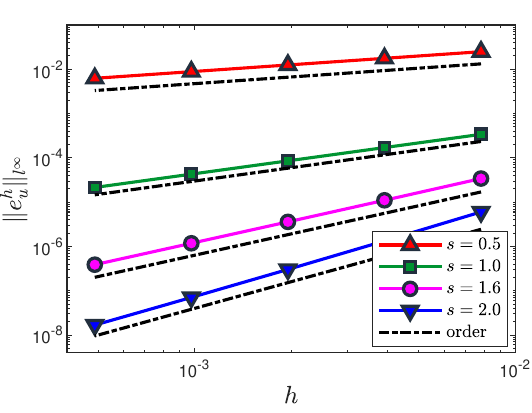}}
    \vspace{-3mm}
\caption{Numerical errors in solving the steady diffusion problem \eqref{Poisson}, where the exact solution is $u(x) = x\big[(1-x^2)_+\big]^s$; (a) $\ap_1(x) = 1-0.5\sin(\pi x)$, and (b) $\ap_2(x) = 1+0.4{\rm tanh}(4x)$.}\label{fig:ex4-3-1}
\end{figure}
The results show that smoother solutions (corresponding to larger values of $s$) yield smaller numerical errors.
Our method achieves an accuracy rate of ${\mathcal O}(h^{\min\{s,2\}})$, consistent with the convergence behavior reported in \cite{Duo2018} for constant-order cases.
Numerical experiments further indicate that the convergence rate is determined by the regularity of the solution $u$.

Fig.~\ref{fig:ex4-3-2} further compares the pointwise errors for different choices of $\alpha(x)$ and $s$.
The results show that for less regular solutions (i.e., smaller $s$), the maximum errors occur near the boundary, and their magnitude is strongly influenced by the values of $\alpha(x)$ near the boundary.
In this case, larger values of $\alpha(x)$ at the boundary lead to larger numerical errors.
When the solution is sufficiently smooth, however, the maximum errors occur in the interior of the domain.
Due to the spatial variability of $\alpha(x)$, the errors are distributed heterogeneously across $(-1,1)$.
\begin{figure}[htb!]
    \centerline{
    \raisebox{5mm}{(a)}
    \includegraphics[width=0.4\linewidth, height=5.1cm]{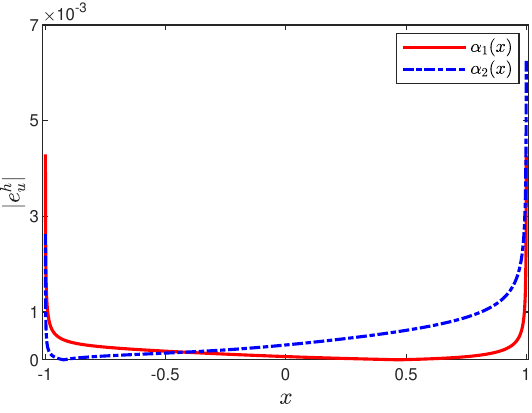}\hspace{4mm}
    \raisebox{5mm}{(b)}
    \includegraphics[width=0.4\linewidth, height=5.1cm]{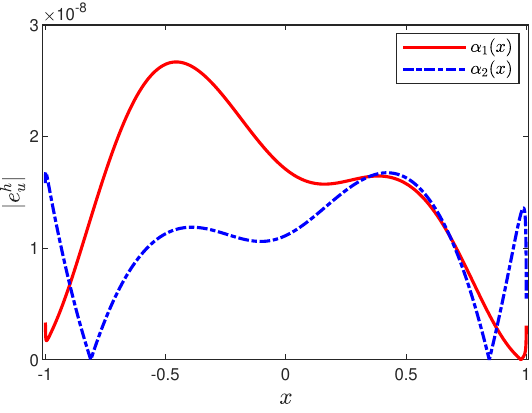}}
    \vspace{-3mm}
\caption{Numerical errors of our method in solving the steady diffusion problem \eqref{Poisson}, where the exact solution is $u_{\rm ex}(x) = x\big[(1-x^2)_+\big]^s$, $\ap_1(x) = 1-0.5\sin(\pi x)$, and $\ap_2(x) = 1+0.4\,{\rm tanh}(4x)$.
(a): $s = 0.5$; (b): $s = 2$.}
   \label{fig:ex4-3-2}
\end{figure}

\medskip
\noindent{\bf Benchmark II. } Choose the function $f(x) \equiv 1$ in (\ref{Poisson}). 
This benchmark problem has been widely used in the literature (see \cite{Duo2018, Zhou2024} and references therein) for testing the accuracy of numerical methods of the constant-order fractional Laplacian.
When $\alpha$ is constant, the exact solution of (\ref{Poisson}) is given by $u(x) = (1-x^2)^{\fl{\alpha}{2}}/\Gamma(1+\alpha)$ for $x \in (-1, 1)$.
To the best of our knowledge, no exact solution has been established for the variable-order case.

Fig.~\ref{fig:poisson-1} presents the numerical solution $u_h$ for two choices of $\alpha(x)$, along with the function $v(x) = (1-x^2)^{\alpha(x)/2}/\Gamma(1+\alpha(x))$ for comparison.
The results show that $u(x)$ and $v(x)$ differ substantially, suggesting that the solution of variable-order problem is not a straightforward extension of that of the constant-order case.
\begin{figure}[htb!]
    \centerline{
    \raisebox{5mm}{(a)} \includegraphics[width=0.39\linewidth, height=5.1cm]{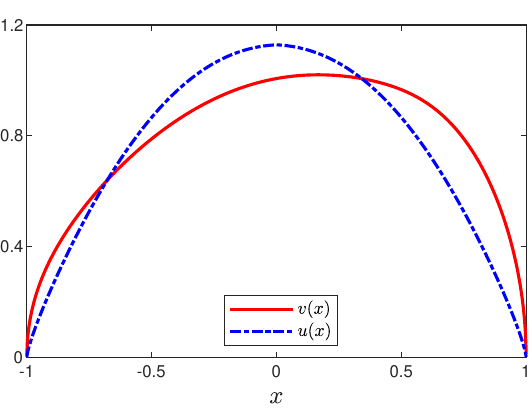}
    \hspace{4mm}
    \raisebox{5mm}{(b)} \includegraphics[width=0.39\linewidth, height=5.1cm]{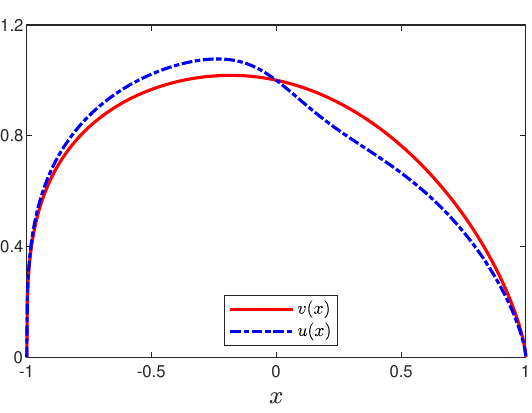}}
    \vspace{-3mm}
\caption{Illustration of the numerical solution $u_h$ of Benchmark II and its comparison with $v(x)=\fl{1}{\Gamma(1+\ap(x))}(1-x^2)^{{\ap(x)}/{2}}$.
(a) $\ap(x) = 1-0.5\sin(\pi x)$; (b) $\ap(x) = 1+0.4\,{\rm tanh}(4x)$.}
    \label{fig:poisson-1}
\end{figure}
We also observe that when $\alpha(x)$ is symmetric with respect to the center of the domain, the solution preserves this symmetry.
Moreover, the smoothness of the solution at the boundary is directly influenced by the boundary values of $\alpha(x)$: larger values of $\alpha(x)$ at the boundary yield smoother behavior, as further illustrated in Fig.~\ref{fig:poisson-2}.

Table~\ref{tab:possion-1} presents the numerical errors in the solution for different choices of $\alpha(x)$.
Since the exact solution is unknown, the errors are estimated by comparing $u_h$ with a reference solution computed on a finer mesh of $h/2$, i.e., $e_u^h = u_h - u_{h/2}$.
To facilitate comparison, Table~\ref{tab:possion-2} reports the numerical errors for the constant-order fractional Poisson problems, with each constant order chosen as $\alpha_i=\min\{\ap_i(1),\ap_i(-1)\}$, the smaller order at the two low-regularity points.
\begin{table}[htb!]
    \centering
    \begin{tabular}{@{}lcccccccc@{}} 
    \toprule
       &\multicolumn{2}{c}{\small $\ap_1(x) = 0.6e^{-x}$} &\multicolumn{2}{c}{\small $\ap_2(x) = 0.5+|x|$} & \multicolumn{2}{c}{\small $\ap_3(x)=1-\fl{1}{2}\sin(\pi x)$} & \multicolumn{2}{c}{\small $\ap_4(x)=1+\fl{2}{5}\tanh(4x)$}\\ 
    \cmidrule(ll){2-3}
    \cmidrule(ll){4-5}
    \cmidrule(ll){6-7}
    \cmidrule(ll){8-9}
     $h$ & error & rate & error & rate & error & rate & error & rate \\
    \midrule
    1/128  &2.885e-2 &-- -- &1.995e-3 &-- -- &9.710e-3 & -- -- &2.081e-2 & -- --\\
    1/256  &2.661e-2 &0.117 &1.195e-3 &0.739 &6.938e-3 &0.485  &1.690e-2 &0.301\\
    1/512  &2.459e-2 &0.114 &7.141e-4 &0.743 &4.936e-3 &0.491  &1.372e-2 &0.301\\
    1/1024 &2.275e-2 &0.112 &4.258e-4 &0.746 &3.502e-3 &0.495  &1.114e-2 &0.300\\
    1/2048 &2.107e-2 &0.111 &2.536e-4 &0.748 &2.481e-3 &0.497  &9.048e-3 &0.300\\
    \bottomrule
    \end{tabular}
\caption{Numerical errors $\|e^h_u\|_{l^\infty(\Og)}$ and convergence rate of our method in solving (\ref{Poisson}) with $f = 1$.}
    \label{tab:possion-1}
\end{table}
By comparing Tables~\ref{tab:possion-1} and \ref{tab:possion-2}, we observe that numerical errors decrease as $h$ reduces.
For the constant-order cases (see Table~\ref{tab:possion-2}), the convergence rate is ${\mathcal O}(h^{\alpha/2})$, which is consistent with the findings in \cite{Duo2018, Duo2019-TFL}.
In contrast, for the variable-order cases, the accuracy is determined by the minimum value of $\alpha(x)$ at the boundary $x = \pm 1$.
More precisely, the convergence rate is ${\mathcal O}(h^{\min\{\ap_i(1),\ap_i(-1)\}/2})$.
\begin{table}[htb!]
    \centering
    \begin{tabular}{@{}lcccccccc@{}} 
    \toprule
       &\multicolumn{2}{c}{$\ap_1 \equiv 0.6/e$} &\multicolumn{2}{c}{$\ap_2 \equiv 1.5$} & \multicolumn{2}{c}{$\ap_3 \equiv 1.0$} & \multicolumn{2}{c}{$\ap_4 \equiv 0.6$}\\ 
    \cmidrule(ll){2-3}
    \cmidrule(ll){4-5}
    \cmidrule(ll){6-7}
    \cmidrule(ll){8-9}
     $h$ & error & rate & error & rate & error & rate & error & rate \\
    \midrule
    1/128  &2.911e-2 &-- -- &1.496e-3 &-- -- &8.411e-3 &-- -- &2.171e-2 & -- --\\
    1/256  &2.695e-2 &0.111 &8.887e-4 &0.752 &5.941e-3 &0.501 &1.762e-2 &0.301 \\
    1/512  &2.496e-2 &0.111 &5.281e-4 &0.751 &4.200e-3 &0.501 &1.431e-2 &0.301 \\
    1/1024 &2.312e-2 &0.110 &3.140e-4 &0.750 &2.969e-3 &0.500 &1.162e-2 &0.300 \\
    1/2048 &2.142e-2 &0.110 &1.867e-4 &0.750 &2.099e-3 &0.500 &9.435e-3 &0.300 \\
    \bottomrule
    \end{tabular}
\caption{Numerical errors $\|e^h_u\|_{l^\infty(\Og)}$ and convergence rate of our method in solving (\ref{Poisson}) with $f = 1$, where the constant order $\ap_i = \min\{\ap_i(1), \ap_i(-1)\}$ with $\ap_i(x)$ given in Table \ref{tab:possion-1}.
Numerical erros show a convergence rate of order ${\mathcal O}(h^{{\alpha}/2})$.}
    \label{tab:possion-2}
\end{table}

Fig.~\ref{fig:poisson-2} provides a direct comparison of the numerical solutions and associated errors for the constant- and variable-order cases.
Fig.~\ref{fig:poisson-2}(a)--(b) show the sharp boundary layers form in the solution; the smaller the value of $\alpha(x)$ at the boundary, the sharper the layers are.
More precisely, as $x \to \pm 1$, the solution shows the asymptotic behavior $u(x) \sim (1-|x|^2)^{\ap(x)/2}$, which characterizes the boundary singularity induced by the variable-order fractional Laplacian.
\begin{figure}[htb!]
    \centerline{
    \raisebox{5mm}{(a)}
    \includegraphics[width=0.42\linewidth, height = 5.06cm]{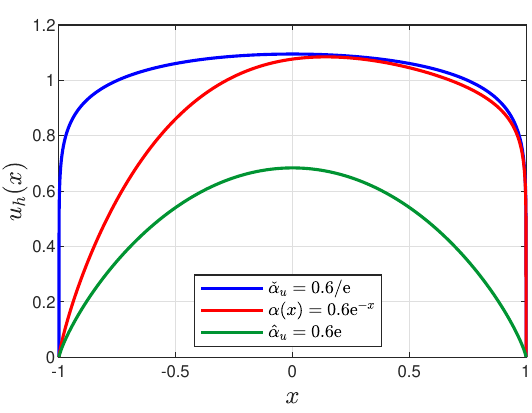} 
    \hspace{4mm}
    \raisebox{5mm}{(b)}
    \includegraphics[width=0.42\linewidth, height = 5.06cm]{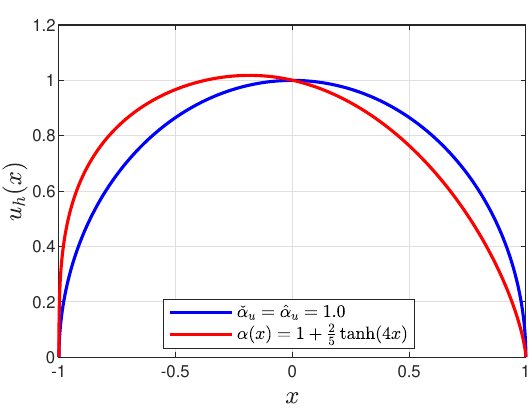}}
    \centerline{
   \raisebox{5mm}{(c)}
   \includegraphics[width=0.42\linewidth, height = 5.06cm]{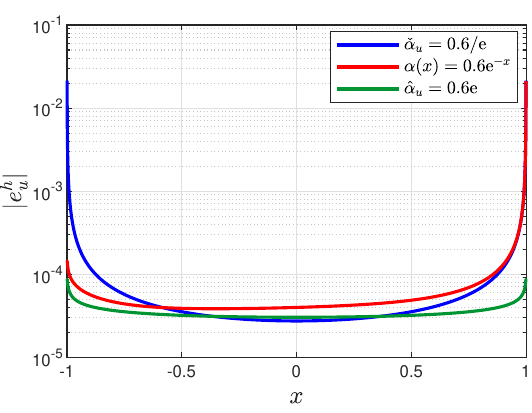}
   \hspace{4mm}
    \raisebox{5mm}{(d)}
    \includegraphics[width=0.42\linewidth, height = 5.06cm]{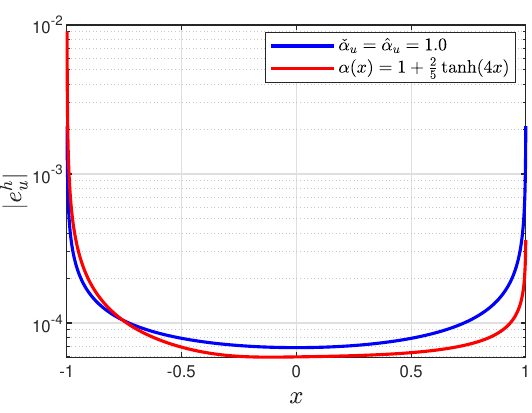}}
    \vspace{-3mm}
\caption{Comparison of numerical solution (a \& b) and their errors (c \& d) in solving the problem \eqref{Poisson} with $f = 1$.}\label{fig:poisson-2}
\end{figure}
Fig.~\ref{fig:poisson-2}(c)--(d), together with our extensive numerical experiments, show that the maximum errors occur near the boundary, while the error behavior differs substantially between the constant- and variable-order cases.
These observations highlight the intrinsic complexity of the variable-order fractional Laplacian.
We note that the regularity of solutions to \eqref{Poisson} with $f(x) \equiv 1$ and constant-order $\alpha$ has been thoroughly studied in \cite{Ros-Oton2014}, whereas corresponding results for the variable-order fractional Laplacian are still absent from the literature.
Establishing such results therefore represents an important open problem for future research.

\begin{example}[Transient diffusion in heterogeneous media]
Consider the time-dependent diffusion problem of the form \cite{Wu2024}:
\bea\label{diffusion}
\p_t u(\bx, t) + \kappa(-\Dt)^{\ap(\bx)/2} u(\bx, t) = 0, &\quad  &\mbox{for} \ \  \bx \in \Og, \ \  t \in (0, T], \\
\label{IC}
u(\bx, 0) = 10 e^{-16|\bx|^2}, &\quad&\mbox{for} \ \ \bx \in \overline{\Og},
\eea
where the diffusion constant $\kappa > 0$, and the homogeneous Dirichlet boundary conditions are applied on ${\mathbb R}^2\backslash\Og$.
The parameter $\ap(\bx)$ can be considered as a material-dependent quantity, for example, in cardiac tissue with localized regions of damage \cite{Farquhar2018}.
\end{example}

Throughout the study, we set the computational domain $\Og = (-3, 3)^2$, the final time $T = 3$, and the constant $\kappa = 0.1$.
Consider the fractional power $\ap(x, y)$ of the form:
\bea \label{alpha}
\ap_i(x,y) = a_i + 0.947
\sum_{x_0, \, y_0 \in C_i}
\tanh\!\Big(\frac{(x - x_0)^2 + (y - y_0)^2 - 0.12}{\varepsilon}\Big),\qquad i = 1, 2,
\eea
where we choose $\varepsilon = 0.04$, $a_1 = 0.953$, $a_2 = -13.258$, $C_1 = \{0.6\}$, and $C_2 = \{\pm 0.6, \pm 1.8\}$.
Fig. \ref{fig:ex4-4-1} (left) displays the radially symmetric initial condition at $t = 0$.
Fig. \ref{fig:ex4-4-1} (middle and right) illustrate the functions $\ap_i(x, y)$ for $i =1, 2$, representing a composite medium, consisting of two different materials, where one material is embedded within the other~\cite{Zhang2011}.
\begin{figure}[htb!]
    \centering
    \includegraphics[width=0.20\linewidth]{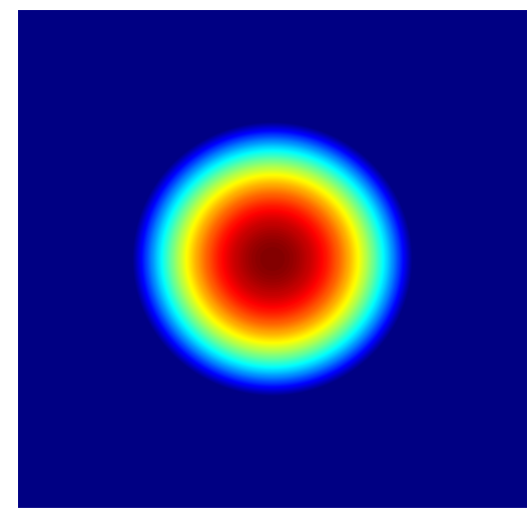}\hspace{3mm}
    \includegraphics[width=0.20\linewidth]{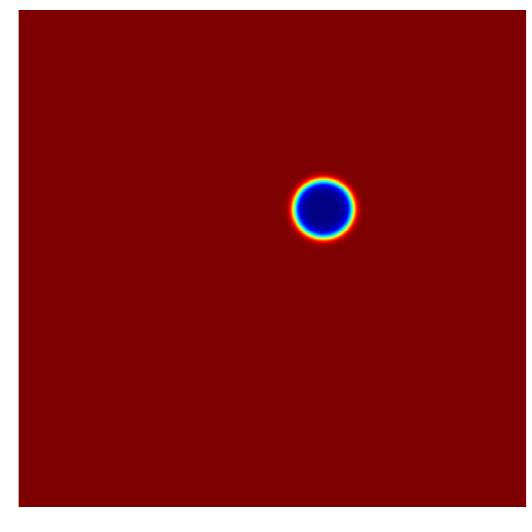}\hspace{3mm}
    \includegraphics[width=0.20\linewidth]{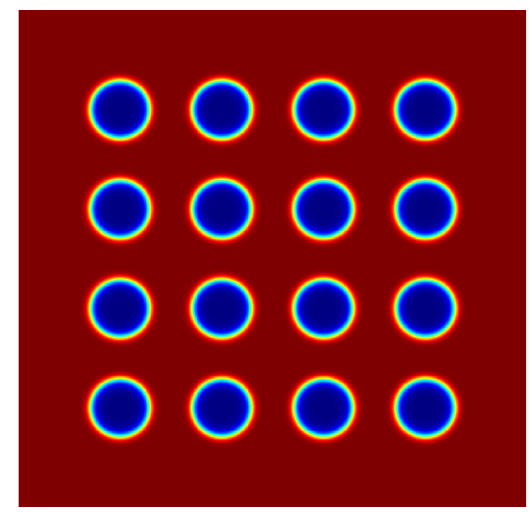}
    \vspace{-3mm}
\caption{Illustration of the initial condition $u(\bx, 0)$ (left), and the heterogeneous power $\ap_1(x, y)$ (middle) and $\ap_{2}(x, y)$ (right) defined in \eqref{alpha}, where the dark-red region indicates $\ap_i(x, y) = 1.9$ and the dark-blue region indidates $\ap_i(x, y) = 0.0107$.}
    \label{fig:ex4-4-1}
\end{figure}
In our simulations, the diffusion problem \eqref{diffusion}--\eqref{IC} is discretized in space by our proposed fFDM and in time with the classical fourth-order Runge--Kutta (RK4) method.
We use the mesh size $h = 1/128$ and the time step $\tau = 10^{-4}$.
A refinement study confirms that the selected mesh size and time step are fine enough, and additional reductions produce no observable change in the results.

\begin{figure}[thb!]
    \centering
    \includegraphics[width=0.22\linewidth, height = 3.3cm]
{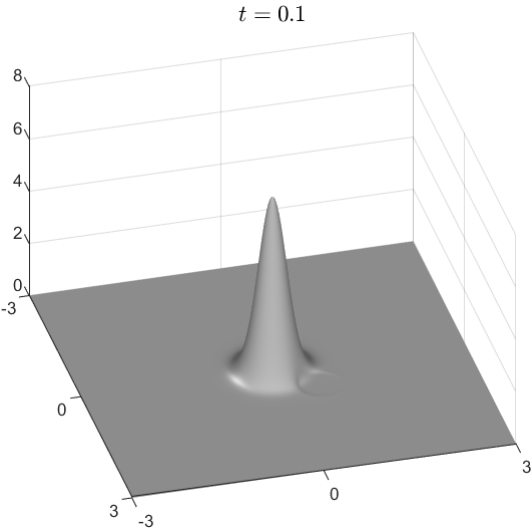}
    \includegraphics[width=0.22\linewidth, height = 3.3cm]
{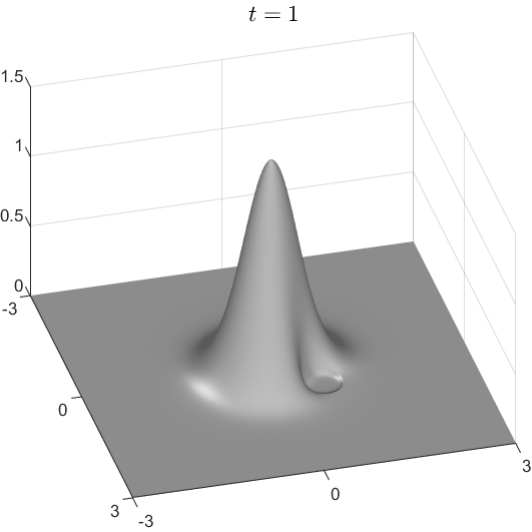}
    \includegraphics[width=0.22\linewidth, height = 3.3cm]
{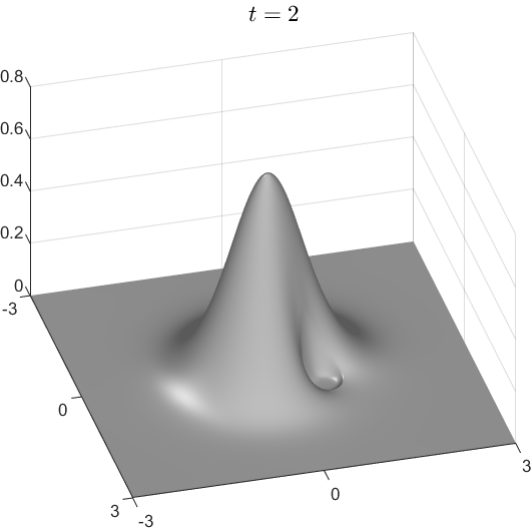}
    \includegraphics[width=0.22\linewidth, height = 3.3cm]
{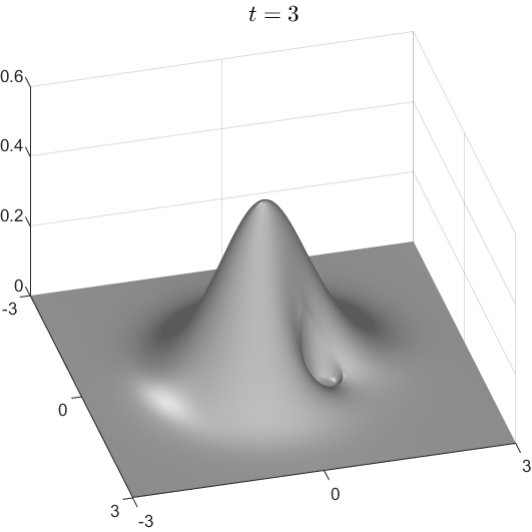}

    \includegraphics[width=0.2\linewidth, height=3.3cm]
{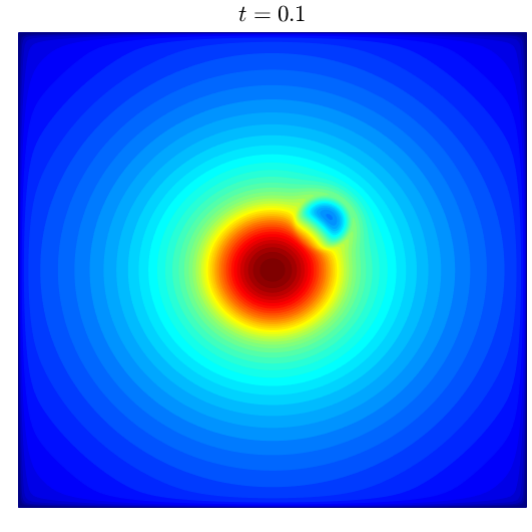}
    \hspace{0.5mm}
    \includegraphics[width=0.2\linewidth, height=3.3cm]{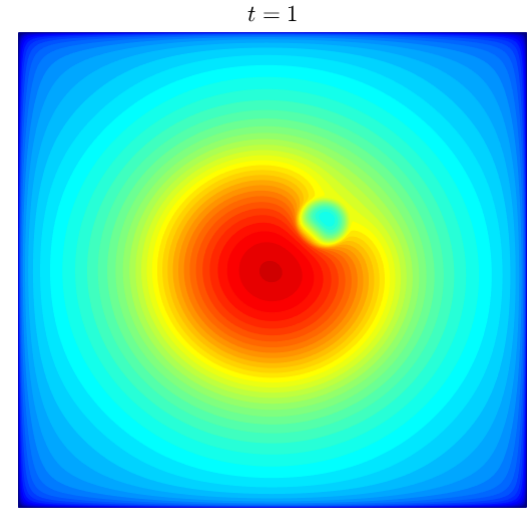} 
    \hspace{0.5mm}
    \includegraphics[width=0.2\linewidth, height=3.3cm]
{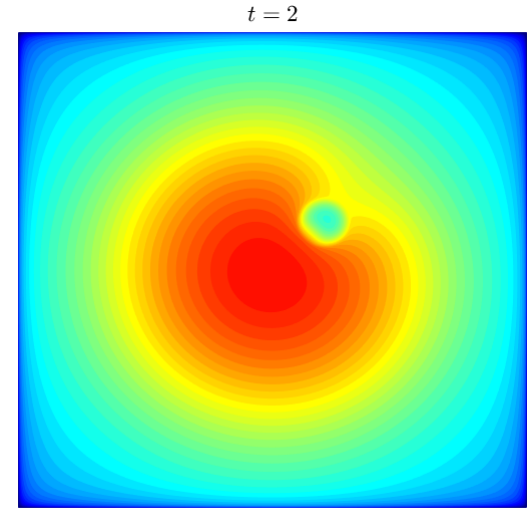}
    \hspace{0.5mm}
    \includegraphics[width=0.22\linewidth, height=3.4cm]
{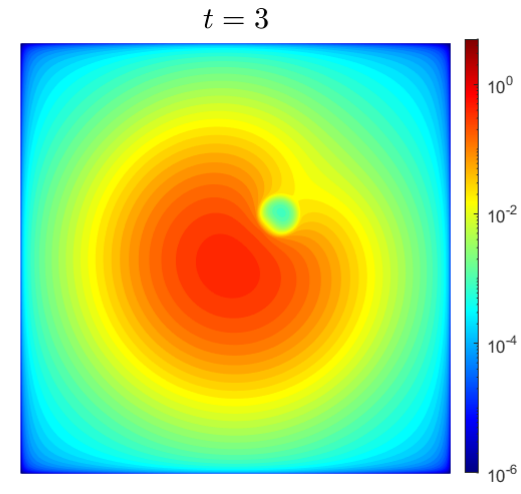}
    \vspace{-3mm}
\caption{Dynamics of the solution to the diffusion problem \eqref{diffusion}--\eqref{IC} in a {\it heterogeneous} medium characterized by $\ap_1(\bx)$ in \eqref{alpha}.
Top row: surface plots; bottom row: corresponding top views.}
    \label{fig:ex4-4-2}
\end{figure}
\begin{figure}[thb!]
    \centering
    \includegraphics[width=0.22\linewidth, height = 3.3cm]
{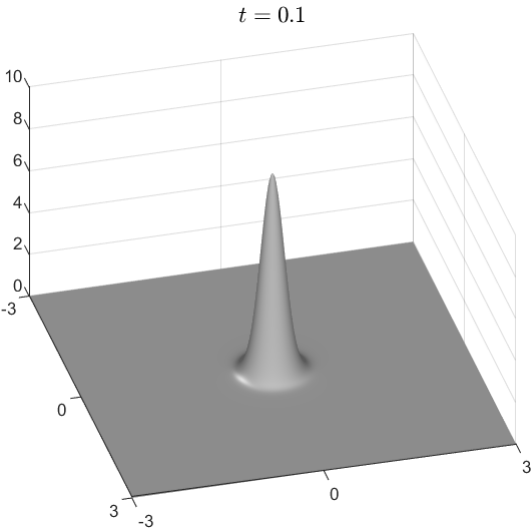}
    \includegraphics[width=0.22\linewidth, height = 3.3cm]
{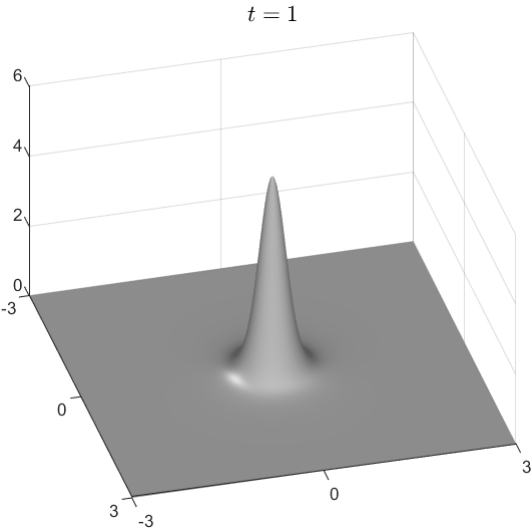}
    \includegraphics[width=0.22\linewidth, height = 3.3cm]
{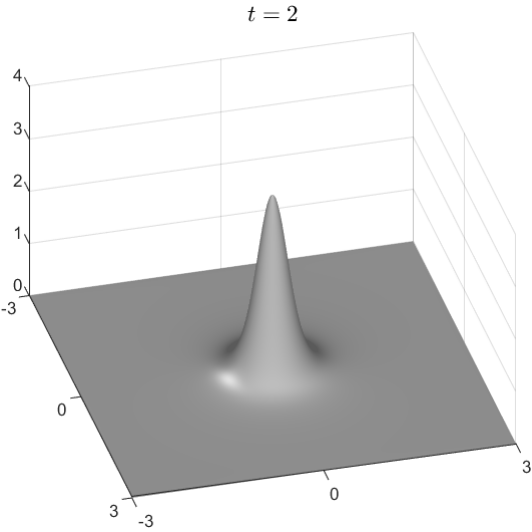}
    \includegraphics[width=0.22\linewidth, height = 3.3cm]
{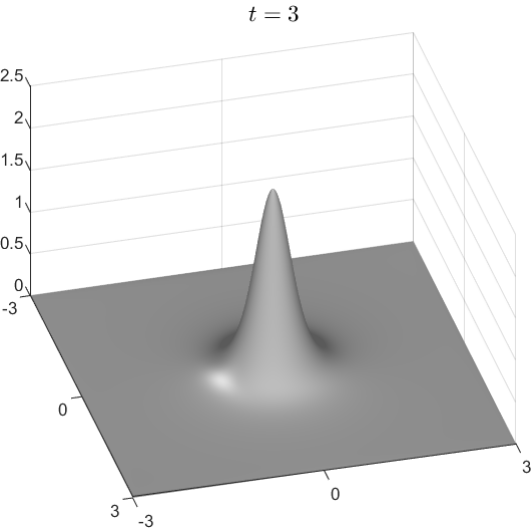}

    \includegraphics[width=0.2\linewidth, height=3.3cm]
{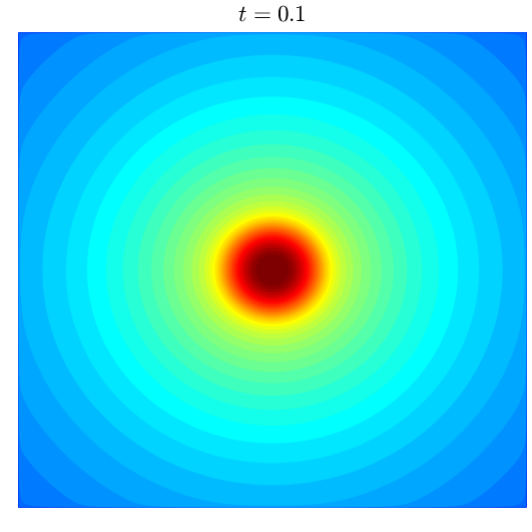}
    \hspace{0.5mm}
    \includegraphics[width=0.2\linewidth, height=3.3cm]{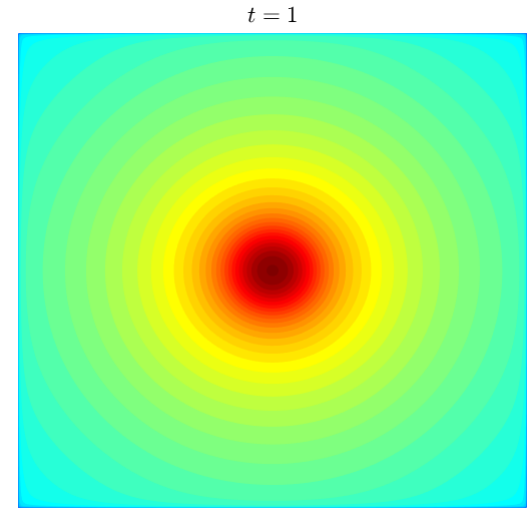} 
    \hspace{0.5mm}
    \includegraphics[width=0.2\linewidth, height=3.3cm]
{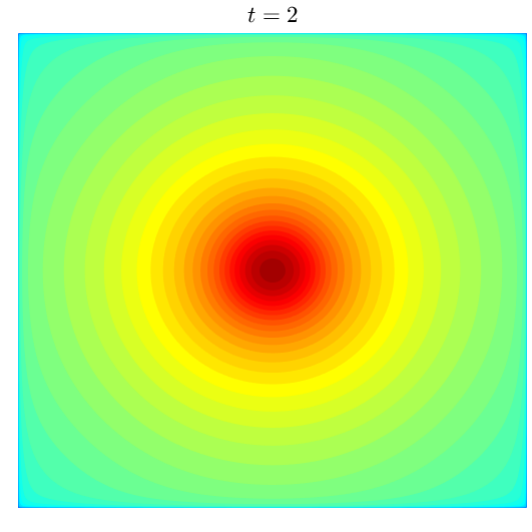}
    \hspace{0.5mm}
    \includegraphics[width=0.22\linewidth, height=3.4cm]
{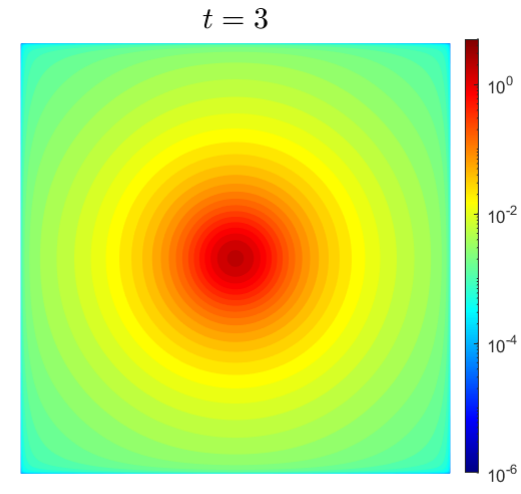}
    \vspace{-3mm}
\caption{Dynamics of the solution to the diffusion problem \eqref{diffusion}--\eqref{IC} in a {\it homogeneous} medium with $\ap \equiv 0.95$.
Top row: surface plots; bottom row: corresponding top views.}
    \label{fig:ex4-4-3}
\end{figure}
Fig. \ref{fig:ex4-4-2} shows the solution dynamics of \eqref{diffusion}--\eqref{IC} with $\ap_1(\bx)$.
For comparison, Fig.~\ref{fig:ex4-4-3} shows the corresponding results of the constant-order diffusion equation with $\ap(\bx)\equiv \ap_0$.
It is clear that the spatial dependence of $\ap(\bx)$ has a significant impact on the diffusion behavior.
The solution initially concentrates at the domain center $(0,0)$, and then propagates outward toward the boundary.
Fig.~\ref{fig:ex4-4-2} shows that once the propagating front reaches the slow medium, heterogeneous effects become evident -- the solution grows more slowly within this region, while it propagates around it and continues to expand through the surrounding fast-diffusion area.
In contrast, the diffusion in Fig.~\ref{fig:ex4-4-3} is isotropic, distinguishing from those in Fig.~\ref{fig:ex4-4-2}.
In the literature, one common approach is to approximate the variable-order fractional Laplacian by an averaged constant-order operator due to the lack of effective numerical methods for the variable-order case; see, e.g., \cite{Zhu2014}.
However, our observations in Figs.~\ref{fig:ex4-4-2}--\ref{fig:ex4-4-3} indicate that a constant-order model fails to capture the heterogeneous phenomena, and our fFDM provides an effective approach for studying diffusion in heterogeneous media.

Fig.~\ref{fig:ex4-4-4} further presents the results for $\ap_2(\bx)$, where the slow-diffusion material is uniformly embedded in a lattice pattern.
Compared to Fig.~\ref{fig:ex4-4-2}, the heterogeneity effects are more significant, and thus the solution dynamics become more complicated.
In this case, numerical simulation becomes considerably more challenging if different diffusion models are used for each material, as discussed in \cite{Delia2021}.
These results further suggest that the variable-order nonlocal operator provides an effective framework for capturing anomalous transport in heterogeneous media.
\begin{figure}[thb!]
    \centering
    \includegraphics[width=0.22\linewidth, height = 3.3cm]
{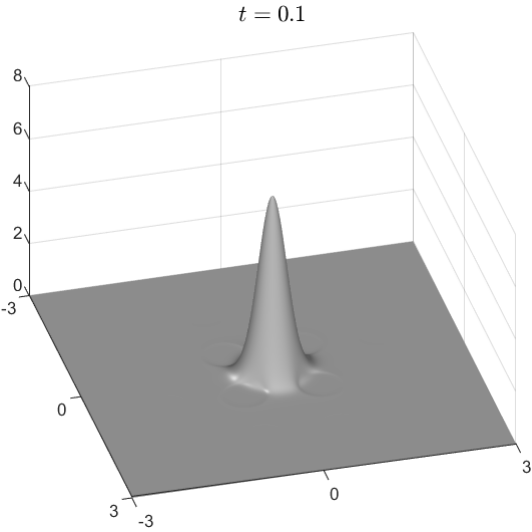}
    \includegraphics[width=0.22\linewidth, height = 3.3cm]
{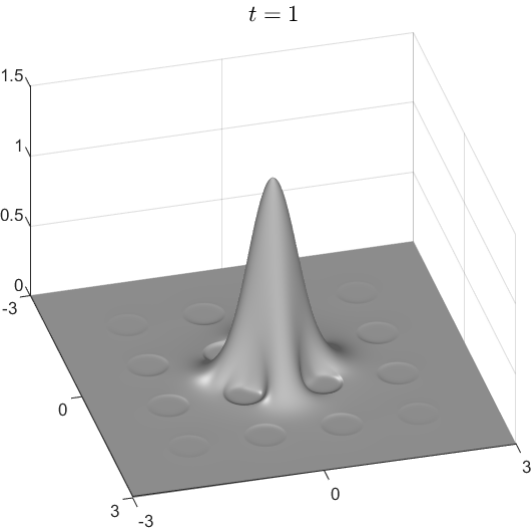}
    \includegraphics[width=0.22\linewidth, height = 3.3cm]
{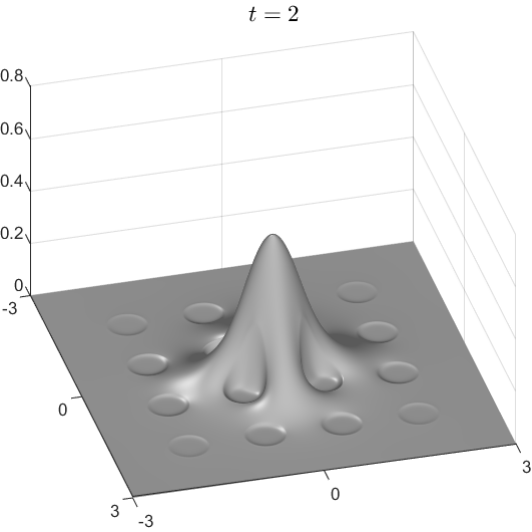}
    \includegraphics[width=0.22\linewidth, height = 3.3cm]
{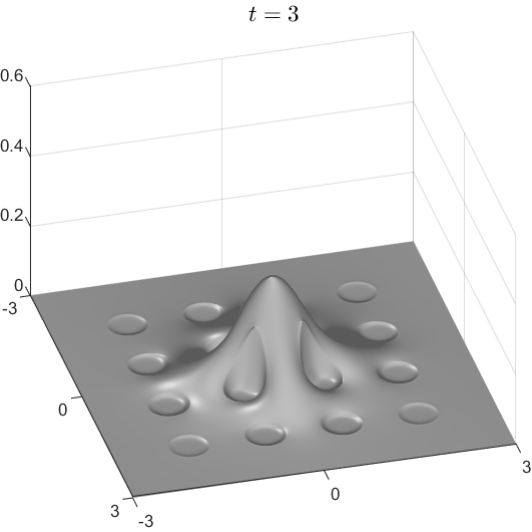}

    \includegraphics[width=0.2\linewidth, height=3.3cm]
{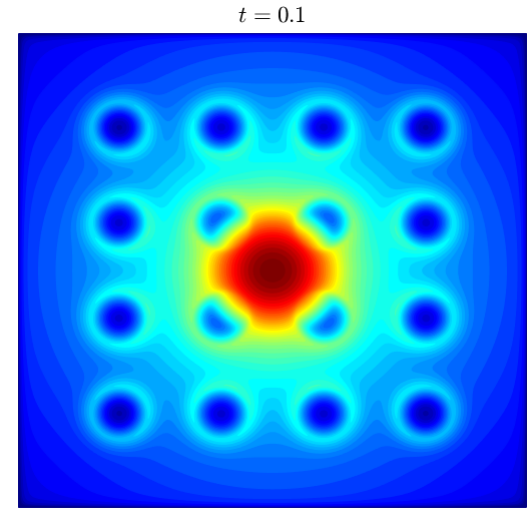}
    \hspace{0.5mm}
    \includegraphics[width=0.2\linewidth, height=3.3cm]{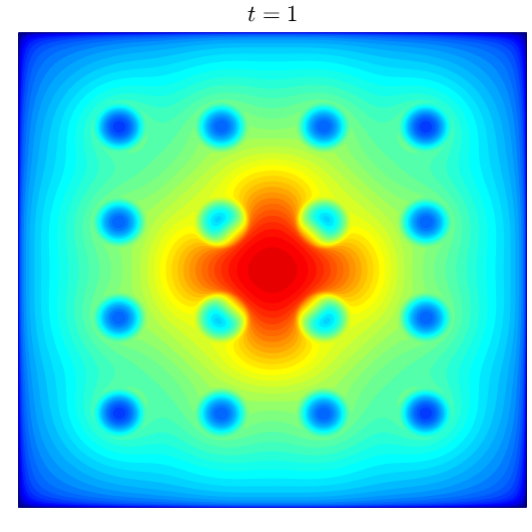} 
    \hspace{0.5mm}
    \includegraphics[width=0.2\linewidth, height=3.3cm]
{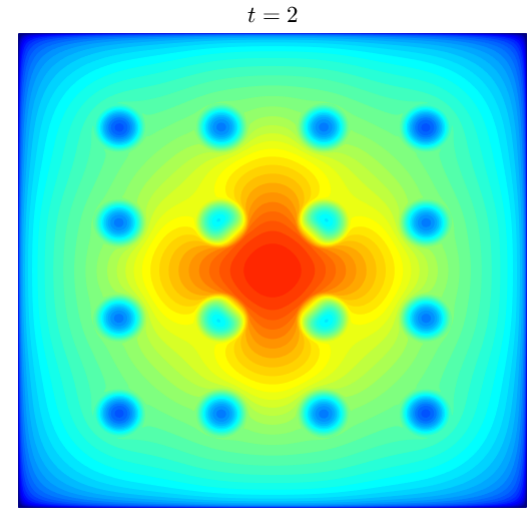}
    \hspace{0.5mm}
    \includegraphics[width=0.22\linewidth, height=3.4cm]
{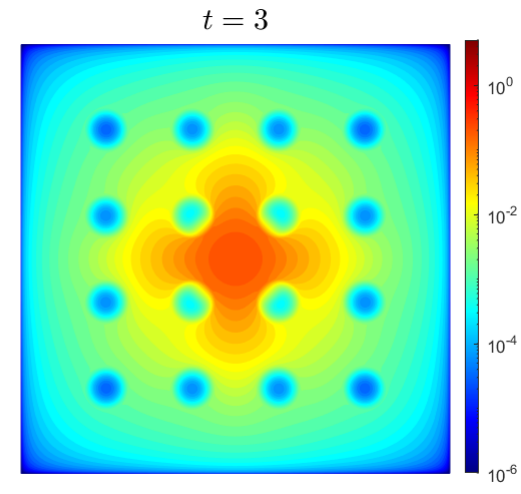}
    \vspace{-3mm}
\caption{Dynamics of the solution to the diffusion problem \eqref{diffusion}--\eqref{IC} in a {\it heterogeneous} medium characterized by $\ap_2(\bx)$ in \eqref{alpha}.
Top row: surface plots; bottom row: corresponding top views.}
    \label{fig:ex4-4-4}
\end{figure}

Fig.~\ref{fig:ex4-4-4}(a)--(c) presents the solution profiles along the line $y = 0.6$ at different times for $\alpha_1(\mathbf{x})$, $\alpha_2(\mathbf{x})$, and $\alpha(\mathbf{x}) \equiv \alpha_0$.
The line $y = 0.6$ passes through the center of the embedded material.
The heterogeneity effects become more evident in these plots, where sharp solution layers form near the interfaces between materials (see Fig.~\ref{fig:ex4-4-4}(b)–(c)).
\begin{figure}[htb!]
    \centerline{
    \raisebox{5mm}{(a)}
    \includegraphics[width=0.42\linewidth, height = 5.cm]{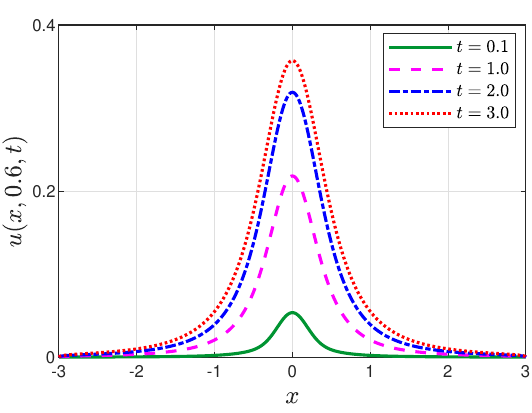}
    \hspace{4mm}
    \raisebox{5mm}{(b)}
    \includegraphics[width=0.42\linewidth, height = 5.cm]{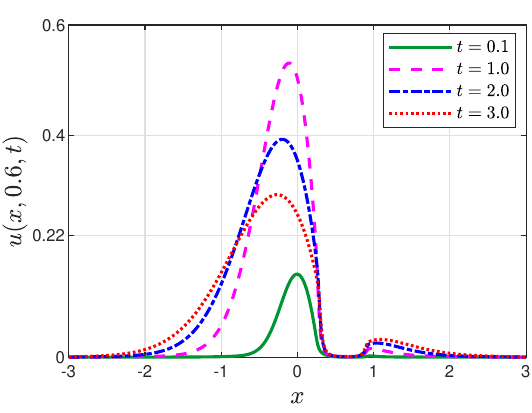}}
    \centerline{
    \raisebox{5mm}{(c)}
    \includegraphics[width=0.42\linewidth, height = 5.cm]{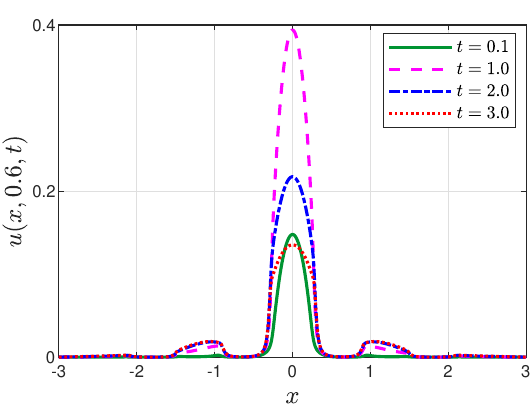}
    \hspace{5mm}
    \raisebox{5mm}{(d)}
    \includegraphics[width=0.40\linewidth, height = 5.cm]{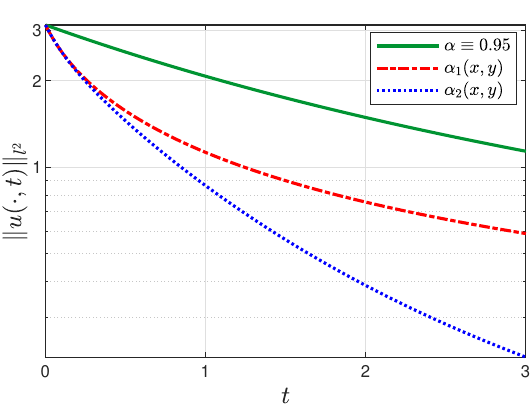}}
    \vspace{-3mm}
\caption{(a)--(c) Time evolution of the solution along the line $y=0.6$, which passes through the center of the slow-diffusion inclusion in the heterogeneous cases.
(a): homogeneous case with $\alpha \equiv 0.95$, (b) and (c): heterogeneous cases with $\alpha_1(\bx)$ and $\alpha_2(\bx)$, respectively.
(d) Evolution of the mass $\|u(\cdot,t)\|_{\ell^2}$, quantifying the solution decay.
}
    \label{fig:ex4-4-5}
\end{figure}
Moreover, Fig.~\ref{fig:ex4-4-4}(a) shows that in the homogeneous medium the solution at $y=0.6$ increases monotonically in time.
In contrast, this monotonic trend breaks down in the heterogeneous cases where the value at $y=0.6$ first rises and then declines (Fig.~\ref{fig:ex4-4-4}(b)–(c)).
This clearly suggestes the complicattion in the heterogeneous media.
Fig.~\ref{fig:ex4-4-4} (d) further compares the temporal evolution of the mass $\|u(\cdot,t)\|_{l^2}$ for these three cases.

\section{Conclusions}
\label{section5}

We introduced an accurate parallel finite-difference discretization for the variable-order fractional Laplacian $(-\Delta)^{\alpha(\bx)/2}$ and applied it to study anomalous diffusion in heterogeneous media.
To address the computational challenges arising from both heterogeneity and nonlocality, we proposed a decomposition strategy that separately handles the heterogeneous order and the nonlocal interactions.
Our method requires memory of order $\mathcal{O}(MN)$ and has a computational complexity $\mathcal{O}\!\big(MN\log(2N)\big)$ with $N$ the total number of spatial points and $M \ll N$.
Moreover, our method is naturally parallel and can further reduce computation time.
Our error analysis shows that the rate depends on the competition between the local regularity gaps at interior and boundary low-regularity points and the regularity exponents of exterior points.
The numerical experiments confirm the predicted convergence rates and illustrate the interplay between solution regularity and the spatially varying fractional order.
Extensive experiments were conducted to demonstrate the effectiveness of our method in solving anomalous diffusion problems, especially in higher dimensions.
In contrast to the constant-order cases, the interaction between the solution regularity and the spatial order $\alpha(\bx)$ plays a critical role in the design and analysis of numerical schemes.
This aspect remains largely unexplored in the literature, and standardized benchmark problems are still lacking.
We will address these issues in future work.

\section*{Acknowledgements}
This work was partially supported by the US National Science Foundation under grant numbers DMS–1913293 and DMS–1953177.

\appendix
\counterwithin{equation}{section}
\renewcommand{\theequation}{\thesection.\arabic{equation}}

\section{Proof of Theorem~\ref{thm}}
\label{thm1proof}

We first introduce two lemmas that will be used repeatedly in the proof.
\begin{lemma}
\label{lem:centered}
For a function $v$, set $\varphi_v(x,\xi):=[v(x+\xi)-2v(x)+v(x-\xi)]/\xi^2$ whenever it is defined.
\begin{enumerate}[label=\textup{(\alph*)}]
\item If $P$ is a polynomial of degree at most three, then $\varphi_P(x,\xi)=P''(x)$ for all $\xi\neq0$.
\item If $v\in C^4([x-r,x+r])$, then for $0<\xi\le r$,
\beas
\label{eq:centered-bounds}
\begin{aligned}
|\varphi_v(x,\xi)-v''(x)|
&\leq C\xi^2\|v^{(4)}\|_{L^\infty([x-r,x+r])},\\
|\partial_\xi\varphi_v(x,\xi)|
&\leq C\xi\|v^{(4)}\|_{L^\infty([x-r,x+r])},\\
|\partial_{\xi\xi}\varphi_v(x,\xi)|
&\leq C\|v^{(4)}\|_{L^\infty([x-r,x+r])},
\end{aligned}
\eeas
where $C$ is an absolute constant independent of $v$, $x$, $r$, and $\xi$.
\end{enumerate}
\end{lemma}

\begin{proof}
Taylor's formula with integral remainder gives
$$
\varphi_v(x,\xi)
=\int_0^1(1-t)
\bigl[v''(x+t\xi)+v''(x-t\xi)\bigr]\,\mathrm dt.
$$
For a polynomial $P$ of degree at most three, $P''$ is affine, so $\varphi_P(x,\xi)=P''(x)$.
The three bounds follow by subtracting $v''(x)$ and differentiating once and twice under the integral sign, using symmetry and the bound on $v^{(4)}$.
\end{proof}
\begin{lemma}[Consequences of the local remainder bounds]
\label{lem:local}
Suppose Assumption~\ref{assump} holds.
For a fixed $k$, denote $\widetilde r_k=u(x)-P_k(x)$ and $d_k=|x-\hat x_k|$.
Let $\rho=\min\{\delta/8,L\}$ and suppose that $x\in\overline{\Omega}$ with $d_k<2\rho$.
Then all sampled points $x\pm\xi$, $0\le\xi\le\rho$, lie in $B_\delta(\hat x_k)\cap\Omega_L$.
If $0<\sigma_k<1$, then
\begin{equation}
\label{eq:local-holder}
|\widetilde r_k(y)-\widetilde r_k(z)|\leq C|y-z|^{\sigma_k},
\qquad y,z\in B_\delta(\hat x_k)\cap\overline\Omega_L,
\end{equation}
and
\begin{equation}
\label{eq:local-psi}
|\psi(x,\xi)|\leq C\xi^{\sigma_k-2},
\qquad 0<\xi\leq\rho.
\end{equation}
If $1\leq\sigma_k<4$, then $\psi(x,\cdot)$ is absolutely continuous on every interval $[\underline{a},\rho]$ with $0<\underline{a}<\rho$, and
\begin{equation}
\label{eq:local-psi-prime}
|\partial_\xi\psi(x,\xi)|\leq C\xi^{\sigma_k-3}
\quad\text{for almost every }0<\xi\leq\rho.
\end{equation}
For every $0<\sigma_k<4$, the following bounds also hold whenever $0<\xi\leq\rho$ and $2d_k\leq\xi$:
\begin{equation}
\label{eq:local-separated}
|\psi(x,\xi)-P_k''(x)|\leq C\xi^{\sigma_k-2},
\qquad
|\psi_\xi(x,\xi)|\leq C\xi^{\sigma_k-3},
\qquad
|\psi_{\xi\xi}(x,\xi)|\leq C\xi^{\sigma_k-4}.
\end{equation}
\end{lemma}

These estimates follow from Lemma~\ref{lem:centered} on intervals avoiding $\hat x_k$ and the local remainder bounds in \ref{A2} near $\hat x_k$.
The H\"older continuity of $\widetilde r_k$ follows by integrating its derivative bound on either side of $\hat x_k$ and using continuity at that point.
For $\sigma_k\ge1$, this bound gives local Lipschitz continuity of $u$, which implies absolute continuity of $\psi(x,\cdot)$ on intervals bounded away from zero.

\bigskip
By~\eqref{newL}--\eqref{part1}, we rewrite the error function in \eqref{error} as
\begin{equation}\label{error-sep}
e_\Delta^h(x)
=-\frac{c_{1,\alpha(x)}}2
\bigg(E_0(x)+\sum_{i=1}^{N-1}E_i(x)\bigg),\qquad  x\in\Og_h, 
\end{equation}
where
\begin{equation}\label{def:cell-errors}
\begin{aligned}
E_0(x)
&=2\int_0^h[\psi(x,\xi)-\psi(x,h)]\xi^{1-\alpha(x)}\,\md\xi,\\
E_i(x)
&=\int_{\xi_i}^{\xi_{i+1}}
[2\psi(x,\xi)-\psi(x,\xi_i)-\psi(x,\xi_{i+1})]
\xi^{1-\alpha(x)}\,\md\xi,
\qquad 1\le i\le N-1.
\end{aligned}
\end{equation}
Throughout the proof, $x$ denote an arbitrary grid point in $\Og_h$, and $C>0$ denotes a generic constant that is independent of $x$ and $h$.
Set $\rho=\min\{\delta/8,L\}$ and take $0<h<\min\{1,\rho/2,L\}$ sufficiently small that
$$
\begin{aligned}
2h&<\operatorname{dist}(\hat x_k,\overline\Omega),
&&\hat x_k\in\bar X\setminus\overline\Omega,\\
C_\alpha(2h)^\gamma
&\le\tfrac12\bigl(\sigma_k-\alpha(\hat x_k)\bigr),
&&\hat x_k\in\bar X\cap\overline\Omega.
\end{aligned}
$$
These conditions can be satisfied by Assumption~\eqref{A3} and the finiteness of $\bar X$.
For $x\in\overline\Omega$ and $1\le i\le N-1$, define
$$
d_k(x)=|x-\hat x_k|,\qquad
s_{ik}(x) =\operatorname{dist}\bigl(d_k(x),[\xi_i,\xi_{i+1}]\bigr)
=\min_{\xi\in[\xi_i,\xi_{i+1}]}|\xi-d_k(x)|,
$$
that is, $d_k(x)$ measures the distance from $x$ to the low-regularity point $\hat x_k$, and $s_{ik}(x)$ measures the distance from the corresponding interaction distance to a cell $[\xi_i, \xi_{i+1}]$.
In the following, we suppress the dependence on $x$ in $d_k$ and $s_{ik}$ when no confusion arises.

\subsection{\texorpdfstring{Estimate of $E_0(x)$}{Estimate of E0(x)}}
We divide our analysis into two cases.

\medskip
\noindent{\bf Case I: $d_k\geq2h$ for all $k$.}
Let $\zeta\in[x-h,x+h]$.
For every $k$, $|\zeta-\hat x_k|\geq d_k-h\geq d_k/2>0$, and thus the interval $[x-h,x+h]$ contains no low-regularity points and $u\in C^4([x-h,x+h])$.
Since the balls $B_\delta(\hat x_k)$ are pairwise disjoint, $\zeta$ lies in at most one of them.
If $\zeta\in B_\delta(\hat x_k)$ for some $k$, then, since $P_k$ has degree at most three, $u^{(4)}(\zeta)=\widetilde r_k^{(4)}(\zeta)$.
By Assumption~\eqref{A2}\ref{A2ii},
\beas
|u^{(4)}(\zeta)|
\leq C|\zeta-\hat x_k|^{\sigma_k-4}
\leq C d_k^{\sigma_k-4},
\eeas
since $\sigma_k<4$.
Otherwise, if $\zeta$ lies outside all these neighborhoods, Assumption~\eqref{A2}\ref{A2i} gives $|u^{(4)}(\zeta)|\leq C$.
Combining these bounds yields
\beas
\|u^{(4)}\|_{L^\infty([x-h,x+h])}
\leq C\biggl(1+\sum_{1 \le k \le K} d_k^{\sigma_k-4}\biggr).
\eeas
By it and Lemma~\ref{lem:centered}, we obtain
\beas
\begin{aligned}[b]
|E_0(x)|
&\leq 2\int_0^h
|\psi(x,\xi)-\psi(x,h)|
\xi^{1-\alpha(x)}\,\md\xi\\
&\leq C \|u^{(4)}\|_{L^\infty([x-h,x+h])} 
\int_0^h(\xi^2+h^2)\xi^{1-\alpha(x)}\,\md\xi\\
&\leq C h^{4-\alpha(x)}
\biggl(1+\sum_{1 \le k \le K} d_k^{\sigma_k-4}\biggr),
\end{aligned}
\eeas
where we used $2-\alpha(x)\geq2-\alpha_+>0$.
For $\hat x_k \in \overline \Omega$, when $d_k\geq2h$, we have the estimate
\bea\label{temp-d}
\begin{aligned}
h^{4-\alpha(x)}d_k^{\sigma_k-4}
=h^2d_k^{\sigma_k-2-\alpha(x)}
\left(\frac{h}{d_k}\right)^{2-\alpha(x)} \leq Ch^2d_k^{\sigma_k-\ap(\hat{x}_k)-2}
\leq Ch^{\min\{\sigma_k-\ap(\hat{x}_k),2\}},
\end{aligned}
\eea
where we used the estimate in \eqref{eq:order-comparison1} with $t=d_k$.
For $\hat x_k\in\bar X\cap\overline{\Og}$, Assumption~\eqref{A1} gives
\begin{equation}\label{eq:order-comparison1}
\begin{gathered}
t^{\ap(\hat{x}_k)-\ap(x)}\leq\exp\bigl(C_\ap 2^\gamma t^\gamma|\log t|\bigr)\leq C, \qquad 
d_k\leq2t,\qquad 0<t\leq L,
\end{gathered}
\end{equation}
where $C$ is independent of $x$ and $t$, since $t^\gamma|\log t|$ is bounded on $(0,L]$.
Combine the above discussion leads to
\begin{equation}\label{E0-caseI}
|E_0(x)|
\leq C\Big[h^2+\sum_{\hat{x}_k\in\bar X\cap\overline{\Og}}h^{\min\{\sigma_k-\ap(\hat{x}_k),2\}}\Big],
\end{equation}
i.e., the contributions associated with $\hat x_k\in\bar X\cap\overline{\Og}$ are $O(h^{\min\{\sigma_k-\ap(\hat x_k),2\}})$, while those associated with $\hat x_k\in\bar X\setminus\overline{\Og}$ are $O(h^2)$.

\medskip
\noindent{\bf Case II: $d_k<2h$ for some $k$.}
Since $h<\delta/16$ and the neighborhoods $B_\delta(\hat x_k)$ are pairwise disjoint, this index is unique, and $[x-h,x+h]\subset B_\delta(\hat x_k)$.
Our choice of $h$ ensures that $\hat x_k\in\overline{\Omega}$ and
\beas
\sigma_k-\alpha(x) \ge 
\sigma_k-\alpha(\hat x_k) - C_{\alpha}(2h)^\gamma 
\geq \frac12\bigl(\sigma_k-\alpha(\hat x_k)\bigr)>0.
\eeas
Lemma~\ref{lem:local} gives
\bea\label{est:psi-near}
|\psi(x,\xi)|\leq C\xi^{\sigma_k-2},
&&\quad 0<\sigma_k<1,\\
\label{est:psi1-near}
|\partial_\xi\psi(x,\xi)|\leq C\xi^{\sigma_k-3},
&&\quad 1\leq\sigma_k<4.
\eea
The first inequality holds for $0<\xi\leq h$, and the second holds for almost every $\xi\in(0,h)$.
For $\sigma_k\geq1$, Lemma~\ref{lem:local} also gives absolute continuity of $\psi(x,\cdot)$ on every $[\underline{a},h]$ with $0<\underline{a}<h$.
We consider the two ranges of $\sigma_k$ separately.

If $0<\sigma_k<1$, the triangle inequality and \eqref{est:psi-near} give
$$
\begin{aligned}
|E_0(x)|
&\leq 2\int_0^h
\bigl(|\psi(x,\xi)|+|\psi(x,h)|\bigr)
\xi^{1-\alpha(x)}\,\md\xi\\
&\leq C\int_0^h
\bigl(\xi^{\sigma_k-2}+h^{\sigma_k-2}\bigr)
\xi^{1-\alpha(x)}\,\md\xi \leq C h^{\sigma_k-\alpha(x)}.
\end{aligned}
$$
If $1\leq\sigma_k<4$, the fundamental theorem for absolutely continuous functions $\psi(x, \cdot)$ and Tonelli's theorem give
$$
\begin{aligned}[b]
|E_0(x)|
&\leq 2\int_0^h
\bigg(\int_\xi^h
|\partial_\eta\psi(x,\eta)|\,\md\eta\bigg)
\xi^{1-\alpha(x)}\,\md\xi\\
&=2\int_0^h
|\partial_\eta\psi(x,\eta)|
\bigg(\int_0^\eta\xi^{1-\alpha(x)}\,\md\xi\bigg)
\md\eta\\
&\leq C\int_0^h
\eta^{\sigma_k-1-\alpha(x)}\,\md\eta \leq C h^{\sigma_k-\alpha(x)},
\end{aligned}
$$
where we used \eqref{est:psi1-near} and the positive lower bounds on $2-\alpha(x)$ and $\sigma_k-\alpha(x)$.
Finally, since $d_k<2h$, the H\"older continuity of $\alpha$ gives
\begin{equation}\label{eq:order-comparison2}
h^{\sigma_k-\alpha(x)}
\leq h^{\sigma_k-\alpha(\hat x_k)}
\exp\bigl(C_\alpha(2h)^\gamma\log(1/h)\bigr)
\leq C h^{\sigma_k-\alpha(\hat x_k)},
\end{equation}
where we used the boundedness of $h^\gamma\log(1/h)$.
Therefore,
\begin{equation}\label{E0-caseII}
|E_0(x)|\leq C h^{\sigma_k-\alpha(\hat x_k)},
\qquad 0<\sigma_k<4.
\end{equation}

Combining the estimates for Case~I in \eqref{E0-caseI} and Case~II in \eqref{E0-caseII}, we obtain
\begin{equation}\label{est:combined-E0}
\max_{x\in \Og_h}|E_0(x)|
\leq C\bigg[h^2+\sum_{\hat{x}_k\in\overline{\Og}\cap\bar{X}}h^{\min\{\sigma_k-\ap(\hat{x}_k),2\}}\bigg].
\end{equation}

\subsection{\texorpdfstring{Estimate of $E_i(x)$}{Estimate of Ei(x)}}
To estimate $E_i(x)$, we distinguish whether the two intervals $[x+\xi_i,x+\xi_{i+1}]$ and $[x-\xi_{i+1},x-\xi_i]$ contain any low-regularity points.
If neither interval intersects $\bar{X}$, the smoothness of $\psi(x,\xi)$ allows us to derive the general cellwise estimate \eqref{pre_Ei}.
Otherwise, we use the local estimate in Lemma \ref{A2} to control the contributions involving low-regularity points.

For $1\le i\le N-1$, if $([x+\xi_i,x+\xi_{i+1}]\cup[x-\xi_{i+1},x-\xi_i])\cap\bar{X}=\emptyset$, then Assumption~\eqref{A2} implies that $\psi(x,\cdot)\in C^2([\xi_i,\xi_{i+1}])$.
Taylor expansion about the midpoint $m_i=\xi_i+h/2$ gives
$$
2\psi(x,\xi)-\psi(x,\xi_i)-\psi(x,\xi_{i+1})
=2(\xi-m_i)\psi_\xi(x,m_i)+R_i(\xi)
$$
with $|R_i(\xi)|\leq Ch^2\|\psi_{\xi\xi}(x,\cdot)\|_{L^\infty([\xi_i, \xi_{i+1}])}$.
Since $\int_{\xi_i}^{\xi_{i+1}}(\xi-m_i)\,\md\xi=0$, the mean value theorem yields
$$
\begin{aligned}
\left|\int_{\xi_i}^{\xi_{i+1}}(\xi-m_i)\xi^{1-\alpha(x)}\,\md\xi\right|
& =\left|\int_{\xi_i}^{\xi_{i+1}}(\xi-m_i)
\bigl[\xi^{1-\alpha(x)}-m_i^{1-\alpha(x)}\bigr]\,\md\xi\right|
\\
&\leq C\xi_i^{-\alpha(x)}\int_{\xi_i}^{\xi_{i+1}}|\xi-m_i|^2\,\md\xi 
\leq Ch^3\xi_i^{-\alpha(x)}.
\end{aligned}
$$
Since $\xi_i\le\xi\le2\xi_i$ for $\xi\in[\xi_i,\xi_{i+1}]$ with $i\ge1$, we have
$$
\int_{\xi_i}^{\xi_{i+1}}\xi^{1-\alpha(x)}\,\md\xi
\leq Ch\xi_i^{1-\alpha(x)}.
$$
Substituting the midpoint expansion into the definition of $E_i(x)$ and using these two bounds gives
\begin{equation}\label{pre_Ei}
|E_i(x)|
\leq Ch^3\left[
\xi_i^{-\alpha(x)}\|\psi_\xi(x,\cdot)\|_{L^\infty([\xi_i,\xi_{i+1}])}
+\xi_i^{1-\alpha(x)}\|\psi_{\xi\xi}(x,\cdot)\|_{L^\infty([\xi_i,\xi_{i+1}])}
\right].
\end{equation} 
If either interval $[x+\xi_i,x+\xi_{i+1}]$ or $[x-\xi_{i+1},x-\xi_i]$ intersects $\bar X$, we use the local bounds in Lemma~\ref{lem:local} to estimate $E_i(x)$, as detailed below.

To facilitate our discussion, we divide the terms $E_i$ (for $1\le i\le N-1$) into two groups and introduce the index sets
\beas
S_{\rm near}
=\{i\in\{1,\dots,N-1\}\mid \xi_{i+1}\le\rho\},
\qquad
S_{\rm far}
=\{i\in\{1,\dots,N-1\}\mid \xi_{i+1}>\rho\},
\eeas
which classify the cells $[\xi_i,\xi_{i+1}]$ according to their interaction distance from the point $x$.

\subsubsection{\texorpdfstring{Estimates of $E_i(x)$ for $i\in S_{\rm near}$}{Estimates of Ei(x) for i in S near}}
For $i\in S_{\rm near}$, we have $x\pm\xi\in[x-\rho,x+\rho]$ for every $\xi\in[\xi_i,\xi_{i+1}]$.
As before, we divide the analysis into two cases.

\medskip
\noindent{\bf Case I: $d_k\geq2\rho$ for all $k$. }
For $\zeta\in[x-\rho,x+\rho]$, we have $|\zeta-\hat x_k|\geq d_k-\rho\geq\rho$ for every $k$.
Thus, Assumption~\eqref{A2} gives $u\in C^4([x-\rho,x+\rho])$ and $\|u^{(4)}\|_{L^\infty([x-\rho,x+\rho])}\leq C$, uniformly in $x$.
By Lemma~\ref{lem:centered} (b), we obtain
$$
|\psi_\xi(x,\xi)|\leq C\xi,
\qquad
|\psi_{\xi\xi}(x,\xi)|\leq C,
\qquad 0<\xi\leq\rho.
$$
Substituting these bounds into \eqref{pre_Ei} and using $\xi_{i+1}\leq2\xi_i$ gives
\beas\label{Ei-near-caseI}
|E_i(x)|\leq Ch^3\xi_i^{1-\alpha(x)}.
\eeas
Consequently,
\bea\label{est:Ei-I}
\begin{aligned}[b]
\sum_{i\in S_{\rm near}}|E_i(x)|
&\leq Ch^3
\sum_{i\in S_{\rm near}}
\xi_i^{1-\alpha(x)} \\
&\leq Ch^2\int_0^\rho\xi^{1-\alpha(x)}\,\md\xi \leq Ch^2, \qquad \mbox{if $d_k \ge 2\rho$ for all $k$}.
\end{aligned}
\eea

\paragraph{Case II: $d_k<2\rho$ for some $k$.} Since the neighborhoods $B_\delta(\hat x_k)$ are pairwise disjoint and $\rho\leq\delta/8$, there is only one such index.
Moreover,
$$
|x\pm\xi-\hat x_k|
\leq d_k+\rho<3\rho<\delta,
$$
so $[x-\rho,x+\rho]\subset B_\delta(\hat x_k)$.
We further distinguish between $d_k\leq2h$ and $d_k>2h$.
\begin{itemize}
\item[{\bf (a)}] $d_k\leq2h$.
For cells with $\xi_i\leq4h$, we have $h\leq\xi_i\leq\xi_{i+1}\leq5h$.
Since $P_k$ has degree at most three, we can write
$$
P_k''(x) = \frac{P_k(x+\xi)-2P_k(x)+P_k(x-\xi)}{\xi^2}.
$$
Since $d_k\le2h$ and $\xi_{i+1}\le5h$, the points $x\pm\xi$ lie in $[\hat x_k-7h,\hat x_k+7h]$ for all $\xi\in[\xi_i,\xi_{i+1}]$.
Then, for $\xi\in[\xi_i,\xi_{i+1}]$, we obtain
\beas\label{psi-P}
|\psi(x,\xi)-P_k''(x)|
=\frac{1}{\xi^2}
\bigl|
\widetilde r_k(x+\xi)
-2\widetilde r_k(x)+\widetilde r_k(x-\xi)
\bigr|
\le Ch^{\sigma_k}\xi^{-2}
\eeas
by Assumption~\eqref{A2}\ref{A2ii}.
Using this bound and canceling the polynomial contribution gives
\bea\label{est:case2a-1}
|E_i(x)|
&\leq&\int_{\xi_i}^{\xi_{i+1}}
|2\psi(x,\xi)-\psi(x,\xi_i)-\psi(x,\xi_{i+1})|
\xi^{1-\alpha(x)}\,\md\xi\nn\\
&=&\int_{\xi_i}^{\xi_{i+1}}
\Bigl|
2[\psi(x,\xi)-P_k''(x)]
-[\psi(x,\xi_i)-P_k''(x)]
-[\psi(x,\xi_{i+1})-P_k''(x)]
\Bigr|
\xi^{1-\alpha(x)}\,\md\xi\nn\\
&\le& Ch^{\sigma_k}\int_{\xi_i}^{\xi_{i+1}}
\bigg(\frac{2}{\xi^2}
+\frac{1}{\xi_i^2}
+\frac{1}{\xi_{i+1}^2}\bigg)\xi^{1-\alpha(x)}\,\md\xi\nn\\
&\leq& Ch^{\sigma_k-2}
\int_{\xi_i}^{\xi_{i+1}}
\xi^{1-\alpha(x)}\,\md\xi \leq Ch^{\sigma_k-\alpha(x)},
\eea
where we used $\xi,\xi_i,\xi_{i+1}\ge h$ in the penultimate inequality and $h\le\xi\le5h$ in the last inequality.

For cells with $\xi_i>4h$, we have $\xi\geq\xi_i>4h\geq2d_k$ for any $\xi \in [\xi_i, \xi_{i+1}]$.
Thus, neither interval $[x+\xi_i,x+\xi_{i+1}]$ nor $[x-\xi_{i+1},x-\xi_i]$ contains low-regularity points.
Using \eqref{pre_Ei} together with the bounds in \eqref{eq:local-separated} and $\xi_i\le\xi\le2\xi_i$, we obtain
\begin{equation}\label{est:case2a-2}
\begin{aligned}[b]
|E_i(x)|
&\leq Ch^3\left[
\xi_i^{-\alpha(x)}\xi_i^{\sigma_k-3}
+\xi_i^{1-\alpha(x)}\xi_i^{\sigma_k-4}
\right]\\
&\leq Ch^3\xi_i^{\sigma_k-3-\alpha(x)}.
\end{aligned}
\end{equation}
Note that there are at most four cells with $\xi_i\leq4h$.
Summing \eqref{est:case2a-1} and \eqref{est:case2a-2} over the two cell groups gives,
\begin{equation}\label{est:case2a}
\begin{aligned}
\sum_{i\in S_{\rm near}}|E_i(x)|
&\leq C\bigg[ 
h^{\sigma_k-\alpha(\hat{x}_k)}
+h^2\int_h^\rho
\xi^{\sigma_k-3-\alpha(\hat{x}_k)}\,\md\xi
\bigg] \\
& \leq C\begin{cases}
h^{\sigma_k-\ap(\hat{x}_k)}, &0<\sigma_k-\ap(\hat{x}_k)<2, \\
h^2\big(1+\log(1/h)\big), &\sigma_k-\ap(\hat{x}_k)=2, \\
h^2, &\sigma_k-\ap(\hat{x}_k) >2.
\end{cases}, \qquad  \mbox{if $d_k\leq2h$}.
\end{aligned}
\end{equation}
Here we used the estimate in \eqref{eq:order-comparison2} for $h^{\sigma_k-\ap(x)}$ and \eqref{eq:order-comparison1} with $t=\xi$ for the integral, since $d_k\le2h\le2\xi$ for $h\le\xi\le\rho$.
The final bound follows by direct integration.
In particular, when $\sigma_k-\alpha(\hat x_k)=2$, the integral equals $\log(\rho/h)$, giving the stated logarithmic factor.

\item[{\bf (b)}] $d_k>2h$.
We divide the cells into three groups: those with $\xi_{i+1}<d_k/2$, those with $\xi_i>2d_k$, and those intersecting $[d_k/2,2d_k]$.

For cells with $\xi_{i+1}<d_k/2$, every point $\xi\in[x-\xi_{i+1},x+\xi_{i+1}]$ satisfies ${d_k}/{2}\leq|\xi-\hat x_k| \leq {3d_k}/{2}$.
Thus, $[x-\xi_{i+1},x+\xi_{i+1}] \subset B_\delta(\hat x_k)$ and contains no low-regularity points, so $u\in C^4([x-\xi_{i+1},x+\xi_{i+1}])$.
Since $P_k^{(4)}=0$, Assumption~\eqref{A2}\ref{A2ii} gives
\beas
\|u^{(4)}\|_{L^\infty([x-\xi_{i+1},x+\xi_{i+1}])}
\leq Cd_k^{\sigma_k-4}.
\eeas
Lemma~\ref{lem:centered}(b) therefore yields
$$
|\psi_\xi(x,\xi)|\leq C\xi d_k^{\sigma_k-4},
\qquad
|\psi_{\xi\xi}(x,\xi)|\leq Cd_k^{\sigma_k-4}.
$$ 
Substituting these bounds into \eqref{pre_Ei}, we obtain
\begin{equation}
\label{caseII(b)-1}
|E_i(x)|
\leq Ch^3d_k^{\sigma_k-4}\xi_i^{1-\alpha(x)}.
\end{equation}

For cells with $\xi_i>2d_k$, we have ${\xi}/{2}\leq|x\pm\xi-\hat x_k| \leq {3\xi}/{2}$.
Thus, neither interval $[x+\xi_i,x+\xi_{i+1}]$ nor $[x-\xi_{i+1},x-\xi_i]$ contains low-regularity points.
Using the bounds in \eqref{eq:local-separated} to \eqref{pre_Ei} yields
\begin{equation}
\label{caseII(b)-2}
|E_i(x)|\leq Ch^3\xi_i^{\sigma_k-3-\alpha(x)}.
\end{equation}
This contribution is zero if $2d_k\geq\rho$.

The remaining cells intersect $[d_k/2,2d_k]$.
Since $h<d_k/2$ and $\xi_i\geq h$, these cells satisfy $d_k/4\leq\xi\leq3d_k$.
Assume $\hat{x}_k>x$, otherwise interchange $x+\xi$ and $x-\xi$.
Using the polynomial cancellation, write
\begin{equation}\label{psi-P2}
    \psi(x,\xi) - P_{k}'' = \fl{\tilde{r}_{k}(x+\xi)}{\xi^2} + \fl{\tilde{r}_{k}(x-\xi)-2\tilde{r}_{k}(x)}{\xi^2}.
\end{equation}
Since $|x-\xi-\hat{x}_k|=d_k+\xi\asymp d_k$, Assumption~\eqref{A2}\ref{A2ii} gives
\begin{equation}\label{est:IIb-smooth-quotient}
\left|
\partial_\xi^m
\left(\frac{\widetilde r_k(x-\xi)-2\widetilde r_k(x)}{\xi^2}\right)
\right|
\leq Cd_k^{\sigma_k-2-m},
\qquad m=1,2.
\end{equation}
We further distinguish between $s_{ik}\geq h$ and $s_{ik}<h$.

\smallskip
If $s_{ik}\geq h$, then $s_{ik}\leq|\xi-d_k|\leq s_{ik}+h\leq2s_{ik}$.
Differentiating the first quotient in \eqref{psi-P2} and using Lemma \eqref{A2} yields
$$
\left|
    \partial_{\xi}^m \left(\fl{\tilde{r}_{k}(x+\xi)}{\xi^2}\right)
\right| \leq C\sum_{j=0}^{m}d_{k}^{-2-m+j}s_{ik}^{\sigma_k-j}\leq Cd_{k}^{-2}s_{ik}^{\sigma_k-m},\qquad m=1,2,
$$
where the last inequality uses $s_{ik}\leq Cd_k$.
Applying \eqref{pre_Ei} to both quotients in \eqref{psi-P2} and using linearity, we obtain
\bea\label{caseII(b)-3}
\begin{aligned}[b]
    |E_{i}(x)| &\leq Ch^3\left[ d_{k}^{-2-\ap(x)} s_{ik}^{\sigma_k-1} + d_{k}^{-1-\ap(x)}s_{ik}^{\sigma_k-2}+d_k^{\sigma_k-3-\ap(x)}\right] \\
    & \leq Ch^3 \left[ d_{k}^{-1-\ap(x)} s_{ik}^{\sigma_k-2} +d_k^{\sigma_k-3-\ap(x)}\right].
\end{aligned}
\eea

If $s_{ik}<h$, then $|\xi-d_k|<2h$ for any $\xi \in [\xi_i, \xi_{i+1}]$, so
\begin{equation}\label{est:IIb-near-quotient}
    \left|\fl{\tilde{r}(x+\xi)}{\xi^2}\right| \leq Ch^{\sigma_k}d_k^{-2}.
\end{equation}
Using the triangle inequality and \eqref{est:IIb-near-quotient} for the first quotient in \eqref{psi-P2}, and \eqref{pre_Ei} with \eqref{est:IIb-smooth-quotient} for the second, we obtain
\begin{equation}
\label{caseII(b)-4}
\begin{aligned}[b]
|E_i(x)|
&\leq Ch^{\sigma_k}d_k^{-2}
\int_{\xi_i}^{\xi_{i+1}}
\xi^{1-\alpha(x)}\,\md\xi+Ch^3d_k^{\sigma_k-3-\ap(x)}\\
&\leq C\left[
h^{\sigma_k+1}d_k^{-1-\alpha(x)}
+h^3d_k^{\sigma_k-3-\alpha(x)}
\right].
\end{aligned}
\end{equation}

There are $O(d_k/h)$ cells intersecting $[d_k/2,2d_k]$.
For each integer $\ell \ge1$, at most two cells satisfy $\ell h\le s_{ik}<(\ell +1)h$, and at most four cells satisfy $s_{ik}<h$.
Summing the estimates~\eqref{caseII(b)-1}--\eqref{caseII(b)-4} over these four cell groups therefore gives
\begin{equation}\label{est:Ei:small-far}
\begin{aligned}[b]
\sum_{i\in S_{\rm near}}|E_i(x)|
\leq C\bigg[
&h^2d_k^{\sigma_k-2-\alpha(x)}
+h^2\int_{2d_k}^{\rho}
\xi^{\sigma_k-3-\alpha(x)}\,\md\xi\\
&+h^2d_k^{-1-\alpha(x)}
\int_h^{3d_k}s^{\sigma_k-2}\,\md s
+h^{\sigma_k+1}d_k^{-1-\alpha(x)}
\bigg].
\end{aligned}
\end{equation}

Next, we further simplify \eqref{est:Ei:small-far} by considering separately the cases $\hat x_k\in\bar X\cap\overline{\Og}$ and $\hat x_k\in\bar X\setminus\overline{\Og}$.
As discussed, $\hat x_k$ is the unique low-regularity point satisfying $d_k<2\rho$.
If $\hat x_k\in\bar X\cap\overline{\Og}$, we use \eqref{eq:order-comparison1} with $t=\xi$ to obtain
\beas\label{est:combined-IIb-uniform}
h^2\int_{2d_k}^{\rho}\xi^{\sigma_k-3-\alpha(x)}\,\md\xi
\le Ch^2\int_h^\rho
\xi^{\sigma_k-\ap(\hat x_k)-3}\,\md\xi.
\eeas
This integral vanishes from~\eqref{est:Ei:small-far} when $2d_k\ge\rho$.
For the other integral term, we use \eqref{eq:order-comparison1} with $t=d_k$ to obtain
\beas
h^2d_k^{-1-\alpha(x)}
\int_h^{3d_k}s^{\sigma_k-2}\,\md s \le
C\bigl(h^{\sigma_k-\ap(\hat x_k)}
+h^2d_k^{\sigma_k-\ap(\hat x_k)-2}\bigr)
\le C\bigl(h^2+h^{\sigma_k-\ap(\hat x_k)}\bigr).
\eeas
Here, we write $s^{\sigma_k-2} =s^{\ap(\hat x_k)}s^{\sigma_k-\ap(\hat x_k)-2}$ and use the esitmate $s^{\sigma_k-\ap(\hat x_k)-2} \le C\bigl( h^{\sigma_k-\ap(\hat x_k)-2} +d_k^{\sigma_k-\ap(\hat x_k)-2} \bigr)$ for $s\in[h,3d_k]$, while the final inequality follows from $d_k\ge2h$ and $d_k<2\rho$.
For the two nonintegral terms, the order comparison with $t=d_k$ and $2h<d_k<2\rho$ gives
\begin{align*}
    h^2 d_{k}^{\sigma_k-2-\ap(x)} &\leq Ch^2 d_{k}^{\sigma_k-\ap(\hat{x}_k)-2} \leq C(h^2+ h^{\sigma_k-\ap(\hat{x}_k)}), \\
    h^{\sigma_k+1}d_k^{-1-\ap(x)} &\leq Ch^{\sigma_k-\ap(\hat{x}_k)}\bigg(\fl{h}{d_k}\bigg)^{1+\ap(\hat{x}_k)} \leq Ch^{\sigma_k-\ap(\hat{x}_k)}.
\end{align*} 
Thus, for $\hat x_k\in\bar X\cap\overline{\Og}$, \eqref{est:Ei:small-far} reduces to
\begin{equation}\label{est:combined-small-interior}
\sum_{i\in S_{\rm near}}|E_i(x)|
\le C\biggl[
h^2+h^{\sigma_k-\ap(\hat x_k)}
+h^2\int_h^\rho
\xi^{\sigma_k-\ap(\hat x_k)-3}\,\md\xi
\biggr],
\end{equation}
which yields the same bounds as \eqref{est:case2a}; we omit the details for brevity.

For $\hat{x}_k\in\bar{X}\backslash\overline{\Og}$, we have $d_k\geq\operatorname{dist}(\hat{x}_k,\overline{\Og})>0$.
Hence, the first two terms in \eqref{est:Ei:small-far} are bounded by $Ch^2$, and the factors involving $d_k$ in the remaining two terms are uniformly bounded.
Using $3d_k<\delta$, we obtain
\begin{equation}\label{est:combined-small-exterior}
\begin{aligned}
    \sum_{i\in S_{\rm near}} |E_i(x)|
&\leq C\bigg[h^2+h^2\int_h^\delta s^{\sigma_k-2}\,\md s+h^{\sigma_k+1}\bigg] \\
&\leq C\begin{cases}
    h^{\sigma_k+1}, & 0<\sigma_k<1, \\
    h^2\big(1+\log(1/h)\big), &\sigma_k=1, \\
    h^2, &\sigma_k>1.
\end{cases}
\end{aligned}
\end{equation} 
\end{itemize}

\subsubsection{\texorpdfstring{Estimate of $E_i(x)$ for $i\in S_{\rm far}$}{Estimate of Ei(x) for i in S far}}

For $i\in S_{\rm far}$, $\xi_{i+1}>\rho$ and $h<\rho/2$ imply $\xi_i>\rho/2$.
Thus, the weight $\xi^{1-\ap(x)}$ and $\xi^{-1}$ are uniformly bounded on these cells.
We distinguish two cases according to $s_{ik}$.

\medskip
\noindent{\bf Case I: $s_{ik}\geq h$ for all $k$.}
For $\xi\in[\xi_i,\xi_{i+1}]$, we have $s_{ik}\leq|\xi-d_k| \leq s_{ik}+h\leq2s_{ik}$.
Neither interval $[x+\xi_i,x+\xi_{i+1}]$ nor $[x-\xi_{i+1},x-\xi_i]$ contains a low-regularity point.
Differentiating $\psi$ and applying Assumption~\eqref{A2} gives
\beas
\|\psi_{\xi}(x,\cdot)\|_{L^\infty([\xi_i,\xi_{i+1}])}
+\|\psi_{\xi\xi}(x,\cdot)\|_{L^\infty([\xi_i,\xi_{i+1}])}
\leq C\bigg(1+\sum_{1\le k\le K}s_{ik}^{\sigma_k-2}\bigg).
\eeas
Thus, \eqref{pre_Ei} yields
\begin{equation}\label{caseIII-1}
|E_i(x)|\leq Ch^3\bigg(1+\sum_{1\le k \le K}s_{ik}^{\sigma_k-2}\bigg).
\end{equation}
For a fixed $k$, we split the sum over $i\in S_{\rm far}$ according to whether $h\le s_{ik}<\delta$ or $s_{ik}\ge\delta$.
For $h\le s_{ik}<\delta$, at most two cells lie in each distance band $[\ell h,(\ell+1)h)$.
For $s_{ik}\ge\delta$, the summands are uniformly bounded.
Hence, using $N=L/h$ we obtain
\begin{equation}\label{eq:far-distance-sum}
h^3\sum_{\substack{i\in S_{\rm far}\\s_{ik}\ge h}}
s_{ik}^{\sigma_k-2}
=
h^3\sum_{\substack{i\in S_{\rm far}\\h\le s_{ik}<\delta}}
s_{ik}^{\sigma_k-2}
+
h^3\sum_{\substack{i\in S_{\rm far}\\s_{ik}\ge\delta}}
s_{ik}^{\sigma_k-2}\\
\le Ch^2\left[
1+\int_h^\delta \xi^{\sigma_k-2}\,\md\xi
\right].
\end{equation} 
Here the last partial distance band is absorbed into the integral over $[\delta/2,\delta]$, since $h<\delta/16$.

\paragraph{Case II: $s_{ik}<h$ for some $k$.} Integrating the first-derivative bounds in Assumption~\eqref{A2} on either side of each low-regularity point and using continuity gives
\beas
|u(\zeta)-u(\eta)|
\leq C\bigg(h+\sum_{\sigma_k<1}h^{\sigma_k}\bigg),
\qquad \zeta,\eta\in\Omega_L,\quad |\zeta-\eta|\leq h.
\eeas 
Since $u$ is bounded and $\xi^{-2}$ has a uniformly bounded derivative for $\xi\geq\rho/2$, it follows that,
$$
\begin{aligned}
|\psi(x,\xi)-\psi(x,\eta)|
&\leq C\Bigl[|u(x+\xi)-u(x+\eta)|
+|u(x-\xi)-u(x-\eta)|+|\xi-\eta|\Bigr]\\
&\leq C\bigg(h+\sum_{\sigma_k<1}h^{\sigma_k}\bigg),
\end{aligned}
$$
for $\xi,\eta\in[\xi_i,\xi_{i+1}]$.
Applying this bound with $\eta=\xi_i$ and $\eta=\xi_{i+1}$ in the definition of $E_i(x)$, and using the boundedness of the weight, gives
\begin{equation}\label{caseIII-2-modulus}
\begin{aligned}
|E_i(x)|
&\leq C\int_{\xi_i}^{\xi_{i+1}}
\Bigl[|\psi(x,\xi)-\psi(x,\xi_i)|
+|\psi(x,\xi)-\psi(x,\xi_{i+1})|\Bigr]\,\md\xi\\
&\leq C\bigg(h^2+\sum_{\sigma_k<1}h^{\sigma_k+1}\bigg).
\end{aligned}
\end{equation}

For each $k$, at most four cells satisfy $s_{ik}<h$, so there are at most $4K$ Case II cells, uniformly in $x$ and $h$.
Summing \eqref{caseIII-2-modulus} over these cells gives
\begin{equation}\label{caseIII-2-modulus-sum}
\sum_{\substack{i\in S_{\rm far}\\\exists k:\ s_{ik}<h}}|E_i(x)|
\leq C\bigg(h^2+\sum_{\sigma_k<1}h^{\sigma_k+1}\bigg).
\end{equation} 
Combining \eqref{caseIII-1} with \eqref{eq:far-distance-sum} and \eqref{caseIII-2-modulus-sum} gives
\begin{equation}\label{est:combined-large}
\begin{aligned}
\sum_{i\in S_{\rm far}}|E_i(x)|
&\leq C\bigg[h^2+h^2\sum_{1 \le k\le K}\int_h^\delta s^{\sigma_k-2}\,\md s
+\sum_{\sigma_k<1}h^{\sigma_k+1}\bigg]\\
&\leq C\bigg[h^2+\sum_{\sigma_k<1}h^{\sigma_k+1}
+ \sum_{\sigma_k=1}h^2\log(1/h) \bigg].
\end{aligned}
\end{equation}
The last inequality follows by direct integration, with all contributions for $\sigma_k>1$ absorbed into $Ch^2$.
Empty sums are understood to be zero.

\subsubsection{Uniform error estimate}

By \eqref{error-sep} and the uniform boundedness of $c_{1,\ap(x)}$, we obtain
\begin{equation}\label{est:combined-bound}
\begin{aligned}[b]
\|e_\Dt^h\|_{\ell^\infty(\Og_h)}
&=\max_{x\in\Og_h}|e_\Dt^h(x)|\\
&\le \max_{x\in\Og_h}
|E_0(x)|
+\max_{x\in\Og_h}\Big\{\sum_{i\in S_{\rm near}}|E_i(x)|\Big\}
+\max_{x\in\Og_h} \Big\{\sum_{i\in S_{\rm far}}|E_i(x)|\Big\}.
\end{aligned}
\end{equation}
For each $x\in\Og_h$, the first-cell term $E_0(x)$ is bounded by \eqref{est:combined-E0}, and the sum of the far-field terms is bounded by \eqref{est:combined-large}.
The summation of near-field terms are bounded by \eqref{est:Ei-I}, \eqref{est:case2a}, \eqref{est:combined-small-interior}, or \eqref{est:combined-small-exterior}, according to the cases considered above.

Recall the definition of $\mu$ in \eqref{mu} that
\begin{equation*}
    \mu:=\min\Bigl\{
\min_{\hat x_k\in\bar X\cap\overline\Omega}
\bigl(\sigma_k-\alpha(\hat x_k)\bigr),\
\min_{\hat x_k\in\bar X\setminus\overline\Omega}
(\sigma_k+1)
\Bigr\}.
\end{equation*}
If $0<\mu<2$, all terms without a logarithmic factor are bounded by $Ch^\mu$, and
$$
h^2\bigl(1+\log(1/h)\bigr)
=h^\mu h^{2-\mu}\bigl(1+\log(1/h)\bigr)
\leq Ch^\mu.
$$
If $\mu=2$, all terms are bounded by $Ch^2(1+\log(1/h))$.
If $\mu>2$, every interior gap exceeds two and every exterior point has $\sigma_k>1$, so all terms are bounded by $Ch^2$.
Since $\bar X$ is finite and all constants are independent of $x$ and $h$, these bounds hold uniformly on $\Og_h$.

Combining the above discussion with \eqref{est:combined-bound} and the estimates in \eqref{est:combined-E0}, \eqref{est:combined-large}, \eqref{est:Ei-I}, \eqref{est:case2a}, \eqref{est:combined-small-interior}, and \eqref{est:combined-small-exterior}, we obtain
\begin{equation}\label{est:combined-bound1}
\begin{aligned}[b]
\|e_\Dt^h\|_{\ell^\infty(\Og_h)} \leq C\begin{cases}
h^\mu, & 0<\mu<2,\\
h^2\bigl(1+\log(1/h)\bigr), & \mu=2,\\
h^2, & \mu>2.
\end{cases}
\end{aligned}
\end{equation} 
For sufficiently small $h$, we have $1+\log(1/h)\leq2\log(1/h)$, which gives the critical rate stated in Theorem~\ref{thm}.
This completes the proof.
\hfill$\square$


\begin{thebibliography}{99}

\bibitem{Alzahrani2022}
H.~Alzahrani, G.~Turkiyyah, O.~Knio, and D.~Keyes.
\newblock Space-fractional diffusion with variable order and diffusivity: discretization and direct solution strategies.
\newblock {\em Commun. Appl. Math. Comput.}, 4(4):1416--1440, 2022.

\bibitem{Bass1988}
R.~F. Bass.
\newblock Uniqueness in law for pure jump {Markov} processes.
\newblock {\em Probab. Theory Related Fields}, 79(2):271--287, 1988.

\bibitem{Chen2020}
X.~Chen, Z.-Q. Chen, and J.~Wang.
\newblock Heat kernel for non-local operators with variable order.
\newblock {\em Stochastic Process. Appl.}, 130(6):3574--3647, 2020.

\bibitem{Delia2021}
M.~D'Elia and P.~Bochev.
\newblock Formulation, analysis and computation of an optimization-based local-to-nonlocal coupling method.
\newblock {\em Results Appl. Math.}, 9:100129, 2021.

\bibitem{Delia2022}
M.~D'Elia and C.~Glusa.
\newblock A fractional model for anomalous diffusion with increased variability: analysis, algorithms and applications to interface problems.
\newblock {\em Numer. Methods Partial Differ. Equ.}, 38(6):2084--2103, 2022.

\bibitem{Du2013a}
Q.~Du, M.~Gunzburger, R.~B. Lehoucq, and K.~Zhou.
\newblock A nonlocal vector calculus, nonlocal volume-constrained problems, and nonlocal balance laws.
\newblock {\em Math. Models Methods Appl. Sci.}, 23(3):493--540, 2013.

\bibitem{Duo2015}
S.~Duo and Y.~Zhang.
\newblock Computing the ground and first excited states of the fractional {Schr{\"o}dinger} equation in an infinite potential well.
\newblock {\em Commun. Comput. Phys.}, 18(2):321--350, 2015.

\bibitem{Duo2019-FDM}
S.~Duo and Y.~Zhang.
\newblock Accurate numerical methods for two and three dimensional integral fractional {Laplacian} with applications.
\newblock {\em Comput. Methods Appl. Mech. Eng.}, 355:639--662, 2019.

\bibitem{Duo2019-TFL}
S.~Duo and Y.~Zhang.
\newblock Numerical approximations for the tempered fractional {Laplacian}: error analysis and applications.
\newblock {\em J. Sci. Comput.}, 81(1):569--593, 2019.

\bibitem{Duo2018}
S.~Duo, H.~W. van Wyk, and Y.~Zhang.
\newblock A novel and accurate finite difference method for the fractional {Laplacian} and the fractional {Poisson} problem.
\newblock {\em J. Comput. Phys.}, 355:233--252, 2018.

\bibitem{Farquhar2018}
M.~E. Farquhar, T.~J. Moroney, Q.~Yang, I.~W. Turner, and K.~Burrage.
\newblock Computational modelling of cardiac ischaemia using a variable-order fractional {Laplacian}.
\newblock {\em arXiv:1809.07936}, 2018.

\bibitem{Giusti2020}
A.~Giusti.
\newblock {MOND}-like fractional {Laplacian} theory.
\newblock {\em Phys. Rev. D}, 101(12):124029, 2020.

\bibitem{Giusti2020on}
A.~Giusti, R.~Garrappa, and G.~Vachon.
\newblock On the {Kuzmin} model in fractional {Newtonian} gravity.
\newblock {\em Eur. Phys. J. Plus}, 135(10):798, 2020.

\bibitem{Hao0024}
Z.~Hao, S.~Shi, Z.~Zhang, and R.~Du.
\newblock A simple and fast finite difference method for the integral fractional {Laplacian} of variable order.
\newblock {\em SIAM J. Sci. Comput.}, 48(2):A748--A783, 2026.

\bibitem{Kuhn2021}
F.~K{\"u}hn.
\newblock Schauder estimates for {Poisson} equations associated with non-local {Feller} generators.
\newblock {\em J. Theor. Probab.}, 34(3):1506--1578, 2021.

\bibitem{Landkof1972}
N.~S. Landkof.
\newblock {\em Foundations of Modern Potential Theory}.
\newblock Springer Berlin, 1972.

\bibitem{Lei2023}
W.~Lei, G.~Turkiyyah, and O.~Knio.
\newblock Finite element discretizations for variable-order fractional diffusion problems.
\newblock {\em J. Sci. Comput.}, 97(1):5, 2023.

\bibitem{Lenzi2016}
E.~K. Lenzi, H.~V. Ribeiro, A.~A. Tateishi, R.~S. Zola, and L.~R. Evangelista.
\newblock Anomalous diffusion and transport in heterogeneous systems separated by a membrane.
\newblock {\em Proc. R. Soc. A Math. Phys. Eng. Sci.}, 472(2195):20160502, 2016.

\bibitem{Liang2022}
S.~Liang, G.~Molica~Bisci, and B.~Zhang.
\newblock Sign-changing solutions for {Kirchhoff}-type problems involving variable-order fractional {Laplacian} and critical exponents.
\newblock {\em Nonlinear Anal.: Model. Control}, 27(3):556--575, 2022.

\bibitem{Luo2018}
D.~Luo and J.~Wang.
\newblock Coupling by reflection and {H\"older} regularity for non-local operators of variable order.
\newblock {\em Trans. Amer. Math. Soc.}, 371(1):431--459, 2019.

\bibitem{Mu2021}
X.~Mu, J.~Huang, L.~Wen, and S.~Zhuang.
\newblock Modeling viscoacoustic wave propagation using a new spatial variable-order fractional {Laplacian} wave equation.
\newblock {\em Geophysics}, 86(6):T487--T507, 2021.

\bibitem{Patnaik2020review}
S.~Patnaik, J.~P. Hollkamp, and F.~Semperlotti.
\newblock Applications of variable-order fractional operators: a review.
\newblock {\em Proc. R. Soc. A Math. Phys. Eng. Sci.}, 476(2234):20190498, 2020.

\bibitem{Patnaik2021-fracture}
S.~Patnaik and F.~Semperlotti.
\newblock Variable-order fracture mechanics and its application to dynamic fracture.
\newblock {\em npj Comput. Mater.}, 7(1):27, 2021.

\bibitem{Ros-Oton2014}
X.~Ros-Oton and J.~Serra.
\newblock The {Dirichlet} problem for the fractional {Laplacian}: regularity up to the boundary.
\newblock {\em J. Math. Pures Appl.}, 101(3):275--302, 2014.

\bibitem{Samko2013}
S.~G. Samko.
\newblock Fractional integration and differentiation of variable order: an overview.
\newblock {\em Nonlinear Dyn.}, 71(4):653--662, 2013.

\bibitem{Samko1993}
S.~G. Samko and B.~Ross.
\newblock Integration and differentiation to a variable fractional order.
\newblock {\em Integral Transforms Spec. Funct.}, 1(4):277--300, 1993.

\bibitem{Sun2019}
H.~Sun, A.~Chang, Y.~Zhang, and W.~Chen.
\newblock A review on variable-order fractional differential equations: mathematical foundations, physical models, numerical methods and applications.
\newblock {\em Fract. Calc. Appl. Anal.}, 22(1):27--59, 2019.

\bibitem{Tao2019}
Y.~Tao, X.~Tian, and Q.~Du.
\newblock Nonlocal models with heterogeneous localization and their application to seamless local-nonlocal coupling.
\newblock {\em Multiscale Model. Simul.}, 17(3):1052--1075, 2019.

\bibitem{Tsuchiya1992}
M.~Tsuchiya.
\newblock L{\'e}vy measure with generalized polar decomposition and the associated {SDE} with jumps.
\newblock {\em Stoch. Stoch. Rep.}, 38(2):95--117, 1992.

\bibitem{Vanthin2021}
N.~Van~Thin.
\newblock On the variable-order fractional {Laplacian} equation with variable growth on $\mathbb{R}^N$.
\newblock {\em Taiwanese J. Math.}, 25(6):1187--1223, 2021.

\bibitem{Wang2021}
L.~Wang and B.~Zhang.
\newblock Infinitely many solutions for {Kirchhoff}-type variable-order fractional {Laplacian} problems involving variable exponents.
\newblock {\em Appl. Anal.}, 100(11):2418--2435, 2021.

\bibitem{Wu2024}
Y.~Wu and Y.~Zhang.
\newblock Variable-order fractional {Laplacian} and its accurate and efficient computations with meshfree methods.
\newblock {\em J. Sci. Comput.}, 99(1):18, 2024.

\bibitem{Xiang2019}
M.~Xiang, B.~Zhang, and D.~Yang.
\newblock Multiplicity results for variable-order fractional {Laplacian} equations with variable growth.
\newblock {\em Nonlinear Anal. Theory Methods Appl.}, 178:190--204, 2019.

\bibitem{Yu2022}
B.~Yu, X.~Zheng, P.~Zhang, and L.~Zhang.
\newblock Computing solution landscape of nonlinear space-fractional problems via fast approximation algorithm.
\newblock {\em J. Comput. Phys.}, 468:111513, 2022.

\bibitem{Zhang2011}
Y.~Zhang, J.~Peterson, and M.~Gunzburger.
\newblock Maximizing critical currents in superconductors by optimization of normal inclusion properties.
\newblock {\em Physica D}, 240(21):1701--1713, 2011.

\bibitem{zhang2025}
Y.~Zhang, X.~Zhao, and S.~Zhou.
\newblock Fourier pseudospectral methods for the variable-order space fractional wave equations.
\newblock {\em Comput. Phys. Commun.}, 318:109876, 2026.

\bibitem{Zhou2024}
S.~Zhou and Y.~Zhang.
\newblock A novel and simple spectral method for nonlocal {PDEs} with the fractional {Laplacian}.
\newblock {\em Comput. Math. Appl.}, 168:133--147, 2024.

\bibitem{Zhu2014}
T.~Zhu and J.~M. Harris.
\newblock Modeling acoustic wave propagation in heterogeneous attenuating media using decoupled fractional {Laplacians}.
\newblock {\em Geophysics}, 79(3):T105--T116, 2014.

\end{thebibliography}
\end{document}